\documentclass[reqno]{amsart}
\usepackage[utf8]{inputenc}
\usepackage{amsmath}
\usepackage{amssymb}
\usepackage{amsthm}
\usepackage[all]{xy}
\xyoption{tips}
\SelectTips{cm}{10}
\usepackage[dvipsnames,table,xcdraw]{xcolor}
\usepackage{fullpage}
\usepackage{stmaryrd}
\usepackage{enumitem}
\usepackage{mathtools}
\usepackage{microtype}
\usepackage{todonotes,booktabs}
\usepackage{float}
\usepackage{aliascnt}
\usepackage{longtable}
\usepackage{tabularray}

\usepackage[maxbibnames=30,doi=false,url=false,
    giveninits=true, isbn=false]{biblatex}
\usepackage[parfill]{parskip}

\usepackage{tikz,scalerel}
\usetikzlibrary{positioning,arrows}
\usetikzlibrary{shapes.geometric, calc}
\usetikzlibrary{decorations.pathreplacing}
\usepackage{quiver}

\theoremstyle{plain}

\newaliascnt{theorem}{equation}  
\newtheorem{theorem}[theorem]{Theorem}  
\aliascntresetthe{theorem}

\newtheorem{ThmAlpha}{Theorem}

\usepackage{pdflscape}

\theoremstyle{definition}

\newaliascnt{proposition}{equation}  
\newtheorem{proposition}[proposition]{Proposition}
\aliascntresetthe{proposition}

 \newaliascnt{hproposition}{equation}  

\aliascntresetthe{hproposition}

\newaliascnt{lemma}{equation}  
\newtheorem{lemma}[lemma]{Lemma}
\aliascntresetthe{lemma}

\newaliascnt{corollary}{equation}  
\newtheorem{corollary}[corollary]{Corollary}
\aliascntresetthe{corollary}

\newaliascnt{claim}{equation}  

\aliascntresetthe{claim}

\newaliascnt{conjecture}{equation}  
\newtheorem{conjecture}[conjecture]{Conjecture}
\aliascntresetthe{conjecture}

\newaliascnt{question}{equation}  

\aliascntresetthe{question}

\newaliascnt{definition}{equation}  
\newtheorem{definition}[definition]{Definition}
\aliascntresetthe{definition}

\newaliascnt{hdefinition}{equation}  

\aliascntresetthe{hdefinition}

\newaliascnt{example}{equation}  
\newtheorem{example}[example]{Example}
\aliascntresetthe{example}

\theoremstyle{remark}

\newaliascnt{remark}{equation}  
\newtheorem{remark}[remark]{Remark}
\aliascntresetthe{remark}

\newaliascnt{convention}{equation}  
\newtheorem{convention}[convention]{Convention}
\aliascntresetthe{convention}

\usepackage{hyperref}
\usepackage[nameinlink,capitalise,noabbrev]{cleveref}

\definecolor{dark-red}{rgb}{0.5,0.15,0.15}
\definecolor{dark-blue}{rgb}{0.15,0.15,0.6}
\definecolor{dark-green}{rgb}{0.15,0.6,0.15}
\hypersetup{
    colorlinks, linkcolor=dark-red,
    citecolor=dark-blue, urlcolor=dark-green
}

\newcommand{\scomment}[1]{\todo[inline,color=yellow!40]{\textbf{Scott: }#1}}
\newcommand{\jdcomment}[1]{\todo[inline,color=cyan!40]{\textbf{J.D.: }#1}}

\newcommand{\sfT}{\mathsf{T}}

\newcommand{\uA}{\underline{A}} 
\newcommand{\burn}{\uA} 

\newcommand{\cO}{\mathcal{O}}
\newcommand{\sfP}{\mathsf{P}}

\newcommand{\sfB}{\mathsf{B}}

\newcommand{\Z}{\mathbb{Z}}

\usepackage{mathrsfs}

\newcommand{\m}{\mathfrak{m}}
\newcommand{\p}{\mathfrak{p}}

\usepackage{bbm}
\usepackage{multicol}

\usepackage{halloweenmath}

\DeclareMathOperator{\Spec}{\mathbf{Spec}}

\DeclareMathOperator{\Sub}{\mathbf{Sub}}
\DeclareMathOperator{\Spc}{\mathbf{Spc}}

\DeclareMathOperator{\tr}{\operatorname{tr}}
\DeclareMathOperator{\res}{\operatorname{res}}
\DeclareMathOperator{\nm}{\operatorname{nm}}
\DeclareMathOperator{\conj}{c}

\definecolor{nice-green}{HTML}{8DD883}
\definecolor{nice-red}{HTML}{D1634D}
\definecolor{nice-blue}{HTML}{63B4D1}
\usetikzlibrary{decorations.pathmorphing}
\tikzset{snake it/.style={decorate, decoration=snake}}

\usepackage{subcaption}

\title{Spectra of bi-incomplete Tambara functors}
\author{Scott Balchin}\address{Mathematical Sciences Research Centre, Queen's University Belfast, UK}\email{s.balchin@qub.ac.uk}
\author{J.D. Quigley}\address{Department of Mathematics, University of Virginia, Charlottesville, VA, USA}\email{mbp6pj@virginia.edu}
\author{Ben Spitz}\address{Department of Mathematics, Indiana University, Bloomington, IN, USA}\email{bespitz@iu.edu}

\date{\today}

\begin{document}

\begin{abstract}
Bi-incomplete Tambara functors are equivariant generalizations of commutative rings. The most common forms of bi-incomplete Tambara functors are coefficient systems of commutative rings, Green functors, and  Tambara functors. In the 1980s, Lewis introduced prime ideals in Green functors, and in the 2010s, Nakaoka introduced prime ideals in Tambara functors. In this work, we define the spectrum of prime ideals for an arbitrary bi-incomplete Tambara functor, simultaneously generalizing Lewis and Nakaoka's notions. We then  produce many computational tools which we apply to several examples of interest. 
\end{abstract}


\maketitle

\tableofcontents


\section{Introduction}

\subsection{Context}

Let $G$ be a finite group. A \emph{$G$-Tambara functor} $T$ is a collection of commutative rings $T(G/H)$, one for each subgroup $H \leq G$, together with structure maps
\begin{align*}
\res^H_K \colon T(G/H) &\to T(G/K),\\
\tr_K^H \colon T(G/K) &\to T(G/H),\\
\nm_K^H \colon T(G/K) &\to T(G/H),\\
\conj_{H,g} \colon T(G/H) &\to T(G/gHg^{-1})
\end{align*}
for each inclusion $K \leq H \leq G$ and for each $g \in G$. These structure maps, called \emph{restriction}, \emph{transfer}, \emph{norm}, and \emph{conjugation}, respectively, are required to satisfy various axioms (see \cref{subsec:tambara} for details). 

Tambara functors arise naturally in equivariant homotopy theory, where they encode the algebraic structures in ``equivariantly commutative" generalized homology theories. They play the role of commutative rings in the genuine $G$-equivariant world, and as is common, one seeks to mirror constructions from algebra in this setting to inform homotopical constructions. Here we are interested in the analogue of the Zariski spectrum. In \cite{nakaoka_spectrum}, Nakaoka defined prime ideals of Tambara functors. These prime ideals with the usual hull-kernel topology form a space, which we call the \emph{Nakaoka spectrum}. The Nakaoka spectra of various Tamara functors have been studied by several authors and are now somewhat well-understood~\cite{CalleGinnett,calle2025spectrumburnsidetambarafunctor,4DS,nakaoka_spectrum}.

Tambara functors represent one extreme among a collection of equivariant gadgets. The algebraic structure present in less structured equivariant homology theories is encoded by \emph{bi-incomplete Tambara functors}: these are essentially Tambara functors, but with only a subset of the transfer and norm maps. The most familiar examples of bi-incomplete Tambara functors are coefficient systems (no nontrivial transfers or norms) and Green functors (all transfers but no nontrivial norms). Such structures are controlled by the selection of a \emph{compatible pair of transfer systems} $(\cO_m,\cO_a)$  which dictates the allowed norms and transfers, respectively \cite{BH_biincomplete, chan}. Even in simple cases, the allowed behavior is surprising. For example, for the cyclic group $C_{p^n}$ there are $\frac{1}{2n+3}\binom{3n+3}{n+1}$ distinct compatible pairs \cite{MR4895458}.

The purpose of this paper is to define ideals in bi-incomplete Tambara functors and study the resulting spectra of prime ideals with a focus on computational methods. This general framework allows us to systematically compare Nakaoka spectra as we vary which transfers and norms are taken into account. This allows us, for example, to unify the aforementioned results regarding Nakaoka spectra with results of Lewis regarding prime ideals of Green functors \cite{lewis:1980}. 

As we have already discussed, Tambara functors form the $G$-equivariant analogue of commutative rings. A generalization of commutative rings in the categorical direction is played by the theory of \textit{tensor-triangular categories}~\cite{balmer_spectrum}. These categories, which have a compatible suspension and monoidal structure, also possess an analogue of the Zariski spectrum, given by the \textit{Balmer spectrum}. In several recent papers, the existence of a merging these two ideas, that is, the theory of \textit{equivariant tensor-triangular geometry}, has been conjectured (c.f., the introduction of \cite{calle2025spectrumburnsidetambarafunctor} or \cite{BalmerBarthelGreenleesPevtsova2023}). While we do not undertake the general study of such categories in this paper, we perform the first step in this direction by explicitly computing a spectrum object for the category of $C_2$-equivariant finite spectra, comparing it to our calculations of the incomplete Burnside Tambara functors for $C_2$, and identifying potential obstructions to  the existence of a comparison map.

\subsection{Main computational results} While the definitions of ideals and prime ideals for Tambara functors are  easily adapted to bi-incomplete Tambara functors, we find that the methods required to compute the resulting spectra of prime ideals depend heavily on the compatible pair as well as the structure of the bi-incomplete Tambara functor itself. Most of our main results are new computational tools to approach this problem.

Our main computational tool, roughly speaking, allows us to reduce from considering general bi-incomplete Tambara functors to only considering specific kinds of bi-incomplete Tambara functors, namely, bi-incomplete Tambara functors which possess the same collection of transfers and norms. We call a transfer systems $\cO$ such that $(\cO, \cO)$ is a compatible pair \emph{self-compatible}. The computation of the spectra of these self-compatible bi-incomplete Tambara functors is the statement of the following theorem: 


\begin{ThmAlpha}[{\cref{thm:pi0}}]\label{thm:main1}
Let $(\cO,\cO)$ be a self-compatible pair, and $T$ be an $(\cO,\cO)$-Tambara functor. Then there is a bijection 
\[
    \Spec^{(\cO,\cO)}(T) \cong \coprod_{\mathcal{O}' \in \pi_0 \mathcal{O}} \Spec(i^*_{\mathcal{O}'} T)
\]
between the $(\mathcal{O},\mathcal{O})$-Nakaoka spectrum of $T$, and the disjoint union over $\pi_0\mathcal{O}$, the set of connected components of $\mathcal{O}$ regarded as a subgraph of the subgroup lattice of $G$, of the Nakaoka spectra of the restrictions $i^*_{\mathcal{O}'}T$ of $T$ to each path component $\mathcal{O}'$. 
\end{ThmAlpha}

There is a nuance in Theorem~\ref{thm:main1} in that for essentially all choices of $\mathcal{O}'$, $i^*_{\mathcal{O}'}T$ need not be a bi-incomplete Tambara functor. Instead it is a more general object called a \emph{restricted Tambara functor}; roughly speaking, this is a Tambara functor with extra conjugation maps at each level. The theory of restricted Tambara functors follows from work of the third author in \cite{spitz_norms_2024}. Constructions such as that of the prime ideals and spectrum go through as expected in this more general framework, and thus allows for Theorem \ref{thm:main1} to be used in practice. We recall the relevant definitions and theory in \cref{subsec:res}.

There are many situations, such as in the Green functor case, where the bi-incomplete Tambara functor has different collections of transfers and norms, and as such we cannot directly apply Theorem \ref{thm:main1} as we are not working with a self-compatible pair. If we simply need to add transfers to a self-compatible pair, then we can appeal to the following result which gives us a checkable condition on determining the prime ideals. We say that $(\mathcal{O}_m, \mathcal{O}_a)$ is a sub-compatible pair of $(\mathcal{O}_m', \mathcal{O}_a')$ if we have inclusions $\cO_m \subseteq \cO_m'$ and $\cO_a \subseteq \cO_a'$.

\begin{ThmAlpha}[{\cref{Prop:ChangeOaPrimes}}]\label{thm:main2}
    Let $(\mathcal{O}_m, \mathcal{O}_a)$ be a sub-compatible pair of $(\mathcal{O}_m, \mathcal{O}_a')$ and let $T$ be an $(\mathcal{O}_m,\mathcal{O}_a')$-Tambara functor. Each $(\mathcal{O}_m,\mathcal{O}_a')$-prime ideal of $T$ is also an $(\mathcal{O}_m,\mathcal{O}_a)$-prime ideal of $T$. Conversely, every $(\mathcal{O}_m,\mathcal{O}_a)$-prime ideal of $T$ which is closed under the transfers in $\mathcal{O}_a'$ is an $(\mathcal{O}_m,\mathcal{O}_a')$-prime ideal. 
\end{ThmAlpha}

While this strategy of looking at self-compatible bi-incomplete Tambara functors and adding transfers gives us many of the cases, it does not cover all of them. The issue is that there are some transfer systems $\cO$ which are not self-compatible, even for small groups such as $C_{p^2}$. The previous tools do not apply for such bi-incomplete Tambara functors, and so in general, an \emph{ad hoc} computation must be done. However, if $T$ is \emph{multiplicatively cohomological} (\cref{def:coh}) we can appeal to the following result:

\begin{ThmAlpha}[{\cref{thm:insensitive}}]\label{thm:main3}
        Let $T$ be a multiplicatively cohomological Tambara functor. Then for any compatible pair $(\cO_m,\cO_a)$ we have a homeomorphism
    \[
    \Spec^{(\cO_m,\cO_a)}(T) \cong \Spec^{(\mathrm{Hull}(\cO_m),\cO_a)}(T)
    \]
    where $\mathrm{Hull}(\cO_m)$ is a transfer system depending on $\cO_m$ which is self-compatible.
\end{ThmAlpha}

We give a graphical representation of the general strategy of computing the Nakaoka spectra of bi-incomplete Tambara functors for $G=C_{p^2}$ using these computational tools in \cref{fig:strat}.

\begin{figure}[h]
\centering
\begin{tikzpicture}[scale=0.95]
    \draw[gray,ultra thick,->] (1,1) -- (11,1);
    \draw[gray,ultra thick,->] (1,1) -- (1,11);
    \draw[gray,thin,dashed] (1,1) -- (11,11);
    \node at (11.5,1) {$\mathcal{O}_{{a}}$};
    \node at (1,11.5) {$\mathcal{O}_{{m}}$};
    \draw[gray,ultra thick] (2,1.2) -- (2,0.8);
    \draw[gray,ultra thick] (4,1.2) -- (4,0.8);
    \draw[gray,ultra thick] (6,1.2) -- (6,0.8);
    \draw[gray,ultra thick] (8,1.2) -- (8,0.8);
    \draw[gray,ultra thick] (10,1.2) -- (10,0.8);
    \node at (2,0.5) {$\mathcal{O}_{\mathrm{triv}}$};
    \node at (4,0.5) {$\mathcal{O}_{\mathrm{1}}$};
    \node at (6,0.5) {$\mathcal{O}_{\mathrm{2}}$};
    \node at (8,0.5) {$\mathcal{O}_{\mathrm{3}}$};
    \node at (10,0.5) {$\mathcal{O}_{\mathrm{comp}}$};
    \begin{scope}[rotate around={90:(1,1)}]
    \draw[gray,ultra thick] (2,1.2) -- (2,0.8);
    \draw[gray,ultra thick] (4,1.2) -- (4,0.8);
    \draw[gray,ultra thick] (6,1.2) -- (6,0.8);
    \draw[gray,ultra thick] (8,1.2) -- (8,0.8);
    \draw[gray,ultra thick] (10,1.2) -- (10,0.8);
    \end{scope}
    \node at (0.25,2) {$\mathcal{O}_{\mathrm{triv}}$};
    \node at (0.25,4) {$\mathcal{O}_{\mathrm{1}}$};
    \node at (0.25,6) {$\mathcal{O}_{\mathrm{2}}$};
    \node at (0.25,8) {$\mathcal{O}_{\mathrm{3}}$};
    \node at (0.15,10) {$\mathcal{O}_{\mathrm{comp}}$};
    \draw[nice-red, ultra thick, snake it] (10,8.276) -- (10,9.8);
    \draw[nice-green, ultra thick,fill=white] (2,2) +(-6pt,-6pt) rectangle +(6pt,6pt);
    \draw[nice-green, ultra thick,fill=white] (4,4) +(-6pt,-6pt) rectangle +(6pt,6pt);
    \draw[nice-green, ultra thick,fill=white] (6,6) +(-6pt,-6pt) rectangle +(6pt,6pt);
    \draw[nice-green, ultra thick,fill=white] (10,10) +(-6pt,-6pt) rectangle +(6pt,6pt);
    \filldraw[black] (2,2) circle (2pt);
    \filldraw[black] (4,4) circle (2pt);
    \filldraw[black] (6,6) circle (2pt);
    \filldraw[black] (10,10) circle (2pt);
    \filldraw[black] (4,2) circle (2pt);
    \filldraw[black] (6,2) circle (2pt);
    \filldraw[black] (8,2) circle (2pt);
    \filldraw[black] (10,2) circle (2pt);
    \filldraw[black] (8,4) circle (2pt);
    \filldraw[black] (10,4) circle (2pt);
    \filldraw[black] (10,6) circle (2pt);
    \filldraw[black] (10,8) circle (2pt);
    \draw[nice-green, ultra thick, ->] (2.2,2) -- (3,2);
    \draw[nice-green, ultra thick, ->] (4.2,4) -- (5,4);
    \draw[nice-green, ultra thick, ->] (6.2,6) -- (7,6);
    \draw[nice-blue, ultra thick] (4,2) circle (6pt);
    \draw[nice-blue, ultra thick] (6,2) circle (6pt);
    \draw[nice-blue, ultra thick] (8,2) circle (6pt);
    \draw[nice-blue, ultra thick] (10,2) circle (6pt);
    \draw[nice-blue, ultra thick] (8,4) circle (6pt);
    \draw[nice-blue, ultra thick] (10,4) circle (6pt);
    \draw[nice-blue, ultra thick] (10,6) circle (6pt);
    \draw[nice-red, ultra thick,rotate around={45:(10,8)}] (10,8) +(-6pt,-6pt) rectangle +(6pt,6pt);
\end{tikzpicture}
\caption{A schematic for computing the Nakaoka spectra for a bi-incomplete Tambara functor in the case $G=C_{p^2}$. There are 5 possible transfer systems for $G$, leading to 12 compatible pairs which are represented by the dots. The self-compatible pairs are those in the green squares along the diagonal, and the spectra for these can be computed via Theorem~\ref{thm:main1}. Moving to the right in the diagram is adding  transfers, and we can deduce which prime ideals cease to be prime using Theorem~\ref{thm:main2}, allowing us to compute the spectra for those pairs in blue circles. The only remaining pair is the one in the red diamond. If we are multiplicatively cohomological then the prime spectrum at diamond is insensitive to the saturated hull, and thus is homeomorphic with the spectrum above it, indicated by the red wavy line, by Theorem~\ref{thm:main3}. If we are not multiplicatively cohomological, however, then this is a computation that needs to be done separately.}\label{fig:strat}
\end{figure}
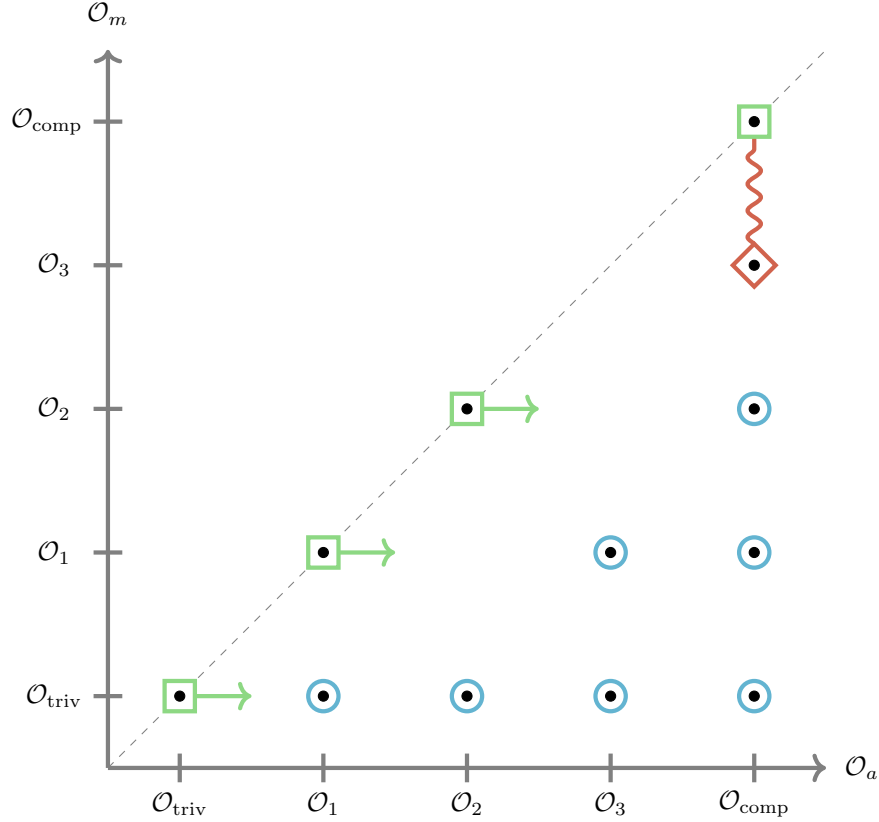

\subsection{Outline of the document}

This paper is separated into two parts, \cref{part:theory} and \cref{part:examples}. The former is dedicated to theory, while the latter provides several comprehensive examples, showcasing the power of the computational tools.

We begin in \cref{sec:prelim} by outlining the preliminary material needed to appreciate the constructions. \cref{subsec:tambara} contains the theory of bi-incomplete Tambara functors, \cref{subsec:res} the theory of restricted Tambara functors, and \cref{subsec:ideals} the structure of (prime) ideals in the aforementioned variants Tambara functors.

\cref{sec:properties} is concerned with properties of bi-incomplete Tambara functors and their ideals. The levelwise structure of bi-incomplete prime ideals is discussed in \cref{subsec:weyl} and \cref{subsec:radical}. The process of adding norms and transfers to a bi-incomplete Tambara functor and its effect on primes (c.f. Theorem \ref{thm:main2}) is the content of \cref{subsec:chaingeofpairs}. Finally, \cref{subsec:hulls} contains the theory and proof of Theorem \ref{thm:main3} regarding saturated hulls and multiplicatively cohomological Tambara functors. Equipped with the required theory, we are then in a position to prove Theorem \ref{thm:main1}, which occupies \cref{sec:pathdecomp}.

In \cref{part:examples} we pivot our attention to computing explicit examples. In \cref{sec:exfp} we compute the Nakaoka spectrum of the fixed point bi-incomplete Tambara functor $\underline{\Z}$ over $C_p$ and $C_{p^2}$ for all compatible pairs. The reader is invited to look ahead at the answer summarized in \cref{tab:cpzunderline} and \cref{fig:zunderline}. It is worth noting that the spectra for certain compatible pairs over $C_{p^2}$ are equal, so the Nakaoka spectrum of fixed point bi-incomplete Tambara functors is not generally sensitive to the compatible pair. 

In \cref{sec:exburn}, we turn to more intricate computations of the Nakaoka spectra of Burnside bi-incomplete Tambara functors, which are \textit{not} multiplicatively cohomological. See \cref{tab:cpBurnside} and \cref{tab:ofspectra} for summaries of the computations over $C_p$ and $C_{p^2}$. In contrast with the analogous computations for $\underline{\mathbb{Z}}$, we obtain distinct homeomorphism types for each possible compatible pair. Interestingly, we apply the ghost construction of \cite{4DS} to study the spectra of bi-incomplete Tambara functors over $C_{p^2}$, even though that construction was defined only over $C_p$. We also compute the spectrum of an interesting bi-incomplete Tambara functor over $C_{pq}$ for which more \emph{ad hoc} arguments are required. 

Finally, in \cref{sec:vista} we will describe how this style of computation allows us to start thinking about ``$G$-equivariant tensor-triangular geometry'' with a computation of a Balmer--Nakaoka spectrum of the category $\mathsf{Sp}_{C_2}^\omega$ of genuine $C_2$-equivariant finite spectra, wherein we see some gremlins regarding topology creep in when trying to compute comparison maps.


\subsection{Acknowledgments}

The authors would like to thank the Isaac Newton Institute for Mathematical Sciences, Cambridge, for support and hospitality during the programme ``Equivariant homotopy theory in context" where work on this paper was undertaken. This work was supported by EPSRC grant no EP/Z000580/1. JDQ was partially supported by NSF grants DMS-2414922 and DMS-2441241. 

\newpage

\part{Spectra of bi-incomplete Tambara functors}\label{part:theory}

\section{Preliminaries}\label{sec:prelim}

\subsection{Bi-incomplete Tambara functors}\label{subsec:tambara}

    $G$ denotes a finite group throughout. We begin by introducing the category of bi-incomplete Tambara functors via the language of transfer systems \cite{bbr,MR4244201}.

\begin{definition}\label{defn:transfersystem}
    A \emph{$G$-transfer system} is a subset $\cO \subseteq \Sub(G) \times \Sub(G)$ such that
    \begin{itemize}
        \item $(K,H) \in \mathcal{O}$ implies $K \leq H$.
        \item $(K,H) \in \cO$ and $(H,L) \in \cO$ implies $(K,L) \in \cO$. 
        \item $(K,H) \in \mathcal{O}$ implies $(gKg^{-1},gHg^{-1}) \in \mathcal{O}$ for all $g \in G$.
        \item $(K,H) \in \cO$ and $L \leq H$ implies $(K \cap L,L) \in \cO$.
    \end{itemize}
    A $G$-transfer system $\cO$ is \emph{saturated} if whenever two  of the pairs $(A,B)$, $(B,C)$, and $(A,C)$ are in $\cO$, then so is the third.
\end{definition}

\begin{example}\label{ex:extremes}
    We highlight two saturated transfer systems which play a distinguished role for all $G$:
    \begin{itemize}
        \item $\cO_\mathrm{triv}$ is the \emph{trivial transfer system} consisting of only the pairs $(K,K)$;
        \item $\cO_\mathrm{comp}$ is the \emph{complete transfer system} consisting of all possible pairs $(K,H)$.
    \end{itemize}
\end{example}

\begin{remark}
Associated to any transfer system $\mathcal{O}$ is a smallest saturated transfer system $\mathrm{Hull}(\cO)$ such that $\cO \subseteq \mathrm{Hull}(\cO)$. This can be explicitly constructed by either intersecting all saturated transfer systems which lie above $\cO$, or equivalently by adding in the pairs $(K,H)$ required to satisfy the two-out-of-three property.
\end{remark}

\begin{example}\label{ex:o3example}
    Let $p$ be prime and $G = C_{p^2}$. Then there is a transfer system for $G$, which we denote $\cO_3$, where
    \[
    \cO_3 \coloneq \{(e,C_p), (e, C_{p^2})\}.
    \]
    Note that $\cO_3$ is not saturated as it does not satisfy two-out-of-three. The saturated hull, $\mathrm{Hull}(\cO_3)$, is obtained by adding in $(C_p,C_{p^2})$. That is, $\mathrm{Hull}(\cO_3) = \cO_\mathrm{comp}$.
\end{example}

For us, the general idea is that a transfer system $\cO$ is an allowed collection of norms or transfers for a Tambara-like structure. As such, to define our bi-incomplete Tambara functors we need to pick suitable norms $\cO_m$ and suitable transfers $\cO_a$. The following definition outlines the compatibilities that these two choices much have.

\begin{definition}[{\cite[Definition 4.6]{chan}}]
    Let $\cO_m$ and $\cO_a$ be two $G$-transfer systems. We say that $(\cO_m,\cO_a)$ is a \emph{compatible pair} if whenever $A$ is a subgroup of $G$ and $B,C \leq A$ are subgroups such that $(B,A) \in \cO_m$ and $(B \cap C,B) \in \cO_a$ then $(C, A) \in \cO_a$. We will sometimes refer to $\cO_m$ as the \emph{multiplicative} part and $\cO_a$ as the \emph{additive} part of the compatible pair $(\cO_m,\cO_a)$. 
\end{definition}

\begin{remark}
    It follows from the definition that in a compatible pair of transfer systems $(\cO_m,\cO_a)$ we necessarily have that $\cO_m \subseteq \cO_a$. Indeed, if $(B,A) \in \cO_m$, then $(B \cap B,B) \in \cO_m$, and thus by the compatibility axiom we have $(B,A) \in \cO_a$.
\end{remark}

\begin{example}\label{ex:verycompatible}
    The two extremes of transfer systems from \cref{ex:extremes} are highly compatible:
    \begin{enumerate}
        \item $(\cO_{\mathrm{triv}},\cO_{a})$ is a compatible pair for any $\cO_a$.
        \item $(\cO_m,\cO_{\mathrm{comp}})$ is a compatible pair for any $\cO_m$.
    \end{enumerate}
\end{example}

\begin{lemma}[{\cite[Proposition 3.5]{MAZUR2025109443}}]\label{lem:hull}
    Let $(\mathcal{O}_m, \mathcal{O}_a)$ be a compatible pair. Then $\operatorname{Hull}(\mathcal{O}_m) \subseteq \mathcal{O}_a$. In particular a transfer system $\mathcal{O}$ forms a \emph{self-compatible} pair $(\cO,\cO)$ if and only if $\cO$ is saturated.
\end{lemma}

Now that we have introduced the concepts of compatible pairs of transfer systems, we can begin constructing the category of bi-incomplete Tambara functors.

\begin{definition}\label{def:polynomial category}
    The \textit{category of polynomials valued in finite $G$-sets}, denoted $\mathbf{P}_G$, has objects finite $G$-sets, and morphisms $\mathbf{P}_G(X,Y)$ isomorphism classes of bispans
    \[
        [X \leftarrow A \to B \to Y]
    \]
    where $A$ and $B$ are finite $G$-sets and all maps are $G$-equivariant.
\end{definition}

\begin{definition}
    Let $(\mathcal{O}_m, \mathcal{O}_a)$ be a compatible pair. Then $\mathbf{P}^{(\mathcal{O}_m, \mathcal{O}_a)}_G$ is the wide subcategory of $\mathbf{P}_G$ whose morphisms are given by isomorphism classes of bispans
    \[
        [X \leftarrow A \xrightarrow{n} B \xrightarrow{t} Y]
    \]
    whenever $(A,B) \in \mathcal{O}_m$ and $(B,Y) \in \mathcal{O}_a$.
\end{definition}

\begin{definition}[\cite{BH_biincomplete}]\label{def:biincomplete}
    Let $(\mathcal{O}_m, \mathcal{O}_a)$ be a compatible pair. An \textit{$(\mathcal{O}_m, \mathcal{O}_a)$-Tambara functor} (or a \emph{bi-incomplete Tambara functor} when we do not wish to record the explicit compatible pair) is a product preserving functor $T \colon \mathbf{P}^{(\mathcal{O}_m, \mathcal{O}_a)}_G \to \mathsf{Set}$ such that the commutative semi-ring $T(X)$ is a commutative ring for all $X$. We will denote the category of such objects as $\mathsf{Tamb}_G^{(\mathcal{O}_m, \mathcal{O}_a)}$.
\end{definition}

\begin{remark}\label{not:biincomplete}
    We unwrap \cref{def:biincomplete} to give a more tangible description of bi-incomplete Tambara functors in the usual style one expects to see. An \textit{$(\mathcal{O}_m, \mathcal{O}_a)$-Tambara functor} is a collection of commutative rings $T(G/H)$, one for each subgroup $H \leq G$, together with natural structure maps
    \begin{align*}
    \res^H_K&: T(G/H) \to T(G/K) \text{ for all } K \leq H \leq G,\\
    \tr_K^H&: T(G/K) \to T(G/H) \text { for all } (K,H) \in \cO_a,\\
    \nm_K^H&: T(G/K) \to T(G/H) \text { for all } (K,H) \in \cO_m,\\
    \conj_{H,g} &: T(G/H) \to T(G/gHg^{-1}) \text { for all } g \in G,
    \end{align*}
    called the \emph{restriction}, \emph{transfer}, \emph{norm}, and \emph{conjugation} maps, respectively. The conjugation maps are ring isomorphisms, the restriction maps are ring homomorphisms, the transfer maps are additive, and the norms are multiplicative. These maps are subject to several relations (see \cite{strickland2012tambarafunctors}):
    
    \begin{proposition}[Frobenius reciprocity] \label{prop: Frobenius reciprocity}
	Let $T$ be a $(\cO_m,\cO_a)$-Tambara functor, let $(K,H)\in \cO_a$, and let $x\in T(G/K)$ and $y\in T(G/H)$.  Then
	\[
		\tr_K^H(x)\cdot y = \tr_K^H(x\cdot \res_K^H(y)).
	\]
\end{proposition}

\begin{proposition}[Tambara reciprocity]\label{proposition: Tambara reciprocity}
	Let $T$ be a $(\cO_m,\cO_a)$-Tambara functor, let $(K,H) \in \cO_m$ and let $x,y\in T(G/K)$. Then
	\[
		\nm_K^H(x+y) = \nm_K^H(x)+\nm_K^H(y)+ \tau
	\] where $\tau$ is a sum of terms, each of which is in the image of a transfer map $\tr^H_{L}$ for some proper subgroup $L<H$.
\end{proposition}

\begin{proposition}[Double coset formula]\label{prop:dcoset}
	For any $L,K\leq H \leq G$ we let $g$ run over a set of representatives of the double cosets $L\backslash H \slash K$.  For any $(\cO_m,\cO_a)$-Tambara functor $T$ and $x\in T(G/K)$ we have
	\[
		\res^H_L\circ \tr^H_K(x) = \sum\limits_{LgK \in  L\backslash H \slash K} \tr^L_{L\cap g K g^{-1}}\res^{g K g^{-1}}_{L\cap g K g^{-1}} \conj_{g}(x) \quad \text{ if } (K,H)\in \cO_a,
	\]
	\[
		\res^H_L\circ \nm^H_K(x) = \prod\limits_{L g K\in  L\backslash H \slash K} \nm^L_{L\cap g K g^{-1}}\res^{g K g^{-1}}_{L\cap g K g^{-1}} \conj_{g}(x) \quad \text{ if } (K,H)\in \cO_m.
	\]
\end{proposition}

\begin{remark}
    We point out that \cref{proposition: Tambara reciprocity} and \cref{prop:dcoset} make sense in that the transfers $\tr^K_L$ (resp., norms $\nm^K_L$) appearing in the formulas all have the property that $(L,K) \in \cO_a$ (resp., $(L,K) \in \cO_m$). This follows (though not obviously) from the definition of transfer systems and compatible pairs \cite{chan}.
\end{remark}
\end{remark}

\begin{example}\label{ex:specifics}
    Certain bi-incomplete Tambara functors recover more familiar notions:
    \begin{enumerate}
        \item An $(\mathcal{O}_\mathrm{triv}, \mathcal{O}_\mathrm{triv})$-Tambara functor is a \emph{coefficient system}.
        \item An $(\mathcal{O}_\mathrm{triv}, \mathcal{O}_\mathrm{comp})$-Tambara functor is a \emph{Green functor}. 
        \item An $(\mathcal{O}_\mathrm{comp}, \mathcal{O}_\mathrm{comp})$-Tambara functor is a (complete) \emph{Tambara functor}.
    \end{enumerate}
\end{example}

We finish this section with some specific examples of bi-incomplete Tambara functors which will be used throughout. We will make repeated use of the following proposition.

\begin{proposition}[{\cite[Proposition 7.50]{BH_biincomplete}}]\label{prop:forget}
    Suppose $(\cO_m,\cO_a) \subseteq (\cO_m',\cO_a')$ is an inclusion of compatible pairs. Then precomposition with the inclusion of polynomial categories yields a forgetful functor
\[
i^\ast \colon \mathsf{Tamb}_G^{(\mathcal{O}_m', \mathcal{O}_a')} \to \mathsf{Tamb}_G^{(\mathcal{O}_m, \mathcal{O}_a)}.
\]
    \end{proposition}

\begin{example}[Burnside bi-incomplete Tambara Functor]\label{ex:burnside}
    We recall that the \emph{Burnside ring} of a finite group $G$, denoted $A(G)$, is the Grothendieck group completion of the semi-ring of isomorphism classes of finite $G$-sets under disjoint union and cartesian product.     The \emph{Burnside Tambara functor} is the $(\cO_\mathrm{comp},\cO_\mathrm{comp})$-Tambara functor $\underline{A}_G$ defined as $\underline{A}_G(G/H) = A(H)$ for all $H \leq G$. The structure maps are defined by linearly extending the following maps:
    \begin{align*}
    \res_{K}^{H}\colon \burn_G(G/H)&\longrightarrow \burn_G(G/K)\\
    Y &\longmapsto Y\text{ with restricted $K$-action,}\\
    \tr_{K}^{H}\colon \burn_G(G/K)&\longrightarrow \burn_G(G/H)\\
    X &\longmapsto H\times_K X,\\
    \nm_{K}^{H}\colon \burn_G(G/K)&\longrightarrow \burn_G(G/H)\\
    X &\longmapsto {\rm Map}_K(H,X),\\
    \conj_{g,H}\colon \burn_G(G/H)&\longrightarrow \burn_G(G/\,gHg^{-1})\\
    X &\longmapsto gXg^{-1}.
\end{align*}
This Tambara functor plays the role of the integers in $\mathsf{Tamb}_G^{(\mathcal{O}_\mathrm{comp}, \mathcal{O}_\mathrm{comp})}$ as the initial object in this category. 

For each compatible pair ${(\mathcal{O}_m, \mathcal{O}_a)}$ we obtain a bi-incomplete Tambara functor $i^\ast \underline{A}_G$ via \cref{prop:forget}. Explicitly, this is obtained from $\underline{A}_G$ by forgetting the norms and transfers which do not inhabit the chosen compatible pair. We will usually abuse notation and simply write $\underline{A}_G$ for $i^\ast \underline{A}_G$ when we do not need to distinguish the two.
\end{example}

\begin{example}[Fixed point bi-incomplete Tambara Functor]\label{ex:fp}
    Let $R$ be a commutative ring with $G$-action. The \emph{fixed point Tambara functor} is the $(\cO_\mathrm{comp},\cO_\mathrm{comp})$-Tambara functor $\mathrm{FP}(R)$ defined as $\mathrm{FP}(R)(G/H) = R^H$ for all $H \leq G$. The restriction is given by the inclusions, conjugation is given by Weyl group actions, and
    \begin{align*}
            \tr_{K}^{H}\colon \mathrm{FP}(R)(G/K)&\longrightarrow \mathrm{FP}(R)(G/H)\\
    x &\longmapsto \sum_{g \in W_H(K)} g \cdot x,\\
    \nm_{K}^{H}\colon \mathrm{FP}(R)(G/K)&\longrightarrow \mathrm{FP}(R)(G/H)\\
    x &\longmapsto \prod_{g \in W_H(K)} g\cdot x.
    \end{align*}
Again, by an application of \cref{prop:forget} we obtain bi-incomplete versions of the fixed-point Tambara functor, $i^\ast\mathrm{FP}(R)$ for each compatible pair.

Of particular importance is the fixed point Tambara functor of $\Z$ equipped with the trivial action, which we will denote $\underline{\Z}$. We will study this Tambara functor in depth in \cref{sec:exfp}.
\end{example}

\begin{example}[Initial bi-incomplete Tambara Functor]\label{ex:initial}
    In \cref{ex:burnside} and \cref{ex:fp} we began with a complete Tambara functor and forgot structure to obtain bi-incomplete versions. We shall now construct an intrinsic bi-incomplete Tambara functor. We note  that while $\burn_G$ is the initial object in complete Tambara functors, $i^\ast \burn_G$ is not initial in $\mathsf{Tamb}_G^{(\mathcal{O}_m, \mathcal{O}_a)}$ in general. 

    Fix a compatible pair ${(\mathcal{O}_m, \mathcal{O}_a)}$ and define $\burn_G^{\mathcal{O}_a}(G/H)$ to be the subring of $A(H)$ spanned by those $H$-sets $H/K$ such that $(K,H) \in \cO_a$. Equipping these rings with the restriction, transfers, and norm maps inherited from $\burn_G$ we obtain a $(\mathcal{O}_m, \mathcal{O}_a)$-Tambara functor that we denote $\underline{A}_G^{(\mathcal{O}_m, \mathcal{O}_a)}$. This is the initial object of $\mathsf{Tamb}_G^{(\mathcal{O}_m, \mathcal{O}_a)}$.
\end{example}

\begin{example}
    To make the distinction between the bi-incomplete Tambara functors of \cref{ex:burnside} and \cref{ex:initial} clear, we provide an explicit example. Fix $G = C_{pq}$ for $p$ and $q$ distinct primes. We let $\cO_a = \cO_m = \{(e,C_p), (C_q,C_{pq})\}$. Then $i^\ast \burn_{C_{pq}}$ is given by the Lewis diagram on the left, while $\burn_{C_{pq}}^{(\cO_m, \cO_a)}$ appears on the right: 
    
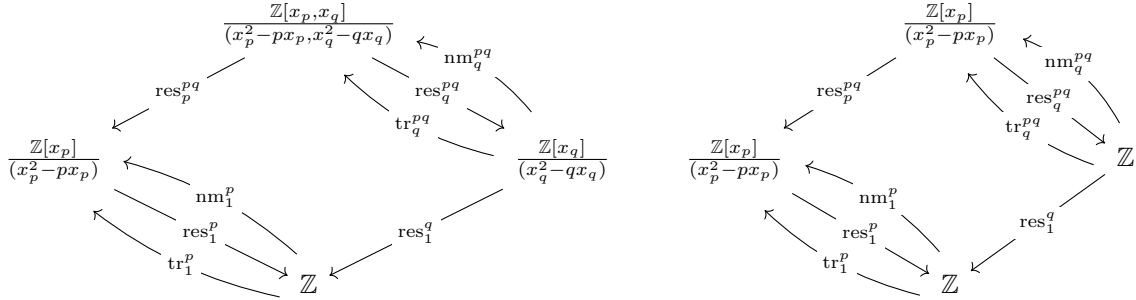
\[\begin{tikzcd}[sep=large]
	& {\frac{\Z[x_p,x_q]}{(x^2_p-px_p,x^2_q-qx_q)}} \\
	{\frac{\Z[x_p]}{(x^2_p-px_p)}} && {\frac{\Z[x_q]}{(x^2_q-qx_q)}} \\
	& \Z
	\arrow["{\res_{p}^{{pq}}}"{description}, from=1-2, to=2-1]
	\arrow["{\res_{q}^{{pq}}}"{description}, from=1-2, to=2-3]
	\arrow["{\res_1^{p}}"{description}, from=2-1, to=3-2]
	\arrow["{\tr_{q}^{{pq}}}"{description}, shift left=3, curve={height=-6pt}, from=2-3, to=1-2]
	\arrow["{\nm_{q}^{{pq}}}"{description, pos=0.6}, shift right=3, curve={height=6pt}, from=2-3, to=1-2]
	\arrow["{\res_{1}^{{q}}}"{description}, from=2-3, to=3-2]
	\arrow["{\tr_1^{p}}"{description}, shift left=3, curve={height=-6pt}, from=3-2, to=2-1]
	\arrow["{\nm_1^{p}}"{description}, shift right=3, curve={height=6pt}, from=3-2, to=2-1]
\end{tikzcd}
\qquad
\begin{tikzcd}[sep=large]
	& {\frac{\Z[x_p]}{(x^2_p-px_p)}} \\
	{\frac{\Z[x_p]}{(x^2_p-px_p)}} && {\Z} \\
	& \Z
	\arrow["{\res_{p}^{{pq}}}"{description}, from=1-2, to=2-1]
	\arrow["{\res_{q}^{{pq}}}"{description}, from=1-2, to=2-3]
	\arrow["{\res_1^{p}}"{description}, from=2-1, to=3-2]
	\arrow["{\tr_{q}^{{pq}}}"{description}, shift left=3, curve={height=-6pt}, from=2-3, to=1-2]
	\arrow["{\nm_{q}^{{pq}}}"{description, pos=0.6}, shift right=3, curve={height=6pt}, from=2-3, to=1-2]
	\arrow["{\res_{1}^{{q}}}"{description}, from=2-3, to=3-2]
	\arrow["{\tr_1^{p}}"{description}, shift left=3, curve={height=-6pt}, from=3-2, to=2-1]
	\arrow["{\nm_1^{p}}"{description}, shift right=3, curve={height=6pt}, from=3-2, to=2-1]
\end{tikzcd}
\]
\end{example}

\subsection{Restricted Tambara functors}\label{subsec:res}

The key ingredient in the definition of Tambara functors is the polynomial category $\mathbf{P}_G$ introduced in \Cref{def:polynomial category}. This category is constructed from the category $\mathsf{Set}_G$ of finite $G$-sets in a straightforward way -- the objects of $\mathbf{P}_G$ are the same as the objects of $\mathsf{Set}_G$, and the morphisms of $\mathbf{P}_G$ are certain diagrams (bispans) in $\mathsf{Set}_G$. One can construct an analogous category $\mathbf{P}_{\mathcal{C}}$ for \emph{any} category $\mathcal{C}$, so long as the composition rule for bispans can be interpreted in $\mathcal{C}$. This is equivalent to requiring that $\mathcal{C}$ be \emph{locally cartesian closed} (meaning that each slice category $\mathcal{C}/x$ is cartesian closed). Work of the third author \cite{spitz_norms_2024} shows that one may moreover replicate the notion of bi-incomplete Tambara functors in this more general setting. The resulting theory of (bi-incomplete) $\mathcal{C}$-Tambara functors is most well-behaved when $\mathcal{C}$ is a ``separable index''. This technical condition will be satisfied in all cases discussed below.

Here, we will need only to consider the case that $\mathcal{C}$ is a category defined as follows:

\begin{enumerate}
    \item Let $\mathcal{O}$ be a self-compatible transfer system on $\Sub(G)$. Let $\mathcal{O}_0$ be a connected component of $\mathcal{O}$. 
    \item $\mathcal{C}$ be the full subcategory of $\mathsf{Set}_G$ spanned by $G$-sets $X$ such that each orbit of $X$ is isomorphic to $G/K$ for some $K \in \mathcal{O}_0$.
\end{enumerate}

Categories $\mathcal{C}$ constructed this way give a \emph{complete separable index} in the sense of \cite{spitz_norms_2024}; what this entails is that the category of product-preserving functors $\mathbf{P}_{\mathcal{C}} \to \mathsf{Set}$ gives a theory of ``restricted Tambara functors'', analogous to the ordinary theory of (bi-incomplete) Tambara functors. Such a restricted Tambara functor $T$ consists of a collection of commutative rings $T(G/H)$ (one for each subgroup $H \in \mathcal{O}_0$), with structure maps between these commutative rings as usual.

\begin{definition}
Let $\mathcal{O}$, $\mathcal{O}_0$, and $\mathcal{C}$ be as above. If $T$ is a $\mathcal{O}$-Tambara functor, we write $i^*_{\mathcal{O}_0}T : \mathbf{P}_{\mathcal{C}} \to \mathsf{Set}$ for the \emph{restricted Tambara functor} determined by $\mathcal{O}_0$. 
\end{definition}

Rather than give an in-depth treatment of restricted Tambara functors here, we illustrate the notion with the following three examples which will arise in our computations. 

\begin{example}\label{ex:Gring}
Let $\mathcal{O} = \mathcal{O}_{\mathrm{triv}}$ be the trivial transfer system on $G$, so a $\mathcal{O}$-Tambara functor is a coefficient system. Let $\mathcal{O}_0$ be the component of the trivial subgroup $e$. If $T$ is a coefficient system, then $i^*_{\mathcal{O}_0}T$ consists of the commutative ring $T(G/e)$ together with the conjugation map $c: T(G/e) \to T(G/e)$, i.e., $i^*_{\mathcal{O}_0}T$ is a $G$-ring. 
\end{example}

\begin{example}\label{ex:CpnRTF}
Let $\mathcal{O}$ be a self-compatible transfer system on $C_{p^n}$ and let $\mathcal{O}_0$ be the connected component of the trivial subgroup. Self-compatibility implies that there exists some $\ell \leq n$ such that $\mathcal{O}_0$ contains $C_{p^k}$ for all $0 \leq k \leq \ell$. If $T$ is a $\mathcal{O}$-Tambara functor, then $i^*_{\mathcal{O}_0}T$ consists of a commutative ring $T(C_{p^n}/C_{p^k})$ for each $0 \leq k \leq \ell$, together with an action of $W_{C_{p^n}}(C_{p^k}) = C_{p^n}/C_{p^k}$ by ring maps, and all restrictions, norms, and transfers between them satisfying the usual axioms. In other words, $i^*_{\mathcal{O}_0}T$ looks like a $C_{p^\ell}$-Tambara functor, except that the Weyl action at level $C_{p^\ell}/C_{p^k}$ is by the larger group $C_{p^n}/C_{p^k}$ instead of $C_{p^\ell}/C_{p^k}$. 
\end{example}

\begin{example}
Let $G = C_{pq}$, $\mathcal{O} = \{ (e,C_p), (e,C_q)\}$, and $\mathcal{O}_0$ the connected component of the the trivial subgroup. If $T$ is a $\mathcal{O}$-Tambara functor, then $i^*_{\mathcal{O}_0}T$ consists of the following data:
\begin{enumerate}
\item the $C_{pq}$-ring $T(C_{pq}/e)$;
\item the $C_q$-ring $T(C_{pq}/C_p)$;
\item the $C_p$-ring $T(C_{pq}/C_q)$;
\item all possible restrictions, norms, and transfers between them satisfying the usual axioms.
\end{enumerate}
This gives an example of a restricted Tambara functor which is not just a Tambara functor for a subgroup with larger Weyl groups. 
\end{example}

\subsection{Ideals}\label{subsec:ideals}

In this section we will introduce the concept of ideals of bi-incomplete Tambara functors, generalizing the work of Nakaoka on Tambara functors \cite{nakaoka_ideals}. These definitions work equally well for the restricted Tambara functors which we defined in \cref{subsec:res}.

\begin{definition}
    Let $T$ be a bi-incomplete Tambara functor for the compatible pair $(\mathcal{O}_m, \mathcal{O}_a)$. An \emph{ideal} $I$ of $T$ is a collection of ideals $I(G/K) \subseteq T(G/K)$ closed under the allowed operations. That is, for all $K \leq H \leq G$ and for all $g \in G$ we have
    \begin{align*}
        \mathrm{res}^H_K(I(G/H)) & \subseteq I(G/K),                                           \\
        \mathrm{tr}^H_K(I(G/K))  & \subseteq I(G/H) \text { for } (K,H) \in \mathcal{O}_a,  \\
        \mathrm{nm}^H_K(I(G/K))  & \subseteq I(G/H)  \text { for } (K,H) \in \mathcal{O}_m, \\
        \mathrm{c}_{K,g}(I(G/K)) & \subseteq I(G/gKg^{-1}).
    \end{align*}
\end{definition}

\begin{example}
    \label{example:Vogeli}
    We now communicate an interesting example of a bi-incomplete Tambara ideal, due to Vogeli (personal communication) in the case that $G$ is abelian. Let $k$ be a field. Let $R_k$ denote the $k$-linear representation ring Tambara functor. Let $I$ be the levelwise collection of ideals
    \[I(G/H) = \{[V] : V \text{ is a projective } kH\text{-module}\}.\]
    Then $I$ is an $(\mathcal{O}_V, \mathcal{O}_{\mathrm{comp}})$-ideal, where $(K,H) \in \cO_V$ if and only if $\operatorname{char}(k)$ does not divide $[H:K]$.

    If $\operatorname{char}(k)$ does not divide $\lvert G \rvert$, then $\mathcal{O}_V = \mathcal{O}_{\mathrm{comp}}$ and $I = R_k$, making the situation uninteresting. However, in the ``modular case'' where $\operatorname{char}(k)$ divides $\lvert G \rvert$, we have $I \neq R_k$ and $\mathcal{O}_V \neq \mathcal{O}_{\mathrm{comp}}$ -- moreover, in this case, $I$ will \emph{not} be an $(\mathcal{O}_{\mathrm{comp}}, \mathcal{O}_{\mathrm{comp}})$-ideal.

    Note that the abelian hypothesis here is essential in the modular case as $\cO_V$ can fail to be a transfer system for non-abelian $G$. For example consider the situation of $G=S_3$ and $p=2$.
\end{example}




\begin{definition}
    Let $I$ be an ideal of a bi-incomplete Tambara functor $T$ for the compatible pair $(\mathcal{O}_m, \mathcal{O}_a)$ and let $H_1,H_2 \leq G$. Let $x \in T(G/H_1)$ and $y \in T(G/H_2)$. Define the proposition $\texttt{Q}(I,x,y)$ as follows: for all $L, K_1, K_2, H_1, H_2 \leq G$ and $g_1, g_2 \in G$ such that $K_1 \leq H_1$, $K_2 \leq H_2$, $(g_1K_1g_1^{-1} , L) \in \mathcal{O}_m$, and $(g_2K_2g_2^{-1} , L) \in \mathcal{O}_m$, we have
    \[
      \left(\mathrm{nm}^L_{g_1K_1g_1^{-1}} \circ c_{K_1,g_1} \circ \mathrm{res}_{K_1}^{H_1}(x)\right)\cdot \left(\mathrm{nm}^L_{g_2K_2g_2^{-1}} \circ c_{K_2,g_2} \circ \mathrm{res}_{K_2}^{H_2}(y)\right) \in I(G/L).
    \]
We refer to products of elements of this form as \emph{generalized products} and refer to each factor as a \emph{multiplicative translate}. 
\end{definition}

\begin{definition}
    An ideal $P$ of a bi-incomplete Tambara functor $T$ is \emph{prime} if
    \begin{enumerate}
        \item $P \neq T$, and
        \item for all subgroups $H,H'\leq G$ and elements $a \in T(G/H)$ and $b \in T(G/H')$, if $\texttt{Q}(P,a,b)$ holds, then $a \in P(G/H)$ or $b \in P(G/H')$.
    \end{enumerate}
\end{definition}

\begin{remark}
    The above definition essentially says that an ideal is prime if whenever all generalized products of two elements are in the ideal, then at least one of the elements is in the ideal. This definition of prime coincides with the other definitions of prime that one may have considered. This is proved in the complete case by Nakaoka \cite[Proposition 4.2]{nakaoka_ideals}, and one can check that the proofs go through in the incomplete case without issue.
\end{remark}

\begin{definition}
Let $T$ be a bi-incomplete Tambara functor for the compatible pair $(\mathcal{O}_m, \mathcal{O}_a)$. The \emph{spectrum} of $T$ is the collection
\[
    \Spec^{(\mathcal{O}_m, \mathcal{O}_a)}(T) \coloneq \{ P \mid P \subset T \text{ is a prime ideal} \}.
\]
    This collection is a topological space when equipped with the usual Zariski topology. In particular the closed sets are generated by the collection
    \[
V(I) = \{ P \in  \Spec^{(\mathcal{O}_m, \mathcal{O}_a)}(T) \mid I \subseteq P \}
    \]
    as $I$ ranges over the ideals of $T$.
\end{definition}

\begin{remark}\label{rem:closureofpoints}
    We extract for later that if ${P} \in \Spec^{(\mathcal{O}_m, \mathcal{O}_a)}(T)$ then
\[\overline{\{{P}\}} = \{ {Q} \in \Spec^{(\mathcal{O}_m, \mathcal{O}_a)}(T) \mid {P} \subseteq {Q} \}\]
\end{remark}

\begin{remark}
    It was recently proved in \cite{chan2026radicalsnilpotentsequivariantalgebra} that $ \Spec^{(\mathcal{O}_\mathrm{comp}, \mathcal{O}_\mathrm{comp})}(T)$ is a spectral space (i.e., is homeomorphic to $\Spec(R)$ for $R$ a commutative ring). We expect that the proof follows also for the incomplete case, but do not pursue that here.
\end{remark}

One may extend the definitions of ideals, generalized products, multiplicative translates, prime ideals, and the spectra of prime ideals to the restricted Tambara functors described in \cref{subsec:res}. Rather than make these definitions precise (which would require some rather unwieldy notation), we simply describe the two key examples which will be used in our computations. 

\begin{example}
Let $\mathcal{O}$, $\mathcal{O}_0$, and $T$ be as in \cref{ex:Gring}, so $i^*_{\mathcal{O}_0}T$ is a $G$-ring. The spectrum of the restricted Tambara functor $\Spec(i^*_{\mathcal{O}_0}T)$ is precisely the spectrum of $G$-prime ideals as discussed in \cite[Def. 4.1]{4DS}. 
\end{example}

\begin{example}
Let $\mathcal{O}$, $\mathcal{O}_0$, and $T$ be as in \cref{ex:CpnRTF}, and assume further that the Weyl action on $T(C_{p^n}/C_{p^k})$ is trivial for all $k$. Then the data of $i^*_{\mathcal{O}_0}T$ is precisely the data of a $C_{p^\ell}$-Tambara functor, and the spectrum of the restricted Tambara functor $\Spec(i^*_{\mathcal{O}_0}T)$ is precisely the Nakaoka spectrum of prime ideals in the Tambara functor obtained by restricting $T$ to the subgroup $C_{p^\ell}$. In particular, this example applies if $T$ is a constant Tambara functor (the fixed point Tambara functor of a trivial $C_{p^n}$-algebra) or the Burnside Tambara functor. 
\end{example}

\subsection{Comparison with Lewis' definition for Green functors}\label{subsec:lewis}

Recall from \cref{ex:specifics} that an $(\cO_\mathrm{triv},\cO_{\mathrm{comp}})$-Tambara functor is nothing more than a Green functor. In \cite{lewis:1980}, Lewis undertakes an in-depth study of Green functors, and in particular defines prime ideals of Green functors. In this section we will prove that Lewis' definition coincides with our definition given in the previous section, and as such we have given a valid generalization of all known constructions. For this section only we will follow Lewis' notation and write $\underline{R}$ for an arbitrary Green functor.

\begin{definition}
    Let $\underline{R}$ be a Green functor. Then a \emph{Lewis prime ideal} of $\underline{R}$ is a proper ideal $\underline{P}$ such that when $a \star b \in \underline{P}(A \times B)$ for $a \in \underline{R}(A)$ and $b \in \underline{R}(B)$ either $a \in \underline{P}(A)$ or $b \in \underline{P}(B)$ where $A$ and $B$ are $G$-sets. Here $a \star b$ denotes the product $(\res_{\pi_1} x)(\res_{\pi_2} y)$ where $\pi_1: A \times B \to A$ and $\pi_2: A \times B \to B$ are the canonical projections. Here, if $f: S \to T$ is a map of finite $G$-sets, then $\res_f$ denotes the basic bisapn \cite[Def. 2.3]{BH18} $\res_f = [T \xleftarrow{f} S \xrightarrow{=} S \xrightarrow{=} S]$. 
\end{definition}

\begin{proposition}
    Let $\underline{R}$ be a Green functor. Then Lewis prime ideals are the same as $(\cO_\mathrm{triv},\cO_{\mathrm{comp}})$-Tambara prime ideals.
\end{proposition}

\begin{proof}
    Let $\underline{P}$ be a proper ideal of $\underline{R}$.

    First, suppose $\underline{P}$ is a Lewis prime. We wish to show that $\underline{P}$ is a bi-incomplete prime. So, let $A,B$ be finite $G$-sets, let $x \in \underline{P}(A)$, $y \in \underline{P}(B)$, and suppose $\mathtt{Q}(\underline{P},x,y)$ holds. In particular, we have $x \star y \in \underline{P}(A \times B)$, so $x \in \underline{P}(A)$ or $y \in \underline{P}(B)$.

    Conversely, suppose $\underline{P}$ is a bi-incomplete prime. We wish to show that $\underline{P}$ is a Lewis prime. So, let $A,B$ be finite $G$-sets, let $x \in \underline{P}(A)$, $y \in \underline{P}(B)$, and suppose $x \star y \in \underline{P}(A \times B)$. We will prove that $\mathtt{Q}(\underline{P},x,y)$ holds. Consider an arbitrary span
    \[\begin{tikzcd}[ampersand replacement=\&]
            \& C \\
            A \&\& B
            \arrow["f"', from=1-2, to=2-1]
            \arrow["g", from=1-2, to=2-3]
        \end{tikzcd}\]
    By the universal property of products, this factors as
    \[\begin{tikzcd}[ampersand replacement=\&]
            \& C \\
            A \& {A \times B} \& B
            \arrow["f"', from=1-2, to=2-1]
            \arrow["h"', dashed, from=1-2, to=2-2]
            \arrow["g", from=1-2, to=2-3]
            \arrow["{\pi_1}", from=2-2, to=2-1]
            \arrow["{\pi_2}"', from=2-2, to=2-3]
        \end{tikzcd}\]
    Thus,
    \[(\res_f x)(\res_g y) = (\res_h \res_{\pi_1} x) (\res_h \res_{\pi_2} y) = \res_h(x \star y).\]
    Since $x \star y \in \underline{P}(A \times B)$ and $\underline{P}$ is closed under restriction, we conclude that $(\res_f x)(\res_g y) \in \underline{P}(C)$.
\end{proof}


\section{Properties of bi-incomplete Tambara ideals}\label{sec:properties}

In this section we prove several properties of bi-incomplete Tambara functors and their ideals which will be useful when we develop computational tools in the next section.

\subsection{Weyl primality}\label{subsec:weyl}

We begin by investigating the equivariant primality of the levels of prime ideals in bi-incomplete Tambara functors. 

\begin{definition}
    Let $R$ be a commutative ring with $G$-action by ring homomorphisms. An ideal $I \subseteq R$ is \emph{$G$-invariant} if $g(I) \subseteq  I$ for all $g \in G$. A $G$-invariant ideal $\mathfrak{p}$ is a \textit{$G$-prime ideal} if 
    \begin{enumerate}
        \item $\mathfrak{p} \neq R$,
        \item For $x,y \in R$, if $x(g \cdot y) \in \mathfrak{p}$ for all $g \in G$, then either $x$ or $y$ is in $\mathfrak{p}$.
    \end{enumerate}
    We denote by $\Spec_G(R)$ the collection of $G$-prime ideals of $R$.
\end{definition}

In \cite[Lemma 4.3]{4DS}, it was shown that the underlying level of any Tambara prime ideal is $G$-prime. The same proof works for bi-incomplete Tambara functors: 

\begin{lemma}\label{lem:bottomlevel}
Let $T$ be an $(\mathcal{O}_a,\mathcal{O}_m)$-Tambara functor. If $P$ is a prime ideal of $T$, then $P(G/e)$ is $G$-prime or the entire ring $T(G/e)$. 
\end{lemma}

Recall from \cref{not:biincomplete} that for any bi-incomplete Tambara functor the Weyl group $W_G(H)$ acts on $T(G/H)$. In general, if $P$ is a prime ideal in a Tambara functor, then $P(G/H)$ need not be a $W_G(H)$-prime ideal for $H \neq e$ (see \cite[Remark. 4.4]{4DS}, or the examples in \cref{part:examples}). In contrast, we have the following:

\begin{proposition}
Let $T$ be a coefficient system, that is, an $(\cO_\mathrm{triv},\cO_\mathrm{triv})$-Tambara functor. If $P$ is a prime ideal of $T$, then $P(G/H)$ is a $W_G(H)$-prime ideal of $T(G/H)$ for all subgroups $H \leq G$. 
\end{proposition}

\begin{proof}
Suppose $x,y \in T(G/H)$ with $x \cdot (gy) \in P(G/H)$ for all $g \in W_GH$. To show that $x$ or $y$ is in $P$, it suffices to show that $\mathtt{Q}(P,x,y)$ holds. Since $T$ contains no nontrivial norms, the only generalized products which we must consider have the form
$$c_{g_1,K_1} \res^H_{K_1}(x) \cdot c_{g_2,K_2} \res^H_{K_2}(x)$$
with $g_1K_1g_1^{-1} = g_2 K_2 g_2^{-1} =: K$. This may be identified with $\res^H_K(x \cdot gy)$ for an appropriate choice of $g$, and thus lies in $P$. 
\end{proof}

The proof clearly generalizes to prove the following:

\begin{corollary}
If $T$ is a bi-incomplete Tambara functor with no nontrivial norms into level $G/K$ for any $K$ subconjugate to $H$, then if $P$ is a prime ideal of $T$, the ideal $P(G/H)$ is  $W_G(H)$-prime.
\end{corollary}

\begin{remark}\label{rem:onewonders}
    It is natural to wonder if Weyl primality holds more generally. In particular one may assume that if we are considering the level $T(G/H)$, and there are no nontrivial norms into this level, then $P(G/H)$ should be a $W_G(H)$-prime regardless of what happens below the $G/H$ level. This is false: see \cref{rem:nonnontriv} for a counterexample.
\end{remark}

\subsection{Levelwise radicality}\label{subsec:radical}

We now prove that prime ideals of bi-incomplete Tambara functors are levelwise radical, extending the analogous result for Tambara functors from \cite[Theorem. 4.7]{4DS}. 

\begin{theorem}\label{thm:levelwise-radical}
Let $T$ be an $(\mathcal{O}_m,\mathcal{O}_a)$-Tambara functor. Any prime ideal of $T$ is levelwise radical. 
\end{theorem}

\begin{proof}
We will modify the strategy from \cite{4DS}. Let $P$ be a prime ideal of $T$ and suppose that $x^n \in P(G/H)$. To prove that $x \in T(G/H)$, it suffices to show that $x$ or $x^{n-1}$ is in $P(G/H)$. By primality of $P$, this holds if we can show that the statement $\mathtt{Q}(P,x,x^{n-1})$ is true, i.e., the generalized product
\begin{equation}\label{eqn:gp}
\nm^L_{g_1K_1g_1^{-1}} \circ c_{K_1,g_1} \circ \res^H_{K_1}(x) \cdot \nm^L_{g_2K_2g_2^{-1}} \circ c_{K_2,g_2} \circ \res^H_{K_2}(x^{n-1}) \in T(G/L)
\end{equation}
is in $P(G/L)$ for all admissible choices of $g_1,g_2 \in G$, $K_1,K_2 \in H$, and $L$. 

For this, it suffices to prove that each multiplicative translate (each factor) is in $P(G/L)$. Following \cite[Def. 4.9]{4DS}, for any $y \in T(G/L)$, we define
$$P(y) := \sup\left\{ m \in \mathbb{Z} : y = \prod_{i=1}^m \nm^L_{g_i K_i g_i^{-1}} c_{K_i,g_i} \res^H_{K_i}(x)\right\},$$
i.e., $P(y)$  is the supremum over the number of terms in a factorizations of $y$ into a product of multiplicative
translates of $x$. Since each factor of \eqref{eqn:gp} has $P(y) \geq 1$, the analogue of \cite[Lem. 4.10]{4DS} reduces us to showing that there exists integers $M(K)$ for each subgroups $K \leq G$ such that for any $y \in T(G/K)$, $P(y) > M(K)$ implies $y \in P(G/K)$. 

As in \cite{4DS}, we proceed by induction up the subgroup lattice, with the induction hypothesis at a subgroup $K \leq G$ that $P(G/M)$ is radical for all proper subgroups $M < K$. The base case, $K=e$, holds since $P(G/e)$ is a radical ideal. 

The induction step is divided into three cases: $H$ not subconjugate to $K$, $H$ and $K$ conjugate, and $H$ properly subconjugate to $K$. The proofs of all three statements from the complete setting \cite[Lems. 4.12-4.14]{4DS} can be used without modification. One might worry about the appearance of norms in the proofs of \cite[Lems. 4.13, 4.14]{4DS} since $T$ may not contain the same norms. However, all of these norms are instantiated by the hypothesis that $P(y) > 0$, i.e., $y$ is a nonempty product of multiplicative translates of $x$, and not by the presence of any particular norms in $T$. Since the analogous arguments in the bi-incomplete setting would also assume $P(y)>0$, the required norms are forced to exist. 
\end{proof}

\subsection{Change-of-pairs and restriction}\label{subsec:chaingeofpairs}

Of interest to the results of this paper is how $\Spec^{(\mathcal{O}_m, \mathcal{O}_a)}(T)$ changes as we traverse the collection of different compatible pairs. In this section we will discuss how the (prime) ideals change as we increase both the additive and multiplicative parts of a compatible pair.

To this end, suppose that $(\mathcal{O}_m, \mathcal{O}_a)$ is a sub-compatible pair of $(\mathcal{O}_m', \mathcal{O}_a')$. Let $T$ be an $(\mathcal{O}_m', \mathcal{O}_a')$-Tambara functor, and abusing notation, let $T$ also denote $i^*T$, i.e., $T$ regarded as an $(\mathcal{O}_m, \mathcal{O}_a)$-Tambara functor.

\begin{lemma}
    Any $(\mathcal{O}_m', \mathcal{O}_a')$-ideal of $T$ is also an $(\mathcal{O}_m,\mathcal{O}_a)$-ideal. 
\end{lemma}

\begin{proof}
    This follows immediately from the observation that the transfers, norms, and restrictions in an $(\mathcal{O}_m,\mathcal{O}_a)$-Tambara functor are a subset of the transfers, norms, and restrictions in an $(\mathcal{O}_m',\mathcal{O}_a')$-Tambara functor. 
\end{proof}


If we only change the additive transfer system $\mathcal{O}_a$, then we can also relate \emph{prime} ideals: 

\begin{proposition}\label{Prop:ChangeOaPrimes}
    Let $(\mathcal{O}_m, \mathcal{O}_a)$ be a sub-compatible pair of $(\mathcal{O}_m, \mathcal{O}_a')$ and let $T$ be an $(\mathcal{O}_m,\mathcal{O}_a')$-Tambara functor. Each $(\mathcal{O}_m,\mathcal{O}_a')$-prime ideal of $T$ is also a $(\mathcal{O}_m,\mathcal{O}_a)$-prime ideal of $T$. Conversely, every $(\mathcal{O}_m,\mathcal{O}_a)$-prime ideal of $T$ which is closed under the transfers in $\mathcal{O}_a'$ is an $(\mathcal{O}_m,\mathcal{O}_a')$-prime ideal. 
\end{proposition}

\begin{proof}
    Since transfers do not appear in generalized products,
    \[
        \texttt{Q}^{(\mathcal{O}_m,\mathcal{O}_a)}(P,x,y) \iff \texttt{Q}^{(\mathcal{O}_m,\mathcal{O}_a')}(P,x,y). 
    \]
    The result follows.
\end{proof}

\begin{corollary}
    Let $T$ be a Green functor. The prime ideals in $T$ are precisely those prime ideals in the underlying coefficient system of $T$ which are closed under all transfers.     
\end{corollary}

\begin{proof}
    Green functors and coefficient systems both have $\mathcal{O}_m$ trivial, so the corollary follows immediately from Proposition~\ref{Prop:ChangeOaPrimes}.
\end{proof}    

If we change the multiplicative indexing system $\mathcal{O}_m$, then we have far less control of the situation given that the $\texttt{Q}$ proposition includes norms. We do however have the following:

\begin{proposition}\label{prop:add_mult}
    Let $(\mathcal{O}_m,\mathcal{O}_a)$ be a sub-compatible pair of $(\mathcal{O}_m',\mathcal{O}_a)$ and let $T$ be an $(\mathcal{O}_m',\mathcal{O}_a)$-Tambara functor. Every $(\mathcal{O}_m,\mathcal{O}_a)$-prime ideal of $T$ which is closed under the norms in $\mathcal{O}_m'$ is an $(\mathcal{O}_m',\mathcal{O}_a)$-prime ideal. 
\end{proposition}

\begin{proof}
    The set of generalized products for $(\mathcal{O}_m,\mathcal{O}_a)$ is a subset of the set of generalized products for $(\mathcal{O}_m',\mathcal{O}_a)$, so 
    \[
        \texttt{Q}^{(\mathcal{O}_m',\mathcal{O}_a)}(P,x,y) \implies \texttt{Q}^{(\mathcal{O}_m,\mathcal{O}_a)}(P,x,y). 
    \] 
\end{proof}

\subsection{Cohomological Tambara functors and interaction with saturated hulls}\label{subsec:hulls}

Suppose that $(\cO_m,\cO_a)$ is a compatible pair. By \cref{lem:hull} we have that $(\mathrm{Hull}(\cO_m),\cO_a)$ is also a compatible pair. The goal of this section is to identify conditions on an $(\cO_m,\cO_a)$-Tambara functor $T$ so that the prime ideals are insensitive to taking the saturated hull of the multiplicative part. That is, we wish to identify conditions on $T$ such that
\[
\Spec^{(\cO_m,\cO_a)}(T) \cong \Spec^{(\mathrm{Hull}(\cO_m),\cO_a)}.
\]
This is not true in general. We refer the reader to \cref{rem:itissensitive} for an explicit example in the Burnside Tambara functor of $C_{p^2}$ where these two spaces differ.

\begin{definition}\label{def:coh}
    Let $T$ be a Tambara functor. Then $T$ is:
    \begin{itemize}
        \item \emph{additively cohomological} if
        \[\tr_K^H \res_K^H x = [H:K] x\]
        for all subgroups $1 \leq K < H \leq G$.
        \item \emph{multiplicatively cohomological} if
        \[\nm_K^H \res_K^H x = x^{[H:K]}\]
        for all subgroups $1 \leq K < H \leq G$.
    \end{itemize}
\end{definition}

\begin{remark}
    The above definition is technically more restricted than we need for a general treatment. Indeed, if we consider an $(\cO_m,\cO_a)$-Tambara functor, then we should ask that the conditions of \cref{def:coh} hold only for the allowed transfers and norms. We would still be able to prove the results that  follow in this generality, but prefer not to in order to avoid excessive bookkeeping.
\end{remark}

\begin{example}\label{ex:iscohmult}
    If $G$ is nontrivial then the Burnside Tambara functor $\underline{A}_G$ is neither additively or multiplicatively cohomological. On the other hand, for a $G$-ring $R$, the fixed point Tambara functor $\mathrm{FP}(R)$ is both additively and multiplicatively cohomological.
\end{example}

\begin{definition}
    Let $T$ be a $G$-Tambara functor.
    \begin{itemize}
    \item We define $T^\mathrm{add}$ to be the quotient of $T$ by the Tambara ideal generated by the relations \[\tr_K^H \res_K^H x - [H:K]x\] for all subgroups $1 \leq K < H \leq G$ and all $x \in T(G/H)$.
    \item We define $T^\mathrm{mult}$ to be the quotient of $T$ by the Tambara ideal generated by the relations \[\nm_K^H \res_K^H x - x^{[H:K]}\] for all subgroups $1 \leq K < H \leq G$ and all $x \in T(G/H)$.
    \end{itemize}
\end{definition}

The following lemma is immediate from the construction.

\begin{lemma}
    Let $T$ be a Tambara functor. Then
    \begin{itemize}
        \item $T^{\mathrm{add}}$ is an additively cohomological Tambara functor and $T = T^{\mathrm{add}}$ if and only if $T$ is additively cohomological.
        \item $T^{\mathrm{mult}}$ is an multiplicatively cohomological Tambara functor and $T = T^{\mathrm{mult}}$ if and only if $T$ is multiplicatively cohomological.
    \end{itemize}
\end{lemma}

\begin{example}
    Let $p$ be prime and $G=C_p$. Consider the free $C_p$-Tambara functor on a fixed generator from \cite[Lemma 3.6]{MR3959596}:
\[\underline{A}[C_p/C_p] =\begin{tikzcd}[row sep=large]
	{\Z[t,x,n]/ \langle t^2-pt,tn-tx^p \rangle} \\
	{\Z[x]}
	\arrow["\res"{description}, from=1-1, to=2-1]
	\arrow["\tr", shift left=3, curve={height=-6pt}, from=2-1, to=1-1]
	\arrow["\nm"', shift right=3, curve={height=6pt}, from=2-1, to=1-1]
\end{tikzcd}\]
where
\[
\res(t)=p \qquad \res(x)=x \qquad \res(n)=x^p \qquad \tr(f)=tf
\]
and the norm is determined by $\nm(k) = k+ \frac{k^p-k}{p}t$ and $\nm(x)=n$. This Tambara functor is neither additively or multiplicatively cohomological.

We can form an additively cohomological Tambara functor from this by taking the quotient by the relations $\tr(\res(\alpha)) = p\alpha$ where $\alpha \in \{t,x,n\}$ which are:
\begin{align*}
    \tr(\res(t)) = tp &\Rightarrow tp-tp \\
    \tr(\res(x)) = tx &\Rightarrow tx-xp \\
    \tr(\res(n)) = tx^p &\Rightarrow tx^p-np 
\end{align*}
The consequence of adding these relations is that $t=p$ yielding
\[\underline{A}[C_p/C_p]^{\mathrm{add}} =\begin{tikzcd}[row sep=large]
	{\Z[x,n]/\langle pn-px^p \rangle} \\
	{\Z[x]}
	\arrow["\res"{description}, from=1-1, to=2-1]
	\arrow["\tr", shift left=3, curve={height=-6pt}, from=2-1, to=1-1]
	\arrow["\nm"', shift right=3, curve={height=6pt}, from=2-1, to=1-1]
\end{tikzcd}\]
which is additively cohomological but not multiplicatively cohomological; indeed, $\nm(\res(x))=n \neq x^p$.    

We can form a multiplicatively cohomological Tambara functor via a similar process:
\begin{align*}
    \nm(\res(t)) = p + (p^{p-1}-1)t &\Rightarrow p + (p^{p-1}-1)t - t^p \\
    \nm(\res(x)) = n &\Rightarrow n-x^p \\
    \nm(\res(n)) = n^p &\Rightarrow n^p-n^p 
\end{align*}
By the first relation we once again have that $t=p$, and the second relation supersedes the relation $pn - px^p$ giving us
\[\underline{A}[C_p/C_p]^{\mathrm{mult}} =\begin{tikzcd}[row sep=large]
	{\Z[x,n]/\langle n-x^p \rangle} \\
	{\Z[x]}
	\arrow["\res"{description}, from=1-1, to=2-1]
	\arrow["\tr", shift left=3, curve={height=-6pt}, from=2-1, to=1-1]
	\arrow["\nm"', shift right=3, curve={height=6pt}, from=2-1, to=1-1]
\end{tikzcd}\]
Clearly this multiplicatively cohomological Tambara functor is also additively cohomological. This is not a coincidence as we will prove momentarily. 
\end{example}

\begin{lemma}\label{lem:burnsideadd}
    Let $G$ be a finite group and $\underline{A}_G$ the Burnside Tambara functor. Then $\underline{A}^{\mathrm{add}}_G = \underline{\Z}$.
\end{lemma}
\begin{proof}
    Imposing $\tr_K^H(\res_K^H(1)) = [H:K]$ says exactly that the class of $H/K$ in $A(H) = \underline{A}_G(G/H)$ becomes identified with the integer $[H:K]$. Thus, $\underline{A}^{\mathrm{add}}_G$ is a quotient of $\underline{\Z}$. Since $\underline{\Z}$ is additively cohomological, we are done.
\end{proof}

\begin{lemma}\label{lemma:norm of integer from normal subgroup}
    Let $G$ be a finite group, and let $H$ be a normal subgroup of $G$. Let $n$ be a natural number. Then in the Burnside Tambara functor $\uA_G$, $\nm_H^G(n)$ is a sum of (isomorphism classes of) orbits of type $G/L$ with $L \geq H$.
\end{lemma}
\begin{proof}
    Recall that $\nm_H^G(n)$ is the isomorphism class of the set $\operatorname{Hom}_H(G, \{1,\dots,n\})$, on which $G$ acts by precomposition with its right translation action, and $\{1,\dots,n\}$ has trivial $H$-action. Let $f \in \operatorname{Hom}_H(G, \{1,\dots,n\})$ and $x \in H$ be arbitrary. Let $g \in G$ be arbitrary. By normality of $H$, we have $gx = hg$ for some $h \in H$. Thus
    \[(x \cdot f)(g) = f(gx) = f(hg) = hf(g) = f(g),\]
    and since $g$ was arbitrary we conclude that $x \cdot f = f$. Since $x$ was arbitrary, we conclude that the stabilizer of $f$ contains $H$.
\end{proof}

\begin{corollary}\label{cor:norm of integer from normal maximal}
    Let $G$ be a finite group, let $H$ be a maximal proper subgroup of $G$ which is also normal in $G$, and let $n$ be a natural number. Then in the Burnside Tambara functor $\uA_G$,
    \[\nm_H^G(n) = n + \frac{n^{[G:H]}-n}{[G:H]}[G/H].\]
\end{corollary}
\begin{proof}
    By \Cref{lemma:norm of integer from normal subgroup}, we have
    \[\nm_H^G(n) = a + b[G/H]\]
    for some natural numbers $a$ and $b$. On the other hand, $|(\nm_H^G(n))^e| = n^{[G:H]}$ and $|(\nm_H^G(n))^H| = n$, so 
    \begin{align*}
        a + [G:H] b &= n^{[G:H]}, \\
        a &= n.
    \end{align*}
    The result follows. 
\end{proof}


\begin{theorem}\label{thm:burnsidemult}
    Let $G$ be a nilpotent finite group and $\uA_G$ the Burnside $G$-Tambara functor. Then $\uA^\mathrm{mult}_G = \underline{\Z}$.
\end{theorem}
\begin{proof}
    Recall that a finite group $G$ is nilpotent if and only if all maximal subgroups of $G$ are normal.

    We must show that $[H/K] = [H:K]$ in $\underline{A}^\mathrm{mult}_G(G/H)$ for all $K \leq H \leq G$. We proceed by strong induction on $H$ in the subgroup lattice of $G$.

    So, let $H \leq G$ be arbitrary and suppose the claim holds for all $H' < H$. We now claim that it suffices to show $[H/K] = [H:K]$ in $\underline{A}^\mathrm{mult}_G(G/H)$ for $K$ a maximal subgroup of $H$. To see this, we note first that $[H/H] = 1 = [H:H]$ in $\underline{A}_G(G/H)$ and thus also in $\underline{A}^\mathrm{mult}_G(G/H)$. Then, for $K$ a proper subgroup of $H$, let $K'$ be a maximal subgroup of $H$ such that $K \leq K'$. We have $[K'/K] = [K' : K]$ in $\underline{A}^\mathrm{mult}_G(G/K')$ by inductive hypothesis, and thus
    \[[H/K] = \tr_{K'}^H [K'/K] = \tr_{K'}^H [K':K] = [K':K] [H/K'] = [K':K] [H:K'] = [H:K]\]
    in $\underline{A}^\mathrm{mult}_G(G/H)$, as desired.

    Now let $K$ be a maximal subgroup of $H$. Since $G$ is nilpotent, $H$ is also nilpotent, so $K$ is normal in $H$. By \Cref{cor:norm of integer from normal maximal} and the inductive hypothesis, we have
    \[\nm_K^H \res_K^H [H/K] = \nm_K^H([H:K]) = [H:K] + ([H:K]^{[H:K]-1} - 1)[H/K]\]
    in $\uA^\mathrm{mult}_G(G/H)$. Moreover, the normality of $K$ tells us that
    \[[H/K]^2 = [H:K][H/K]\]
    in $\uA_G(G/H)$ (and hence also in $\uA^\mathrm{mult}_G(G/H)$). Now the condition of being multiplicatively cohomological says that
    \[[H:K] + ([H:K]^{[H:K]-1} - 1)[H/K] = [H/K]^{[H:K]} = [H:K]^{[H:K]-1}[H/K]\]
    in $\underline{A}^\mathrm{mult}_G(G/H)$. We conclude $[H:K] = [H/K]$ in $\underline{A}^\mathrm{mult}_G(G/H)$, as desired.
\end{proof}

We expect that $\uA_G^\mathrm{mult} = \underline{\Z}$ holds for \emph{all} finite groups $G$, and leave this open as a conjecture.

\begin{conjecture}\label{conj:burnsidemult}
    Let $G$ be a finite group, and $\underline{A}_G$ the Burnside $G$-Tambara functor. Then $\underline{A}^{\mathrm{mult}} = \underline{\Z}$.
\end{conjecture}

\begin{corollary}
    Let $G$ be a group such that $\uA_G^\mathrm{mult} = \underline{\Z}$ (e.g., $G$ nilpotent). Let $T$ be a complete $G$-Tambara functor which is multiplicatively cohomological. Then $T$ is additively cohomological.
\end{corollary}

\begin{proof}
    Consider the unique map $\underline{A}_G \to T$. We can form the following diagram
\[\begin{tikzcd}
	{\underline{A}_G} & T \\
	{\underline{A}_G^\mathrm{mult}} & {T^\mathrm{mult}} \\
	{\underline{\Z}} & T
	\arrow[from=1-1, to=1-2]
	\arrow[from=1-1, to=2-1]
	\arrow[from=1-2, to=2-2]
	\arrow[from=2-1, to=2-2]
	\arrow[equals, from=2-1, to=3-1]
	\arrow[equals, from=2-2, to=3-2]
	\arrow[from=3-1, to=3-2]
\end{tikzcd}\]
where the right hand equality is by our hypothesis on $T$. The map $\underline{\Z} \to T$ equips $T$ with the structure of a $\underline{\Z}$-module, which is equivalent to $T$ being additively cohomological (see \cite[Proposition 16.3]{tw_mackey}).
\end{proof}

\begin{theorem}\label{thm:insensitive}
    Let $T$ be a multiplicatively cohomological $(\cO_\mathrm{comp},\cO_\mathrm{comp})$-Tambara functor. Then for any compatible pair $(\cO_m,\cO_a)$ we have a homeomorphism
    \[
    \Spec^{(\cO_m,\cO_a)}(T) \cong \Spec^{(\mathrm{Hull}(\cO_m),\cO_a)}(T)
    \]
\end{theorem}

\begin{proof}
    Let $P$ be an $(\mathcal{O}_m,\mathcal{O}_a)$-prime ideal of $T$. We will show that $P$ is closed under all norms in $\operatorname{Hull}(\cO_m)$, yielding the result by \cref{prop:add_mult}.

    Suppose that $\cO_m$ is not saturated (if it is, then $\cO_m = \operatorname{Hull}(\cO_m)$ and we are done). By the virtue of $\cO_m$ not being saturated, we can find $A < B < C \leq G$ such that $(A,C) \in \cO_m$ and $(B,C) \not\in \cO_m$. We have that $(B,C) \in \mathrm{Hull}(\cO_m)$, and in fact it suffices to show that $P$ is closed under the norms $\nm_B^C$ of this form. That is, we will prove that $\nm_B^C P(T/B) \subseteq P(T/C)$.

    To this end, take $\beta \in P(T/B)$. Consider the term $\nm_A^C \res_A^C \nm_B^C(\beta) \in T(G/C)$. As we have assumed that $T$ is multiplicatively cohomological, this expression is equal to $\nm_B^C(\beta)^{[C:A]}$. Using the double coset formula we have that
    \[
    \res_A^C \nm_B^C(\beta) = \prod_{AgB \in A\backslash C/B} \nm^A_{A \cap gBg^{-1}} \res^{gBg^{-1}}_{A \cap gBg^{-1}} c_g(\beta)
    \]
    and we observe that the right-hand side lies in $P(G/A)$, and as such so does the left. Indeed, it suffices to observe that at the identity coset we obtain $\nm_A^A \res^B_A(\beta)$ which is certainly in $P(G/A)$, and then we use the fact that ideals are absorbing under multiplication.
    As $(A,C) \in \cO_m$ we have that
    \[
    \nm_A^C(\res_A^C \nm_B^C(\beta)) = \nm_B^C(\beta)^{[C:A]} \in P(G/C).
    \]
    As $P$ was prime for $\cO_m$ we have that $P(G/C)$ is levelwise radical by \cref{thm:levelwise-radical} and hence $\nm_B^C(\beta)\in P(T/C)$ as required.
\end{proof}

\begin{remark}
    While being multiplicatively cohomological is sufficient for the result of \cref{thm:insensitive} we do not know if it is necessary.
\end{remark}

\begin{example}
    We provide an example to show that \cref{thm:insensitive} cannot be weakened to an equality between all ideals, and that primality is required.  Let $G = C_4$ and consider the ring $\Z[x_1,\dots, x_4]$ equipped with the $C_4$-action which cycles the variables. We can then construct the fixed point Tambara functor of this ring (\cref{ex:fp}) (which is multiplicatively cohomological) and consider the following subset of it
\[\begin{tikzcd}[column sep=tiny]
	{\langle x_1x_3+x_2x_4,(x_1x_3x_2x_4)^2 \rangle} & \subseteq & {\Z[x_1,\dots,x_4]^{C_4}} \\
	{\langle x_1x_3,x_2x_4 \rangle} & \subseteq & {\Z[x_1,\dots,x_4]^{2C_4}} \\
	{\langle x_1x_3,x_2x_4\rangle} & \subseteq & {\Z[x_1,\dots,x_4]}
	\arrow[from=1-1, to=2-1]
	\arrow[from=1-3, to=2-3]
	\arrow[shift left=3, curve={height=-6pt}, from=2-1, to=1-1]
	\arrow[shift right=3, curve={height=6pt}, from=2-1, to=1-1]
	\arrow[from=2-1, to=3-1]
	\arrow[shift left=3, curve={height=-6pt}, from=2-3, to=1-3]
	\arrow[shift right=3, curve={height=6pt}, from=2-3, to=1-3]
	\arrow[from=2-3, to=3-3]
	\arrow[shift left=3, curve={height=-6pt}, from=3-1, to=2-1]
	\arrow[shift right=3, curve={height=6pt}, from=3-1, to=2-1]
	\arrow[shift left=3, curve={height=-6pt}, from=3-3, to=2-3]
	\arrow[shift right=3, curve={height=6pt}, from=3-3, to=2-3]
\end{tikzcd}\]
Note that this is not an ideal for the complete multiplicative transfer system as
\[
\nm_{C_p}^{C_{p^2}}(x_1x_3)=(x_1x_2x_3x_4) \not\in {\langle x_1x_3+x_2x_4,(x_1x_3x_2x_4)^2 \rangle}
\]
It is, however, an ideal for the non-saturated multiplicative transfer system from \cref{ex:o3example}. This does not contradict \cref{thm:insensitive} as the ideal is not prime. Indeed the top level is not a radical ideal but by \cref{thm:levelwise-radical} bi-incomplete Tambara primes are always levelwise radical.
\end{example}

\section{Path component decomposition of bi-incomplete prime ideals}\label{sec:pathdecomp}

In this section we will present a method for computing the spectra of bi-incomplete Tambara functors using the tools of \cref{sec:properties}. Our strategy relies on the observation that if $\cO$ is a saturated transfer system, then the prime ideals of the self-compatible pair $(\cO,\cO)$ are easily understood. We can then use \cref{Prop:ChangeOaPrimes} to adjust the additive structure to obtain other compatible pairs. Under the assumption of $T$ being multiplicatively cohomological, by virtue of \cref{thm:insensitive} we are able to easily compute the spectra for all possible compatible pairs. Of course, it is not always the case that a Tambara functor is multiplicatively cohomological; in \cref{subsec:cp2burnside}, we will demonstrate how to deal with the case when we are not multiplicatively cohomological.

\subsection{Self-compatible transfer systems}

In \cref{defn:transfersystem} we defined a transfer system for $G$ as a subset $\cO \subseteq \Sub(G) \times \Sub(G)$ satisfying various conditions. We can identify these pairs with a directed subgraph of the lattice $\Sub(G)$ in an obvious way.

Let $\mathcal{O}$ be a self-compatible transfer system and let $T$ be an $(\mathcal{O},\mathcal{O})$-Tambara functor.

\begin{definition}
Let $\pi_0\mathcal{O}$ denote the set of path components of $\mathcal{O}$ (regarded as a directed subgraph of the subgroup poset). For any $\mathcal{O}' \in \pi_0 \mathcal{O}$, let $i^*_{\mathcal{O}'} T$ denote the restriction of $T$ to $\mathcal{O}'$ as a restricted Tambara functor. 
\end{definition}

We will need three observations about self-compatible transfer systems. The first lemma says that each path component in $\mathcal{O}$ ``looks complete". 

\begin{lemma}
If $\mathcal{O}' \in \pi_0 \mathcal{O}$ and $H, K \in \mathcal{O}'$ with $K \leq H$, then $(K,H) \in \mathcal{O}'$. 
\end{lemma}

\begin{proof}
Since $\mathcal{O}$ is self-compatible, it is saturated, and thus each path component $\mathcal{O}' \in \pi_0 \mathcal{O}$ is saturated. If $H, K \in \mathcal{O}'$ with $K \leq H$, then $H$ and $K$ are connected by a chain of pairs $(L_i,L_{i+1})$. Induction on the length of the chain together with saturation gives the desired element. 
\end{proof}

Our second lemma explains how the position of $\mathcal{O}'$ within $\mathcal{O}$ prohibits the existence of certain pairs outside of $\mathcal{O}'$. 

\begin{lemma}\label{lemma:compare}
Suppose $\mathcal{O}$ is self-compatible and $\mathcal{O}' \in \pi_0 \mathcal{O}$. Suppose $H \leq G$ contains a subgroup $K  \in \mathcal{O}'$ and $L \leq G$ is another subgroup such that $(L , H) \in \mathcal{O}$. Then $(L \cap K , K) \in \mathcal{O}'$. 
\end{lemma}

\begin{proof}
The restriction axiom for transfer systems implies that $(L \cap K , H \cap K ) = (L \cap K , K ) \in \mathcal{O}$. Since $K \in \mathcal{O}'$, this implies that $(L \cap K, K) \in \mathcal{O}'$. 
\end{proof}

Our final lemma says that each path-component has a minimal element. 

\begin{lemma}
Let $\mathcal{O}' \in \pi_0 \mathcal{O}$. There exists a unique subgroup $K \in \mathcal{O}'$ such that $K \leq H$ for all $H \in \mathcal{O}'$. 
\end{lemma}

\begin{proof}
Existence of minimal subgroups in $\mathcal{O}'$ is clear, so it suffices to show uniqueness. Suppose $H$ and $K$ are minimal subgroups in $\mathcal{O}'$. Since they are in the same path component, there exists $L$ such that $(H , L)$ and $(K , L)$ are both in $\mathcal{O}'$. Intersecting the first norm with $K$ and the second with $H$ shows that $(K \cap H, K \cap H)$ is in $\mathcal{O}'$. Since $K$ and $H$ were minimal, we conclude $K=H$. 
\end{proof}

\subsection{Spectra over self-compatible pairs}

Our goal is to prove the following:

\begin{theorem}\label{thm:pi0}
Let $\mathcal{O}$ be a self-compatible transfer system and let $T$ be an $(\mathcal{O},\mathcal{O})$-Tambara functor. There is a bijection
$$\coprod_{\mathcal{O}' \in \pi_0 \mathcal{O}} \Spec(i^*_{\mathcal{O}'}T) \cong \Spec^{(\cO,\cO)}(T)$$
where $\Spec(i^*_{\mathcal{O}'}T)$ is the spectrum of the restricted Tambara functor $i^*_{\mathcal{O}'}T$. \end{theorem}

\begin{example}
Let $\mathcal{O} = \cO_{\mathrm{comp}}$, so that $T$ is a (complete) Tambara functor. Then $\pi_0 \mathcal{O} = \{ \mathcal{O} \}$ and the theorem is a tautology. 
\end{example}

\begin{example}\label{ex:coefficientsystem}
If $\mathcal{O} = \cO_{\mathrm{triv}}$, so $T$ is a coefficient system, then $\pi_0 \mathcal{O} \cong \Sub(G)/G$, where a conjugacy class of subgroup comes equipped with its Weyl group action. We then obtain a homeomorphism
   $$\coprod_{[H] \leq G} \Spec_{W_G(H)}(T(G/H)) \xrightarrow{\cong} \Spec^{(\cO_\mathrm{triv},\cO_\mathrm{triv})}(T).$$
\end{example}

Our strategy to prove \Cref{thm:pi0} is as follows. First, we will show that every prime in $\Spec(i^*_{\mathcal{O}'}T)$ extends to a prime in $\Spec(T)$, giving a map from the left-hand side to the right-hand side. We will then show that this map is a bijection, i.e., that every prime in $\Spec^{(\cO,\cO)}(T)$ arises uniquely in this way. 

\begin{definition}\label{def:IJ}
Let $\mathcal{O}' \in \pi_0\mathcal{O}$ and let $J \in \Spec(i^*_{\mathcal{O}'}T)$ be a prime $i^*_{\mathcal{O}'}$-ideal. Define $I(J) \subseteq T$ by
\[
I(J)(G/K) := 
\begin{cases}
J(G/K) \quad & \text{ if } K \in \mathcal{O}', \\
\bigcap_{L \in \mathcal{O}', \  L \leq K} (\res^K_L)^{-1}(J(G/L)) \quad & \text{ otherwise},
\end{cases}
\]
where we take the convention that the empty intersection is $(1)$. 
\end{definition}

\begin{proposition}
The subset $I(J) \subseteq T$ is an $(\mathcal{O},\mathcal{O})$-ideal. 
\end{proposition}

\begin{proof}
It is clear that $I(J)$ is levelwise an ideal, so we only need to check closure under the operations parametrized by $\mathcal{O}$. Closure under restriction and conjugation is clear from the definition.

To see closure under transfers and norms, we break into cases. Let $H \leq G$:

\begin{itemize}

\item If $H \in \mathcal{O}'$, then the only allowable transfers and norms from $H$ have target $H' \in \mathcal{O}'$. At these levels, $I(J)$ is just $J$, which is closed under these operations. 

\item Suppose $H \notin \mathcal{O}'$ but $H \geq L_1,\ldots,L_n$ for some maximal list of subgroups $L_1,\ldots,L_n \in \mathcal{O}'$ with $n \geq 1$. Any allowable transfer or norm from $H$ lands in a subgroup $H'$ containing $H$, so $H'$ also contains $L_1, \ldots, L_n$ and thus $I(J)(G/H') = \bigcap_{i=1}^n (\res^{H'}_{L_i})^{-1}(J(G/L_i))$. Since $\res^{H'}_{L_i} = \res^{H'}_H \circ \res^H_{L_i}$, we may rewrite this as $I(J)(G/H') = (\res^{H'}_H)^{-1}(I(J)(G/H))$. This is clearly closed under transfers and norms. 

\item Suppose $H \notin \mathcal{O}'$ and that $H \not\geq L$ for any $L \in \mathcal{O}'$, so $I(J)(G/H) = (1)$. By \Cref{lemma:compare}, there are no transfers or norms from $H$ to a subgroup containing a subgroup in $\mathcal{O}'$, so any transfer or norm from $H$ lands in a subgroup $K$ for which $I(J)(K)$ also equals $(1)$. 

\end{itemize}
\end{proof}

\begin{proposition}\label{prop:lefttoright}
The $(\mathcal{O},\mathcal{O})$-ideal $I(J) \subseteq T$ is $(\mathcal{O},\mathcal{O})$-prime.
\end{proposition}

\begin{proof}
We will verify that $\mathtt{Q}(I(J),x,y)$ implies $x \in I(J)$ or $y \in I(J)$ for all $x,y \in T$, so suppose that $x \in T(G/H)$, $y \in T(G/K)$, and that $\mathtt{Q}(I(J),x,y)$. We divide into cases:

\begin{itemize}

\item Suppose $H,K \in \mathcal{O}'$. The only generalized products of $x$ and $y$ are restrictions of generalized products contained in $i^*_{\mathcal{O}'}T$, and since $J$ was assumed to be an $i^*_{\mathcal{O}'}T$-prime ideal, it follows that $\mathtt{Q}(I(J),x,y)$ implies $x \in I(J)$ or $y \in I(J)$. 

\item Suppose $H$ does not contain any subgroups in $\mathcal{O}'$. Then $I(J)(G/H) = (1)$ and thus $x \in I(J)(G/H)$. Similarly, if $K$ does not contain any subgroups in $\mathcal{O}'$, then $y \in I(J)(G/K)$. 

\item Suppose $H \in \mathcal{O}'$ and that $K \notin \mathcal{O}'$ but that $K$ does contain a subgroup in $\mathcal{O}'$. Then the only generalized products of $x$ and $y$ are products of restrictions of $x$ and restrictions of norms of $y$. Primality of $J$ implies that $x$ or $y$ are in $I(J)$. 

\item Suppose $H, K \notin \mathcal{O}'$ but that both contain subgroups in $\mathcal{O}'$. Then both contain the minimal subgroup $L$ of $\mathcal{O}'$. Since $I(J)(G/H) \subseteq (\res^H_L)^{-1}(J(G/L))$ and $I(J)(G/K) \subseteq (\res^K_L)^{-1}(J(G/L))$, we have $\res^H_L(x) \res^K_L(y) \in J(G/L)$ and thus either $\res^H_L(x)$ or $\res^K_L(y)$ is in $J(G/L)$. This implies that $x \in I(J)(G/H)$ or $y \in I(J)(G/K)$. 

\end{itemize}
\end{proof}

\begin{proposition}\label{prop:itssurjective}
Let $M$ be an $(\mathcal{O},\mathcal{O})$-prime ideal of $T$. Then $M = I(J)$ for some prime ideal $J$ of $i^*_{\mathcal{O}'}T$ for some $\mathcal{O}' \in \pi_0 \mathcal{O}$. 
\end{proposition}

\begin{proof}
Since $M$ is prime, there exists a subgroup $H \leq G$ such that $M(G/H) \neq (1)$. Let $\mathcal{O}'$ be the path-component containing $H$ and let $J$ be the restriction of $M$ to $i^*_{\mathcal{O}'}T$. Then primality of $M$ implies primality of $J$, and we claim that $M = I(J)$. To see this, we will check that their values agree on $G/K$ for all subgroups $K \leq G$. As usual, we divide into cases:

\begin{itemize}

\item If $K \in \mathcal{O}'$, then $I(J)(G/K) = J(G/K) = M(G/K)$ by definition. 

\item If $K$ does not contain any subgroups in $\mathcal{O}'$, then $I(J)(G/K) = (1)$. Suppose $M(G/K) \neq (1)$ and let $y \notin M(G/K)$. Let $x \notin M(G/H)$. Our hypotheses on $K$ imply that the only generalized products of $x$ and $y$ lie in proper subgroups of $H$, and minimality of $H$ implies $M$ is $(1)$ at all of these levels. Thus $\mathtt{Q}(M,x,y)$ holds, which is a contradiction. Thus $M(G/K) = (1)$.

\item If $K \notin \mathcal{O}'$ but $K$ contains a subgroup in $\mathcal{O}$, then similar arguments imply that $I(J)(G/K) \subseteq M(G/K)$. If this were a strict inclusion, then $M$ would not be closed under restriction, contradicting the fact that it is an ideal. 
\end{itemize}
\end{proof}

\begin{proof}[{Proof of \cref{thm:pi0}}]
    The map from the left to the right is given by \cref{prop:lefttoright} and is clearly injective. The fact that this is surjective is \cref{prop:itssurjective} hence the result.
\end{proof}

\newpage

\part{Examples and further computations}\label{part:examples}



In \cref{part:theory} we have produced many results regarding the computation of Nakaoka spectra of bi-incomplete spectra. In this part we will show that these results are actually usable by computing many examples. This will also highlight some of the phenomena that have been alluded to in the theory section such as sensitivity to the saturated hull.

In \cref{sec:exfp} we will focus on the fixed point Tambara functor $\underline{\Z}$ for several groups, while in \cref{sec:exburn} we focus on the Burnside Tambara functor. Finally, in \cref{sec:vista} we will describe how this style of computation allows us to start thinking about ``$G$-equivariant tensor-triangular geometry'' with a computation of a Balmer--Nakaoka spectrum of the category $\mathsf{Sp}_{C_2}^\omega$ of genuine $C_2$-equivariant finite spectra.

\begin{convention}
From now on, we write $\res^{m}_{\ell} := \res^{C_{m}}_{C_{\ell}}$, $m, \ell \geq 1$, for the restriction maps in a bi-incomplete Tambara functor over any cyclic group. We adopt similar abbreviations for transfers and norms in this section, and for induction, coinduction, and Hill--Hopkins--Ravenel norms in the next section. 
\end{convention}

\section{Fixed point Tambara functors}\label{sec:exfp}

Recall that for a fixed group $G$, $\underline{\Z}$ is the fixed point Tambara functor of $\Z$ equipped with the trivial $G$-action. Using \cref{prop:forget} we obtain a bi-incomplete Tambara functor $i^*(\underline{\Z})$ for each compatible pair $(\cO_m,\cO_a)$. In this section we will compute the spectra of these bi-incomplete Tambara functors for $C_p$ and $C_{p^2}$ in totality.


Being a fixed point  Tambara functor it is multiplicatively cohomological (c.f., \cref{ex:iscohmult}). This allows us to use the full power of the results in \cref{part:examples} and we will see that the computations follow rather quickly.

\subsection{$G = C_p$}

The Lewis diagram for the Tambara functor is given by
\[
\underline{\Z} = \begin{tikzcd}[row sep=huge, column sep=large]
	{\Z} \\
	{\Z}
	\arrow["{k \mapsto k}"{description}, from=1-1, to=2-1]
	\arrow["{k \mapsto pk}", shift left=4, curve={height=-6pt}, from=2-1, to=1-1]
	\arrow["{k \mapsto k^p}"', shift right=4, curve={height=6pt}, from=2-1, to=1-1]
\end{tikzcd}
\]
where the Weyl action at both levels is trivial.

When $G=C_p$ we only have two transfer systems, namely:
\begin{itemize}
    \item $\cO_\mathrm{triv}$, and
    \item $\cO_\mathrm{comp} = \{ (e,C_p) \}$.
\end{itemize}

These fit into precisely three compatible pairs:
\begin{itemize}
    \item $(\cO_\mathrm{triv}, \cO_\mathrm{triv})$, i.e., coefficient systems;
    \item $(\cO_\mathrm{triv}, \cO_\mathrm{comp})$, i.e., Green functors;
    \item $(\cO_\mathrm{comp}, \cO_\mathrm{comp})$, i.e., Tambara functors;
\end{itemize}

We begin with the coefficient system primes, appealing to \cref{ex:coefficientsystem} which tells us that there is a homeomorphism
$$\Spec^{(\cO_\mathrm{triv},\cO_\mathrm{triv})}(\underline{\Z}) \cong \Spec(\Z) \coprod \Spec(\Z) .$$
Using the construction of \cref{def:IJ} we obtain two families of prime ideals
\[
\mathcal{A}_q = \begin{bmatrix}
    \langle q \rangle \\ \langle q \rangle
\end{bmatrix} \qquad \text{ and } \qquad \mathcal{B}_q = \begin{bmatrix} \langle q \rangle  \\\langle  1 \rangle \end{bmatrix}
\]
where $q$ ranges over all primes and 0. Clearly there is an inclusion $\mathcal{A}_q \subseteq \mathcal{B}_q$, and this fully determines the topology. Thus we have
\[
\Spec^{(\cO_\mathrm{triv},\cO_\mathrm{triv})}(\underline{\Z}) \cong \{\mathcal{A}_q, \mathcal{B}_q \}_{q \in \mathbb{P} \cup \{0\}}.
\]

Moving to the case of Green functors, which is changing the additive structure from $\cO_\mathrm{triv}$ to $\cO_\mathrm{comp}$ we use \cref{Prop:ChangeOaPrimes}, which tells us we need only check which ideals $\mathcal{A}_q$ and $\mathcal{B}_q$ remain ideals under the transfer. We see that $\mathcal{A}_q$ always remains an ideal as $qp \in \langle q \rangle$. However, for $q \neq p$, $\mathcal{B}_q$ is no longer an ideal. All in all we have
\[
\Spec^{(\cO_\mathrm{triv},\cO_\mathrm{comp})}(\underline{\Z}) \cong \{\mathcal{A}_q \}_{q \in \mathbb{P} \cup \{0\}} \coprod \mathcal{B}_p.
\]

Finally, we must tackle the case of the full Tambara structure. We know from \cref{lem:bottomlevel} that a Tambara prime $P$ has either $\langle p \rangle$ or $\langle 1 \rangle$ at the bottom level. We can eliminate the option of $\langle 1 \rangle$ as any ideal closed under the norm operation would have $1^p = 1$ at the top level, and thus would not be a proper ideal. It immediately follows by inspection that the only prime ideals in this case are the $\mathcal{A}_q$. That is:
\[
\Spec^{(\cO_\mathrm{comp},\cO_\mathrm{comp})}(\underline{\Z}) \cong \{\mathcal{A}_q \}_{q \in \mathbb{P} \cup \{0\}}.
\]
We summarize these constructions in \cref{tab:cpzunderline}, where one can see the three homeomorphism types of space, and also that in this case adding in structure results in taking a subspace. A red arrow denotes inclusion.

\begin{longtable}{cc|c}
$\mathcal{O}_m$ & $\mathcal{O}_a$ & $\Spec^{(\cO_m,\cO_a)}(\underline{\Z})$  \\ \hline
$\cO_\mathrm{triv}$ & $\cO_\mathrm{triv}$ & $\begin{gathered}\begin{tikzpicture}[scale=0.45]
\draw[gray, thick] plot [smooth] coordinates {(0,0) (2.5,0)  (5.5,0) (8,0)} node [black, right] {$\mathcal{A}_q$};
\draw[gray, thick] plot [smooth] coordinates {(0,1) (2.5,1)  (5.5,1) (8,1)} node [black, right] {$\mathcal{B}_q$};
\draw[black, fill=gray] (0,0) circle (5pt) node{};
\draw[black, fill=gray] (0,1) circle (5pt) node{};
\draw[black!25!red, thick, <-] (4,0.8) -- (4,0.2);
\draw[black!25!red, thick, <-] (2,0.8) -- (2,0.2);
\draw[black!25!red, thick, <-] (6,0.8) -- (6,0.2);
\draw[black!25!red, thick, <-] (0,0.8) -- (0,0.2);
\draw[black!25!red, thick, <-] (8,0.8) -- (8,0.2);
\end{tikzpicture}
\end{gathered}$
\\ \hline
$\cO_\mathrm{triv}$ & $\cO_\mathrm{comp}$& $\begin{gathered}\begin{tikzpicture}[scale=0.45]
\draw[gray, thick] plot [smooth] coordinates {(0,0) (2.5,0)  (5.5,0) (8,0)} node [black, right] {$\mathcal{A}_q$};
\draw[white] plot [smooth] coordinates {(0,1) (2.5,1)  (5.5,1) (4,1)} node [black, right] {$\mathcal{B}_p$};
\draw[black, fill=gray] (0,0) circle (5pt) node{};
\draw[black!25!red, thick, <-] (4,0.8) -- (4,0.2);
\filldraw[black] (4,1) circle (3pt);
\end{tikzpicture}
\end{gathered}$
\\ \hline
$\cO_\mathrm{comp}$ & $\cO_\mathrm{comp}$& $\begin{gathered}\begin{tikzpicture}[scale=0.45]
\draw[gray, thick] plot [smooth] coordinates {(0,0) (2.5,0)  (5.5,0) (8,0)} node [black, right] {$\mathcal{A}_q$};
\draw[black, fill=gray] (0,0) circle (5pt) node{};
\end{tikzpicture}
\end{gathered}$
\\
\caption{The three spectra arising from the bi-incomplete Tambara functors $i^\ast \underline{\Z}$ for $G=C_p$.}\label{tab:cpzunderline}
\end{longtable}

\subsection{$G=C_{p^2}$}\label{subsec:cp2}

We now move to the next non-trivial case, that of $G=C_{p^2}$, where the Lewis diagram is just a repeated version of that for $C_p$ as seen in the previous section. The first thing we must do is discuss the transfer systems and the structure of the compatible pairs. We shall take this opportunity to fix some notation:
\begin{itemize}
    \item $\cO_\mathrm{triv}$.
    \item $\cO_1 = \{(e,C_p)\}$.
    \item $\cO_2 = \{(C_p,C_{p^2})\}$.
    \item $\cO_3 = \{(e,C_p), (e,C_{p^2})\}$.
    \item $\cO_\mathrm{comp}$.
\end{itemize}
We note for later that the transfer system $\cO_3$ is not saturated, and thus it is not self compatible. We have that $\mathrm{Hull}(\cO_3) = \cO_\mathrm{comp}$.

These fit into precisely 12 compatible pairs, which we group by multiplicative structure:
\[
(\cO_\mathrm{triv},\cO_\mathrm{triv}), \, (\cO_\mathrm{triv},\cO_1), \, (\cO_\mathrm{triv},\cO_2), \, (\cO_\mathrm{triv},\cO_3), \, (\cO_\mathrm{triv},\cO_\mathrm{comp}),
\]
\[
(\cO_1,\cO_1), \, (\cO_1, \cO_3), \, (\cO_1,\cO_\mathrm{comp}),
\]
\[
(\cO_2,\cO_2), \, (\cO_2,\cO_\mathrm{comp}),
\]
\[
(\cO_3, \cO_\mathrm{comp}), \, (\cO_\mathrm{comp}, \cO_\mathrm{comp}).
\]
The description of the prime ideals for $\underline{\Z}$ works almost identically to the $C_p$ case. 

Beginning with the coefficient system case, we have three families of prime ideals
\[
\mathcal{A}_q = \begin{bmatrix}
    \langle q \rangle \\ \langle q \rangle \\ \langle q \rangle
\end{bmatrix} \qquad \text{ and } \qquad \mathcal{B}_q = \begin{bmatrix} \langle q \rangle  \\\langle q \rangle\\\langle  1 \rangle \end{bmatrix} \qquad \text{ and } \qquad \mathcal{C}_q = \begin{bmatrix} \langle q \rangle  \\\langle 1 \rangle\\\langle  1 \rangle \end{bmatrix}
\]
Following \cref{ex:coefficientsystem} again we have
\[
\Spec^{(\cO_\mathrm{triv},\cO_\mathrm{triv})}(\underline{\Z}) \cong \{\mathcal{A}_q, \mathcal{B}_q, \mathcal{C}_q \}_{q \in \mathbb{P} \cup \{0\}}
\]
with inclusions $\mathcal{A}_q \subseteq \mathcal{B}_q \subseteq \mathcal{C}_q$.

Again, we can now appeal to \cref{Prop:ChangeOaPrimes}, and add in the additive structure. This is a routine exercise to check which ideals remain closed under the transfers, and as such we leave it as an exercise for the reader. Of particular note is that is that these ideals are closed under the transfers in $\cO_3$ if and only if they are closed under the transfers in $\mathrm{Hull}(\cO_3) = \cO_\mathrm{comp}$. We expect that this is simply a consequence of working with potentially the most basic of Tambara functors, and we do not see this behavior in more complex Tambara functors (c.f., the computations in \cref{sec:exburn}).

Let us now compute the spectrum for the compatible pair $(\cO_1,\cO_1)$, which is self-compatible, and thus we can use \cref{thm:pi0}. We have that $\pi_0\cO_1$ consists of two elements, namely $e \to C_p$ and the singleton $C_{p^2}$. Restricting to the first component gives us a restricted (complete) $C_p$-Tambara functor, the prime ideals of which coincide with those for $\underline{\Z}$ for $C_p$ since the Weyl action is trivial. From the previous section we know that this have just one family of prime ideals given by $\begin{bmatrix}
    \langle q \rangle \\ \langle q \rangle
\end{bmatrix}$. Therefore our $(\cO_1,\cO_1)$-ideal will be
\[
\begin{bmatrix}
\res^{p^2}_1 \langle q \rangle \cap \res^{p^2}_p \langle q \rangle \\
 \langle q \rangle \\ \langle q \rangle
\end{bmatrix} = \begin{bmatrix}
\langle q \rangle  \\
 \langle q \rangle \\ \langle q \rangle
\end{bmatrix} = \mathcal{A}_q .\]
For the other connected component we need to take a prime ideal of $\Z$, that is, some $\langle q \rangle$, and then the entries below are formed of the empty intersection, that is, $\langle 1 \rangle$, and thus we get the ideal $\mathcal{C}_q$ yielding
\[
\Spec^{(\cO_1,\cO_1)}(\underline{\Z}) \cong \{\mathcal{A}_q, \mathcal{C}_q \}_{q \in \mathbb{P} \cup \{0\}}.
\]
We then proceed to use \cref{Prop:ChangeOaPrimes} to add in transfers to compute the prime spectra for $(\cO_1,\cO_3)$ and $(\cO_1,\cO_\mathrm{comp})$.

The only other computation here which warrants discussion is that of the spectrum for $(\cO_3,\cO_\mathrm{comp})$. As $\cO_3$ is not self-compatible we cannot use our usual trick of the path component decomposition. However, as $\underline{\Z}$ is multiplicatively cohomological we can appeal to \cref{thm:insensitive}, which tells us that there is a homeomorphism
\[
\Spec^{(\cO_3,\cO_{\mathrm{comp}})}(\underline{\Z}) \cong \Spec^{(\mathrm{Hull}(\cO_3),\cO_{\mathrm{comp}})} (\underline{\Z}) \cong \Spec^{(\cO_\mathrm{comp},\cO_{\mathrm{comp}})}(\underline{\Z}).
\]
In \cref{fig:zunderline} we explicitly display the primes of all spectra for the 12 compatible pairs for $\underline{\Z}$, mirroring the schematic presented in \cref{fig:strat} in the introduction.

\begin{remark}
    While \cref{fig:zunderline} displays equality of spaces, it does not highlight that there are homeomorphisms appearing which are not equalities on the primes. For example there is a homeomorphism 
    \[
\Spec^{(\cO_1,\cO_1)}(\underline{\Z})  \cong \Spec^{(\cO_2,\cO_2)}(\underline{\Z})
    \]
    even though the primes appearing differ. Among the 12 compatible pairs, we only see 7 homeomorphism types of space among the spectra. Due to the combinatorial nature of the calculations for $\underline{\Z}$, we expect that in general $C_{p^n}$ will yield $2^{n+1}-1$ homeomorphism type of spectra for this Tambara functor.
\end{remark}

\begin{landscape}
\begin{figure}
    \centering
\begin{tikzpicture}[yscale =1.2, xscale = 1.5]
    \draw[gray,ultra thick,->] (1,1) -- (11,1);
    \draw[gray,ultra thick,->] (1,1) -- (1,11);
    \draw[gray,thin,dashed] (1,1) -- (11,11);
    \node at (11.5,1) {$\mathcal{O}_{{a}}$};
    \node at (1,11.5) {$\mathcal{O}_{{m}}$};
    \draw[gray,ultra thick] (2,1.2) -- (2,0.8);
    \draw[gray,ultra thick] (4,1.2) -- (4,0.8);
    \draw[gray,ultra thick] (6,1.2) -- (6,0.8);
    \draw[gray,ultra thick] (8,1.2) -- (8,0.8);
    \draw[gray,ultra thick] (10,1.2) -- (10,0.8);
    \node at (2,0.5) {$\mathcal{O}_{\mathrm{triv}}$};
    \node at (4,0.5) {$\mathcal{O}_{\mathrm{1}}$};
    \node at (6,0.5) {$\mathcal{O}_{\mathrm{2}}$};
    \node at (8,0.5) {$\mathcal{O}_{\mathrm{3}}$};
    \node at (10,0.5) {$\mathcal{O}_{\mathrm{comp}}$};
    \begin{scope}[rotate around={90:(1,1)}]
    \draw[gray,ultra thick] (2,1.2) -- (2,0.8);
    \draw[gray,ultra thick] (4,1.2) -- (4,0.8);
    \draw[gray,ultra thick] (6,1.2) -- (6,0.8);
    \draw[gray,ultra thick] (8,1.2) -- (8,0.8);
    \draw[gray,ultra thick] (10,1.2) -- (10,0.8);
    \end{scope}
    \node at (0.25,2) {$\mathcal{O}_{\mathrm{triv}}$};
    \node at (0.25,4) {$\mathcal{O}_{\mathrm{1}}$};
    \node at (0.25,6) {$\mathcal{O}_{\mathrm{2}}$};
    \node at (0.25,8) {$\mathcal{O}_{\mathrm{3}}$};
    \node at (0.15,10) {$\mathcal{O}_{\mathrm{comp}}$};

    \draw[nice-blue,ultra thick] (8,2.1) -- (10,2.1);
    \draw[nice-blue, ultra thick, snake it] (8,2) -- (10,2);
    \draw[nice-blue, ultra thick, snake it] (8,4) -- (10,4);
    \draw[nice-blue,ultra thick] (8,4.1) -- (10,4.1);
    \draw[nice-blue,ultra thick] (8,1.9) -- (10,1.9);
    \draw[nice-blue,ultra thick] (8,3.9) -- (10,3.9);
    
    \draw[nice-red, ultra thick, snake it] (10,8.276) -- (10,9.8);
    \node[fill=white,draw=nice-green, ultra thick] at (2,2) {$\begin{bmatrix}q \\ 1 \\ 1 \end{bmatrix}$,$\begin{bmatrix}q \\ q \\ 1 \end{bmatrix}$,$\begin{bmatrix}q \\ q \\ q \end{bmatrix}$};
    \node[fill=white,draw=nice-green, ultra thick] at (4,4) {$\begin{bmatrix}q \\ 1 \\ 1 \end{bmatrix}$,$\begin{bmatrix}q \\ q \\ q \end{bmatrix}$};
    \node[fill=white,draw=nice-green, ultra thick] at (6,6) {$\begin{bmatrix}q \\ q \\ 1 \end{bmatrix}$,$\begin{bmatrix}q \\ q \\ q \end{bmatrix}$};
    \node[fill=white,draw=nice-green, ultra thick] at (10,10) {$\begin{bmatrix}q \\ q \\ q \end{bmatrix}$};
    \node[fill=white,draw=nice-blue, ultra thick,rounded corners=0.5cm] at (4,2) {$\begin{bmatrix}q \\ 1 \\ 1 \end{bmatrix}$,$\begin{bmatrix}p \\ p \\ 1 \end{bmatrix}$,$\begin{bmatrix}q \\ q \\ q \end{bmatrix}$};
    \node[fill=white,draw=nice-blue, ultra thick,rounded corners=0.5cm] at (6,2) {$\begin{bmatrix}p \\ 1 \\ 1 \end{bmatrix}$,$\begin{bmatrix}q \\ q \\ 1 \end{bmatrix}$,$\begin{bmatrix}q \\ q \\ q \end{bmatrix}$};
    \node[fill=white,draw=nice-blue, ultra thick,rounded corners=0.5cm] at (8,2) {$\begin{bmatrix}p \\ 1 \\ 1 \end{bmatrix}$,$\begin{bmatrix}p \\ p \\ 1 \end{bmatrix}$,$\begin{bmatrix}q \\ q \\ q \end{bmatrix}$};
    \node[fill=white,draw=nice-blue, ultra thick,rounded corners=0.5cm] at (10,2) {$\begin{bmatrix}p \\ 1 \\ 1 \end{bmatrix}$,$\begin{bmatrix}p \\ p \\ 1 \end{bmatrix}$,$\begin{bmatrix}q \\ q \\ q \end{bmatrix}$};
    \node[fill=white,draw=nice-blue, ultra thick,rounded corners=0.5cm] at (8,4) {$\begin{bmatrix}p \\ 1 \\ 1 \end{bmatrix}$,$\begin{bmatrix}q \\ q \\ q \end{bmatrix}$};
    \node[fill=white,draw=nice-blue, ultra thick,rounded corners=0.5cm] at (10,4) {$\begin{bmatrix}p \\ 1 \\ 1 \end{bmatrix}$,$\begin{bmatrix}q \\ q \\ q \end{bmatrix}$};
    \node[fill=white,draw=nice-blue, ultra thick,rounded corners=0.5cm] at (10,6) {$\begin{bmatrix}p \\ p \\ 1 \end{bmatrix}$,$\begin{bmatrix}q \\ q \\ q \end{bmatrix}$};
    \node[fill=white,fill=white,draw=nice-red, ultra thick] at (10,8) {$\begin{bmatrix}q \\ q \\ q \end{bmatrix}$};
\end{tikzpicture}
\caption{The prime ideals for $\underline{\Z}$ for all possible compatible pairs where $G = C_{p^2}$. The spectra on the diagonal are computed using the connected component decomposition of \cref{thm:pi0}. Moving to the right is adding in transfers and uses \cref{Prop:ChangeOaPrimes}. The red vertical equivalence comes from \cref{thm:insensitive}. Horizontal squiggly lines are representative of equality. }\label{fig:zunderline}
\end{figure}
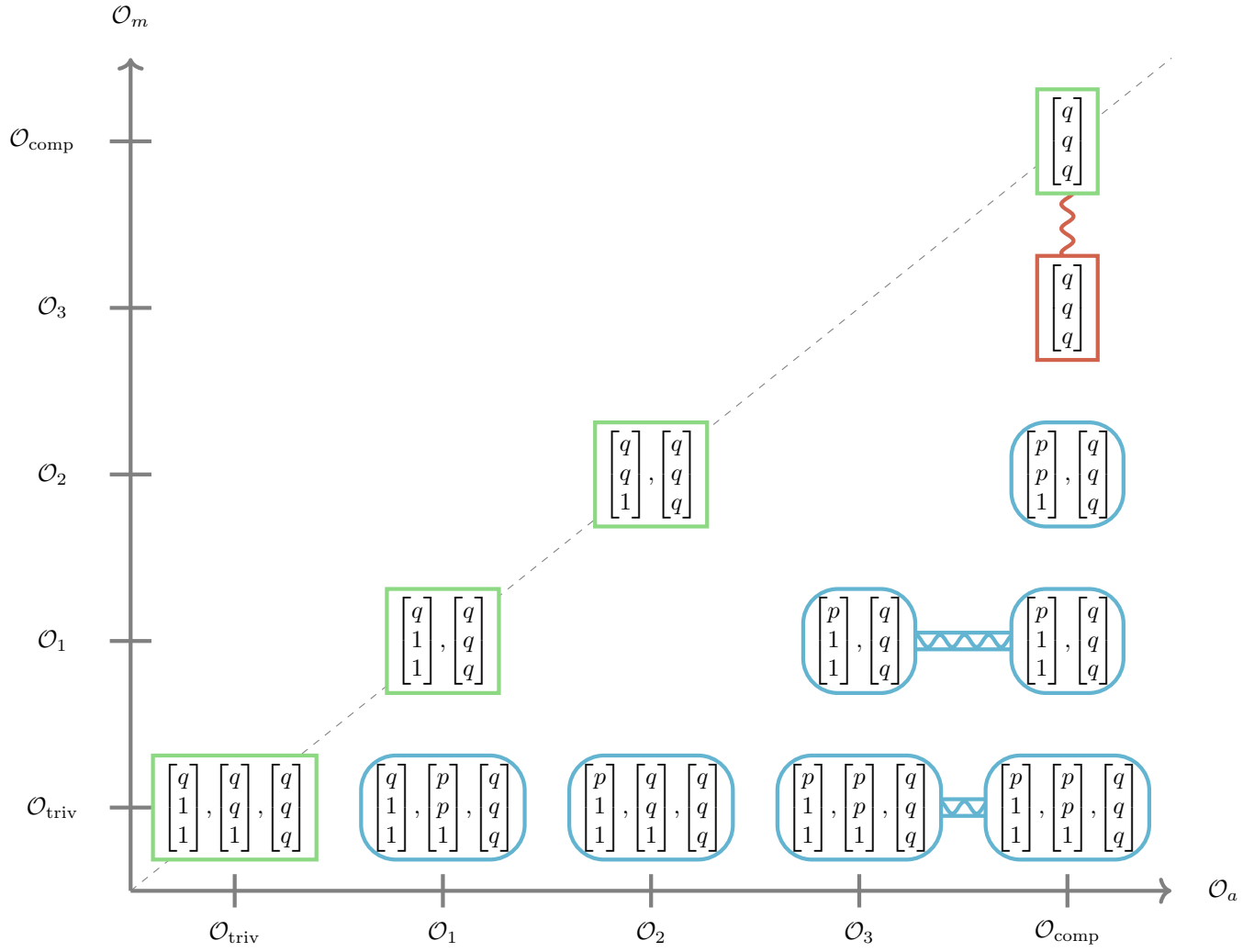
\end{landscape}

\section{Burnside Tambara functors}\label{sec:exburn}

Now that we have seen the techniques in action for a very simple Tambara functor, we move our attention to perhaps the most important Tambara functor, that is, the Burnside Tambara functor $\underline{A}_G$. Unlike the case of $\underline{\Z}$, the Burnside Tambara functor is not multiplicatively cohomological and thus we cannot appeal to \cref{thm:insensitive}. Therefore we shall see in \cref{subsec:cp2burnside} how to compute these exceptional cases.

\subsection{$G = C_p$}\label{subsec:cpburnside}

The Lewis diagram for $\underline{A}_{C_p}$ is given as
\[\begin{tikzcd}[row sep=huge, column sep=large]
	{\Z[t]/\langle t^2-pt \rangle} \\
	\Z
	\arrow["{t \mapsto p}"{description}, from=1-1, to=2-1]
	\arrow["{k \mapsto tk}", shift left=4, curve={height=-6pt}, from=2-1, to=1-1]
	\arrow["{k \mapsto k + \frac{k^p-k}{p}t}"', shift right=4, curve={height=6pt}, from=2-1, to=1-1]
\end{tikzcd}\]
where the Weyl action at each level is trivial. 

To begin we record the levelwise prime ideals which can be obtained using Dress' construction \cite{dress_notes}.

\begin{lemma}
    Let $p$ be prime. Then as sets:
\begin{align*}
    \Spec(\Z) &= \left\{\langle q \rangle\right\}_{q \in \mathbb{P} \cup \{0\}} \\[10pt]
     \Spec(\Z[t]/\langle t^2-pt \rangle) &= \dfrac{\{\langle q,t \rangle, \langle q,t-p \rangle\}_{q \in \mathbb{P} \cup \{0\}}}{\langle p,t \rangle = \langle p,t-p \rangle}
\end{align*}
\end{lemma}

From \cref{ex:coefficientsystem} we compute the coefficient system primes. We once again give names to the families of primes for ease of notation:
\[
\mathcal{A}_q = \begin{bmatrix} \langle q ,t \rangle \\ \langle 1 \rangle \end{bmatrix} \qquad \text{ and } \qquad \mathcal{B}_q^1 = \begin{bmatrix} \langle q ,t-p \rangle \\ \langle 1 \rangle \end{bmatrix} \qquad \text{ and } \qquad \mathcal{B}_q^2 = \begin{bmatrix} \langle q ,t-p \rangle \\ \langle q \rangle \end{bmatrix}
\]
We clearly have inclusions $\mathcal{B}_q^2 \subseteq \mathcal{B}_q^1$, and an identification $\mathcal{A}_p = \mathcal{B}_p^1$.

\[
\Spec^{(\cO_\mathrm{triv},\cO_\mathrm{triv})}(\underline{A}_{C_p}) \cong \dfrac{\{\mathcal{A}_q, \mathcal{B}_q^1, \mathcal{B}_q^2 \}_{q \in \mathbb{P} \cup \{0\}}}{\mathcal{A}_p = \mathcal{B}_p^1}.
\]

We now move to the Green functor primes. Again, we need only check which primes remain ideals under the addition of the transfer $\tr_{1}^{p}(-)$. We see that $\tr_1^p(1) = t \not\in \langle t-p, q \rangle$, and as such $\mathcal{B}_q^1$ fails to be an ideal when $q \neq p$. In particular:
\[
\Spec^{(\cO_\mathrm{triv},\cO_\mathrm{comp})}(\underline{A}_{C_p}) \cong {\{\mathcal{A}_q, \mathcal{B}_q^2 \}_{q \in \mathbb{P} \cup \{0\}}}.
\]
One can check that this agrees with the computation that can be found in \cite{lewis:1980} (c.f., \cref{subsec:lewis}).

Finally, the spectrum for the complete Tambara Tambara functor has been computed by Nakaoka in \cite{nakaoka_spectrum}. For this we need to introduce one further family of primes:
\[
\mathcal{C}_q = \begin{bmatrix}
    \langle q \rangle \\ \langle q \rangle
\end{bmatrix}
\]
Note that $\langle q \rangle$ is not a prime ideal of $\Z[t]/(t^2-tp)$, but it is a radical ideal. We have inclusions $\mathcal{C}_q \subseteq \mathcal{B}_q^2$, and an identification $\mathcal{B}_p^2 = \mathcal{C}_p$. This leads to
\[
\Spec^{(\cO_\mathrm{comp},\cO_\mathrm{comp})}(\underline{A}_{C_p}) \cong \dfrac{\{\mathcal{B}_q^2, \mathcal{C}_q \}_{q \in \mathbb{P} \cup \{0\}}}{\mathcal{C}_p = \mathcal{B}_p^2}.
\]

We display the homeomorphism types of these spectra in \cref{tab:cpBurnside}, once again noting that a red arrow  indicates inclusion.

\begin{longtable}{cc|c}
$\mathcal{O}_m$ & $\mathcal{O}_a$ & $\Spec^{(\cO_m,\cO_a)}(\underline{A}_{C_p})$  \\ \hline
$\cO_\mathrm{triv}$ & $\cO_\mathrm{triv}$ & $\begin{gathered}
\begin{tikzpicture}[scale=0.45]
\draw[gray, thick] plot [smooth] coordinates {(0,-1) (8,-1)} node [black, right] {$\mathcal{B}_q^2$};
\draw[gray, thick] plot [smooth] coordinates {(0,1) (2.5,1) (4,0.5) (5.5,1) (8,1)} node [black, right] {$\mathcal{A}_q$};
\draw[gray, thick] plot [smooth] coordinates {(0,0) (2.5,0) (4,0.5) (5.5,0) (8,0)} node [black, right] {$\mathcal{B}_q^1$};
\draw[black, fill=gray] (0,0) circle (5pt) node{};
\draw[black, fill=gray] (0,1) circle (5pt) node{};
\draw[black, fill=gray] (0,-1) circle (5pt) node{};
\filldraw[black] (4,0.5) circle (3pt);
\filldraw[black] (4,-1) circle (3pt);
\draw[black!25!red, thick, ->] (4,-0.8) -- (4,0.3);
\begin{scope}[yshift=-30]
\draw[black!25!red, thick, <-] (2,0.8) -- (2,0.2);
\draw[black!25!red, thick, <-] (6,0.8) -- (6,0.2);
\draw[black!25!red, thick, <-] (7,0.8) -- (7,0.2);
\draw[black!25!red, thick, <-] (8,0.8) -- (8,0.2);
\draw[black!25!red, thick, <-] (1,0.8) -- (1,0.2);
\draw[black!25!red, thick, <-] (0,0.7) -- (0,0.3);
\draw[black!25!red, thick, <-] (5,1) -- (5,0.2);
\draw[black!25!red, thick, <-] (3,1) -- (3,0.2);
\end{scope}
\end{tikzpicture}
\end{gathered}$
\\ \hline
$\cO_\mathrm{triv}$ & $\cO_\mathrm{comp}$& $\begin{gathered}
\begin{tikzpicture}[scale=0.45]
\draw[gray, thick] plot [smooth] coordinates {(0,-1) (8,-1)} node [black, right] {$\mathcal{B}_q^2$};
\draw[gray, thick] plot [smooth] coordinates {(0,1) (2.5,1) (4,0.5) (5.5,1) (8,1)} node [black, right] {$\mathcal{A}_q$};
\draw[black, fill=gray] (0,1) circle (5pt) node{};
\draw[black, fill=gray] (0,-1) circle (5pt) node{};
\filldraw[black] (4,0.5) circle (3pt);
\filldraw[black] (4,-1) circle (3pt);
\draw[black!25!red, thick, ->] (4,-0.8) -- (4,0.3);
\end{tikzpicture}
\end{gathered}$
\\ \hline
$\cO_\mathrm{comp}$ & $\cO_\mathrm{comp}$& $\begin{gathered}
\begin{tikzpicture}[scale=0.45]
\draw[gray, thick] plot [smooth] coordinates {(0,0) (2.5,0) (4,0.5) (5.5,0) (8,0)} node [black, right] {$\mathcal{C}_q$};
\draw[gray, thick] plot [smooth] coordinates {(0,1) (2.5,1) (4,0.5) (5.5,1) (8,1)} node [black, right] {$\mathcal{B}_q^2$};
\draw[black, fill=gray] (0,0) circle (5pt) node{};
\draw[black, fill=gray] (0,1) circle (5pt) node{};
\filldraw[black] (4,0.5) circle (3pt);
\draw[black!25!red, thick, ->] (0,0.3) -- (0,0.7);
\draw[black!25!red, thick, ->] (1,0.1) -- (1,0.9);
\draw[black!25!red, thick, ->] (2,0.1) -- (2,0.9);
\draw[black!25!red, thick, ->] (6,0.1) -- (6,0.9);
\draw[black!25!red, thick, ->] (7,0.1) -- (7,0.9);
\draw[black!25!red, thick, ->] (8,0.1) -- (8,0.9);
\draw[black!25!red, thick, ->] (3,0.3) -- (3,0.7);
\draw[black!25!red, thick, ->] (5,0.3) -- (5,0.7);
\end{tikzpicture}
\end{gathered}$
\\
\caption{The three spectra arising from the bi-incomplete Tambara functors $i^\ast \underline{A}_{C_p}$.}\label{tab:cpBurnside}
\end{longtable}


\subsection{$G=C_{p^2}$}\label{subsec:cp2burnside}

We now move to our most complicated and comprehensive example, that of the Burnside Tambara functor for $C_{p^2}$:
\[\begin{tikzcd}[row sep=huge, column sep=large]
	{\Z[t,u]/\langle u^2-p^2u,t^2-pt,tu-pu \rangle} \\
	{\Z[t]/\langle t^2-pt \rangle} \\
	\Z
	\arrow["{\substack{u \mapsto pt \\ t \mapsto p}}"{description}, from=1-1, to=2-1]
	\arrow["{\substack{t \mapsto u \\ k \mapsto tk}}", shift left=4, curve={height=-6pt}, from=2-1, to=1-1]
	\arrow["{\substack{t \mapsto p^{p-2}u \\ k \mapsto k + \frac{k^p-k}{p}t}}"', shift right=4, curve={height=6pt}, from=2-1, to=1-1]
	\arrow["{t\mapsto p}"{description}, from=2-1, to=3-1]
	\arrow["{k \mapsto tk}", shift left=4, curve={height=-6pt}, from=3-1, to=2-1]
	\arrow["{k \mapsto k + \frac{k^p-k}{p}t}"', shift right=4, curve={height=6pt}, from=3-1, to=2-1]
\end{tikzcd}\]
The Weyl action at each level is trivial. Note that $\res_1^{p^2}(k)=\res_1^p(k)=k$ and $\nm_{p}^{p^2}(a+bt)$ can be computed using the Mazur sum formula:
\[
\nm_p^{p^2}(a+bt) = \nm_p^{p^2}(a) + \nm_p^{p^2}(b)p^{p-2}u + \sum_{\substack{i+j=p \\ i,j>0}} \tr_{p}^{p^2} (a^ib^jt^j).
\]
We once again record the levelwise prime ideals which can be obtained using Dress' construction \cite{dress_notes}.

\begin{lemma}\label{lem:zariskiprimes}
    Let $p$ be prime. Then as sets:
\begin{align*}
    \Spec(\Z) &= \left\{\langle q \rangle\right\}_{q \in \mathbb{P} \cup \{0\}} \\[10pt]
     \Spec(\Z[t]/\langle t^2-pt \rangle) &= \dfrac{\{\langle q,t \rangle, \langle q,t-p \rangle\}_{q \in \mathbb{P} \cup \{0\}}}{\langle p,t \rangle = \langle p,t-p \rangle} \\[10pt]
     \Spec(\Z[t,u]/ \langle t^2-pt,u^2-p^2u,tu-pu \rangle) &= \dfrac{\{\langle q,t,u\rangle,\langle q,t-p,u\rangle, \langle q,t-p,u-p^2 \rangle\}_{q \in \mathbb{P} \cup \{0\}}}{\langle p,t, u \rangle = \langle p,t-p \rangle = \langle p,t-p, u-p^2 \rangle}
\end{align*}
\end{lemma}

In this section we will compute the bi-incomplete spectrum for all 12 compatible pairs for $C_{p^2}$, which are listed in \cref{subsec:cp2}, using the results of this paper, as well as leveraging results from \cite{4DS}. A summary of the results of this section are presented in \cref{tab:1} and \cref{tab:ofspectra}.

We will group these calculations via the choice of $\cO_m$ which makes sense bearing in mind \cref{Prop:ChangeOaPrimes}. 

\subsection*{$\cO_m = \cO_{\mathrm{triv}}$}

There are 5 compatible pairs with $\cO_m = \cO_{\mathrm{triv}}$. We begin with computing the coefficient system primes.

\begin{lemma}\label{lem:coefficientcp2}
The coefficient system primes of $\underline{A}_{C_{p^2}}$ are given by the following families as $q \in \mathbb{P} \cup \{0\}$:
\[
\mathcal{A}_q = \begin{bmatrix}\langle q,t,u \rangle \\ \langle 1 \rangle \\ \langle 1 \rangle\end{bmatrix}
\]
\vspace{2mm}
\[ 
\mathcal{B}^1_q = \begin{bmatrix}\langle q,t-p,u \rangle \\ \langle 1 \rangle \\ \langle 1 \rangle\end{bmatrix}, \, 
\mathcal{B}^2_q = \begin{bmatrix}\langle q,t-p,u \rangle \\ \langle q,t \rangle \\ \langle 1 \rangle\end{bmatrix}
\]
\vspace{2mm}
\[
\mathcal{C}_q^1 = \begin{bmatrix}\langle q,t-p,u-p^2 \rangle \\ \langle 1 \rangle \\ \langle 1 \rangle\end{bmatrix}, \, 
\mathcal{C}^2_q = \begin{bmatrix}\langle q,t-p,u-p^2 \rangle \\ \langle q,t-p \rangle \\ \langle 1 \rangle\end{bmatrix}, \, 
\mathcal{C}^3_q = \begin{bmatrix}\langle q,t-p,u-p^2 \rangle \\ \langle q,t-p \rangle \\ \langle q \rangle\end{bmatrix}
\]
\vspace{2mm}
with inclusions 
\[\mathcal{B}^2_q \subseteq \mathcal{B}^1_q \qquad
\mathcal{C}^3_q \subseteq \mathcal{C}^2_q \subseteq \mathcal{C}^1_q\]
and identifications
\[\mathcal{A}_p = \mathcal{B}_p^1=\mathcal{C}_p^1 \qquad
\mathcal{B}^2_p=\mathcal{C}^2_p.\]
As such 
\[
\Spec^{(\cO_\mathrm{triv} , \cO_\mathrm{triv})}(\underline{A}_{C_{p^2}}) \cong \dfrac{\{\mathcal{A}_q,\mathcal{B}_q^1,\mathcal{B}_q^2,\mathcal{C}_q^1,\mathcal{C}_q^2,\mathcal{C}_q^3\}_{q \in \mathbb{P} \cup \{0\}}}{\mathcal{A}_p = \mathcal{B}_p^1=\mathcal{C}_p^1,\,\mathcal{B}^2_p=\mathcal{C}^2_p}
\]
\end{lemma}

\begin{proof}
This is a simple computation which requires picking levelwise prime ideals from \cref{lem:zariskiprimes} which are closed under restriction (\cref{ex:coefficientsystem}).
\end{proof}

Next we adjust the additive structure. In each case, \cref{Prop:ChangeOaPrimes} implies we need only check which of the ideals in \cref{lem:coefficientcp2} are closed under the included transfers. 

\begin{lemma}\label{lem:calc1}
    Let $\cO_a = \cO_1 = \{(e,C_p)\}$. Then
     \[
\Spec^{(\cO_\mathrm{triv} , \cO_1)}(\underline{A}_{C_{p^2}}) \cong \dfrac{\{\mathcal{A}_q,\mathcal{B}_q^1,\mathcal{B}_q^2,\mathcal{C}_q^1,\mathcal{C}_q^3\}_{q \in \mathbb{P} \cup \{0\}}}{\mathcal{A}_p = \mathcal{B}_p^1=\mathcal{C}_p^1}
\]
\end{lemma}
\begin{proof}
    All ideals aside from $\mathcal{C}_q^2$ remain ideals. Indeed, $\tr_1^p(1) = t \not\in \langle q,t-p \rangle$ when $q\neq p$, and the $q=p$ case is covered by $\mathcal{B}^2_p = \mathcal{C}^2_p$.
\end{proof}

\begin{lemma}\label{lem:calc2}
    Let $\cO_a = \cO_2 = \{(C_p,C_{p^2})\}$. Then
    \[
\Spec^{(\cO_\mathrm{triv} , \cO_2)}(\underline{A}_{C_{p^2}}) \cong \dfrac{\{\mathcal{A}_q,\mathcal{B}_q^2,\mathcal{C}_q^2,\mathcal{C}_q^3\}_{q \in \mathbb{P} \cup \{0\}}}{\mathcal{B}_p^2=\mathcal{C}_p^2}
\]
\end{lemma}

\begin{proof}
     When $q \neq p$, $\mathcal{B}^1_q$ and $\mathcal{C}^1_q$ are no longer ideals as $\tr_p^{p^2}(1)=t \not\in \langle q,t-p,u \rangle$. The $q=p$ case is covered by $\mathcal{A}_p = \mathcal{B}^1_q = \mathcal{C}^1_q$. 
     
     It is not obvious from our choices of generators that $\mathcal{C}^2_q$ and $\mathcal{C}^3_q$ are ideals with respect to these transfers so we produce the argument here. It requires checking that $\tr_p^{p^2}(t-p) = u-pt \in \langle q,t-p,u-p^2 \rangle$. As ideals are absorbing under multiplication it follows that $p (t-p)$ is in the ideal, and then we can use the underlying abelian group structure to form $(u-p^2)-p(t-p) = u-pt$ as required. 
\end{proof}

\begin{lemma}\label{lem:calc3}
    Let $\cO_a = \cO_3 = \{(e,C_p),(e,C_{p^2})\}$. Then
    \[
\Spec^{(\cO_\mathrm{triv} , \cO_3)}(\underline{A}_{C_{p^2}}) \cong \dfrac{\{\mathcal{A}_q,\mathcal{B}_q^1,\mathcal{B}_q^2,\mathcal{C}_q^3\}_{q \in \mathbb{P} \cup \{0\}}}{\mathcal{A}_p = \mathcal{B}_p^1}
\]
\end{lemma}

\begin{proof}
    We begin with the primes from \cref{lem:calc1} and see which ones are moreover closed under the additional transfer. For this we observe that $\tr_1^{p^2}(k)=\tr_p^{p^2}\tr_1^p(k) = ku$. In particular we see that $\mathcal{C}^1_p$ is no longer an ideal as $\tr_1^{p^2}(1)=u \not\in \langle q,t-p,u-p^2\rangle$ whenever $q \neq p$.
\end{proof}

\begin{lemma}\label{lem:calc4}
    Let $\cO_a = \cO_\mathrm{comp}$. Then
    \[
\Spec^{(\cO_\mathrm{triv} , \cO_\mathrm{comp})}(\underline{A}_{C_{p^2}}) \cong \{\mathcal{A}_q,\mathcal{B}_q^2,\mathcal{C}_q^3\}_{q \in \mathbb{P} \cup \{0\}}
\]
\end{lemma}

\begin{proof}
    We intersect the primes from the $\cO_a = \cO_1$ case (\cref{lem:calc1}) and the $\cO_a = \cO_2$ case (\cref{lem:calc2}).
\end{proof}

\begin{remark}\leavevmode
\begin{itemize}
    \item The computation of $\Spec^{(\cO_\mathrm{triv} , \cO_\mathrm{comp})}(\underline{A}_{C_{p^2}})$ agrees with what can be found in \cite{lewis:1980}.
    \item \cref{lem:calc2} and \cref{lem:calc3} show that the computation of the spectrum is sensitive to the saturated hull of the additive part in general.
    \item While it may seem that the spaces $\Spec^{(\cO_\mathrm{triv} , \cO_2)}(\underline{A}_{C_{p^2}})$ and $\Spec^{(\cO_\mathrm{triv} , \cO_3)}(\underline{A}_{C_{p^2}})$ are homeomorphic, they are not. This can be seen by checking the inclusions between primes and we refer the reader to \cref{tab:ofspectra} for diagrams.

\end{itemize}    
\end{remark}

\subsection*{$\cO_m = \cO_1$}

Using \cref{thm:pi0}, we can compute the $(\cO_1,\cO_1)$ bi-incomplete primes. We remind ourselves from \cref{subsec:cpburnside} that the Tambara primes of $\underline{A}_{C_p}$ are of the form
\[
\begin{bmatrix}
    \langle q \rangle \\\langle q \rangle
\end{bmatrix} \quad \text{and} \quad \begin{bmatrix}
    \langle q , t-p \rangle \\ \langle q \rangle 
\end{bmatrix}
\]
with an identification when $p=q$.

We have two connected components of our self-compatible pair to deal with, the bottom and the top components. For $\mathfrak{p}$ a prime ideal in $A(C_{p^2})$ we obtain an $(\cO_1,\cO_1)$-prime ideal by adding $\langle 1 \rangle$'s below (given by considering the empty intersection). In the notation used previously this gives us the ideals $\mathcal{A}_q, \mathcal{B}_q^1$, and $\mathcal{C}_q^1$. For the other connected component we need to pick one of the above prime ideals and compute the preimage of restriction for the entry at the top level:
\begin{itemize}
    \item Let $\mathcal{J} = \begin{bmatrix}
    \langle q \rangle \\\langle q \rangle
\end{bmatrix}$. Then we need to compute $(\res_1^{p^2})^{-1}(\langle q \rangle) \cap (\res_p^{p^2})^{-1}(\langle q \rangle)$. The value of this intersection is given by $\langle q, t-p \rangle$ whenever $q \neq p$ and $\langle q,t \rangle$ if $q=p$. This provides an infinite family of primes:
\[
\mathcal{D}_q = \begin{bmatrix}
        \langle q,t-p\rangle \\ \langle q \rangle \\ \langle q \rangle
    \end{bmatrix}
\]

\begin{remark}\label{rem:nonnontriv}
We return to \cref{rem:onewonders} which discussed the levelwise Weyl-primality of prime ideas at levels where there are no nontrivial norms into that level. Consider $\mathcal{D}_q$ for $q=0$. Then we have $\mathcal{D}_0(C_{p^2}/e) = \langle t-p \rangle$ which is not Weyl-prime even though the top level has no non-trivial norms entering it. The issue is that there is a norm $e \to C_p$ and this adds to the potential generalized products that we see.
\end{remark}

\item Next we let $\mathcal{J} = \begin{bmatrix}
    \langle q, t-p \rangle \\\langle q \rangle
    \end{bmatrix}$ and we once again want to compute the inverse image of restriction. In this case, we obtain the infinite family $\mathcal{C}_q^3$ which we have already seen.
\end{itemize}
We have now computed all of the ideals, and it remains to discuss the topology which boils down to giving any identifications and inclusions. As before we have $\mathcal{A}_p = \mathcal{B}_p^1=\mathcal{C}_p^1$. There is an inclusion $\mathcal{D}_q \subseteq \mathcal{C}_q^3$. Finally there are inclusions $\mathcal{D}_p, \mathcal{C}_p^3 \subseteq \mathcal{A}_p$. All in all we have:

\begin{lemma}\label{lem:calc01}
 Let $\cO_a = \cO_1$. Then
\[
\Spec^{(\cO_1 , \cO_1)}(\underline{A}_{C_{p^2}}) \cong \dfrac{\{\mathcal{A}_q,\mathcal{B}_q^1,\mathcal{C}_q^1, \mathcal{C}_q^3, \mathcal{D}_q\}_{q \in \mathbb{P} \cup \{0\}}}{\mathcal{A}_p = \mathcal{B}_p^1= \mathcal{C}_p^1}
\]    
\end{lemma}

We now include more additive structure. In this case, we have options of $\cO_a = \cO_3$ and $\cO_a = \cO_\mathrm{comp}$, which we address separately. As in the previous section we need only see which primes remain ideals closed under the new additive transfers.

\begin{lemma}\label{lem:calc5}
    Let $\cO_a = \cO_3$. Then
\[
\Spec^{(\cO_1 , \cO_3)}(\underline{A}_{C_{p^2}}) \cong \dfrac{\{\mathcal{A}_q,\mathcal{B}_q^1, \mathcal{C}^3_q, \mathcal{D}_q\}_{q \in \mathbb{P} \cup \{0\}}}{\mathcal{A}_p = \mathcal{B}_p^1}
\]
\end{lemma}

\begin{proof}
    As in \cref{lem:calc3} we have that $\mathcal{C}_q^1$ ceases to be an ideal as it is no longer closed under the transfer $(e,C_{p^2})$. The ideal $\mathcal{D}_q$ is still an ideal under this transfer as $\tr_1^{{p^2}}(q) = qu$ and ideals are absorbing under multiplication.
\end{proof}

\begin{lemma}
    Let $\cO_a = \cO_\mathrm{comp}$. Then
\[
\Spec^{(\cO_1 , \cO_\mathrm{comp})}(\underline{A}_{C_{p^2}}) \cong {\{\mathcal{A}_q, \mathcal{C}_q^3, \mathcal{D}_q\}_{q \in \mathbb{P} \cup \{0\}}}
\]
\end{lemma}

\begin{proof}
    We need only check the closure of the primes from \cref{lem:calc5} under the additional transfer $(C_p,C_{p^2})$. We know from the Green functor case (\cref{lem:calc4}) that $\mathcal{B}_q^1$ is no longer an ideal. Again, $\mathcal{C}_q^3$ and $\mathcal{D}_q$ remain ideals.
\end{proof}

\subsection*{$\cO_m = \cO_2$}

For the compatible pairs with $\cO_m = \cO_2$, which is self compatible, we can use \cref{thm:pi0}, which tells us we need to consider the two connected components in turn.

Firstly, we can consider the connected component which is the singleton at the bottom. We pick a prime ideal of $\Z$, say $\langle q \rangle$ and the rest of the ideal is determined by preimages of restrictions. This is exactly the ideal $\mathcal{C}_q^3$.

Next, we must compute the Nakaoka spectrum of the (restricted) $C_p$ Tambara functor
\begin{equation}\label{eq:resBurnside}
i^\ast(\underline{A}_{C_{p^2}}) = \begin{tikzcd}[row sep=huge, column sep=large]
	{\Z[t,u]/\langle u^2-p^2u,t^2-pt,tu-pu \rangle} \\
	{\Z[t]/\langle t^2-pt \rangle}
	\arrow["{\substack{u \mapsto pt \\ t \mapsto p}}"{description}, from=1-1, to=2-1]
	\arrow["{\substack{t \mapsto u \\ k \mapsto tk}}", shift left=4, curve={height=-6pt}, from=2-1, to=1-1]
	\arrow["{\substack{t \mapsto p^{p-2}u \\ k \mapsto k + \frac{k^p-k}{p}t}}"', shift right=4, curve={height=6pt}, from=2-1, to=1-1]
\end{tikzcd}\end{equation}
where the group action is trivial at both levels. To compute this spectrum we will appeal to the $C_p$ ghost construction of \cite{4DS} which we will now recall.

\begin{definition}\label{def:ghost}
Let $T$ be a $C_p$ Tambara functor. Denote by $\tau$ the image of the transfer $\tr_1^{p}$ which is an ideal of $T(C_p/C_p)$. We then denote by $\Phi^{C_p}T \coloneq T(C_p/C_p)$ the quotient by this ideal, and refer to it as the \emph{geometric fixed points} of $T$.

The \emph{ghost} of $T$ is the Tambara functor
\vspace{-10mm}
\[
\mathghost(T) = \begin{tikzcd}[row sep=huge, column sep=large]
	{T(C_p/e)^{C_p} \times \Phi^{C_p}T}\\
	{T(C_p/e)}
	\arrow["{\res}"{description}, from=1-1, to=2-1]
	\arrow["{\tr}", shift left=4, curve={height=-6pt}, from=2-1, to=1-1]
	\arrow["{\nm}"', shift right=4, curve={height=6pt}, from=2-1, to=1-1]
    \arrow[from=2-1, to=2-1, loop, in=305, out=235, distance=10mm]
\end{tikzcd}
\qquad
\begin{aligned}
   \\ { } \\ \res(x,y) &= x \\ \tr(x) &= \left( \sum_{g \in G} g \cdot x, 0\right) \\ \nm(x) &= \left( \prod_{g \in G} g \cdot x, \nm(x) + \tau \right)
\end{aligned}
\]
The ghost of $T$ comes equipped with a map of Tambara functors, the \emph{ghost map}
\[
\chi_T \colon T \to \mathghost(T)
\]
which is the identity on level $C_p/e$ and given on the $C_p/C_p$ level by
\[
(\chi_T)_{C_p/C_p} = (\res,q) \colon T(C_p/C_p) \to T(C_p/e)^{C_p} \times \Phi^{C_p}T
\]
where $q \colon T(C_p/C_p) \to \Phi^{C_p}T$ is the quotient map.
\end{definition}

By \cite[Proposition 7.22]{4DS} we have a full understanding of the Nakaoka spectrum of $\mathghost(T)$. We will denote by $(\mathfrak{a};\mathfrak{b})$ the ideal of $\mathghost(T)$ defined as
\begin{align*}
    (\mathfrak{a};\mathfrak{b})(C_p/e) &= \mathfrak{a} \\
    (\mathfrak{a};\mathfrak{b})(C_p/C_p) &= \mathfrak{a}^{C_p} \times \mathfrak{b}
\end{align*}
where $\mathfrak{a} \subseteq T(C_p/e)$ is a $C_p$-prime ideal and $\mathfrak{b} \subseteq \Phi^{C_p}T$ is either a prime ideal or the entire ring.

\begin{proposition}[{\cite[Proposition 7.22]{4DS}}]\label{prop:ghostprimes}
    Every prime ideal of $\mathghost(T)$ is one of the following two forms:
    \begin{enumerate}
        \item $(\mathfrak{a};\Phi^{C_p}T)$ for some $C_p$-prime ideal $\mathfrak{a}$ of $T(C_p/e)$.
        \item $((\nm)^{-1}(\mathfrak{b}); \mathfrak{b})$ where $\mathfrak{b}$ is a prime ideal of $\Phi^{C_p}T$ and $\nm \colon T(C_p/e) \to \Phi^{C_p}T$ is the ring map induced by the norm.
    \end{enumerate}
\end{proposition}
In particular there is a bijection
\[
\Spec^{(\cO_{\mathrm{comp}},\cO_{\mathrm{comp}})}(\mathghost(T)) \cong \Spec(T(C_p/e)^{C_p}) \coprod \Spec(\Phi^{C_p}T)
\]

We now apply this machinery to the $C_p$ Tambara functor $i^\ast(\underline{A}_{C_{p^2}})$ from \eqref{eq:resBurnside}. The image of the transfer contains $u$ and $t$, and thus taking the quotient by this kills the polynomial generators. We conclude that
\begin{equation}\label{eq:resBurnside2}
\mathghost(i^\ast(\underline{A}_{C_{p^2}})) = \begin{tikzcd}[row sep=huge, column sep=large]
	{\Z[t]/\langle t^2-pt \rangle \times \Z} \\
	{\Z[t]/\langle t^2-pt \rangle}
	\arrow[from=1-1, to=2-1]
	\arrow[shift left=4, curve={height=-6pt}, from=2-1, to=1-1]
	\arrow[shift right=4, curve={height=6pt}, from=2-1, to=1-1]
\end{tikzcd}\end{equation}
where the description of the restriction, transfer, and norm follows from \cref{def:ghost}. 

Using \cref{prop:ghostprimes} we obtain a description of the Nakaoka spectrum of the ghost. We have an abstract bijection
\[
\Spec^{(\cO_{\mathrm{comp}},\cO_{\mathrm{comp}})}(\mathghost(i^\ast(\underline{A}_{C_{p^2}}))) \cong \Spec(\mathbb{Z}[t]/\langle t^2-pt\rangle) \coprod \Spec(\mathbb{Z})
\]
but it will be worthwhile having explicit descriptions of these prime ideals for our continued computations. Recall from \cref{lem:zariskiprimes} that the primes of $\mathbb{Z}[t]/(t^2-pt)$ come in two forms:
\begin{enumerate}
    \item $\langle q,t \rangle$
    \item $\langle q,t-p \rangle$
\end{enumerate}
with identification when $q=p$.

\begin{corollary}
    The prime ideals of $\mathghost(i^\ast(\underline{A}_{C_{p^2}}))$ are:
    \[
        \mathcal{J}_1 = \begin{bmatrix}
           \langle q,t \rangle \times \langle 1 \rangle  \\ \langle q,t \rangle
        \end{bmatrix}
        \qquad
        \mathcal{J}_2 = \begin{bmatrix}
           \langle q,t-p \rangle \times \langle 1 \rangle  \\ \langle q,t-p \rangle
        \end{bmatrix}
        \qquad
        \mathcal{K} = \begin{bmatrix}
           \langle q, t \rangle \times \langle q \rangle \\ \langle q, t \rangle
        \end{bmatrix}
        \]
    for $q \in \mathbb{P} \cup \{0\}$.
\end{corollary}

Now that we understand the prime ideals of the ghost we can use the fact that the ghost map $\chi_T$ satisfies going up and lying over \cite[Sec. 5]{4DS}, so applying $\chi^\ast_T$ yields the Nakaoka spectrum of $T$. This will allow us to compute the prime ideals of the Tambara functor that we actually care about.




\begin{corollary}
    The prime ideals of $i^\ast(\underline{A}_{C_{p^2}})$ are:
    \[
        \mathcal{J}_1 = \begin{bmatrix}
           (\res,q)^{-1}(\langle q,t \rangle \times \langle 1 \rangle)  \\ \langle q,t \rangle
        \end{bmatrix} = \begin{bmatrix}
           \langle q,t-p,u \rangle  \\ \langle q,t \rangle
        \end{bmatrix}
        \qquad
        \mathcal{J}_2 = \begin{bmatrix}
           (\res,q)^{-1}(\langle q,t-p \rangle \times \langle 1) \rangle ) \\ \langle q,t-p \rangle
        \end{bmatrix} = \begin{bmatrix}
           \langle q,t-p, u-p^2 \rangle  \\ \langle q,t-p \rangle
        \end{bmatrix}
    \]
    \vspace{2mm}
    \[
        \mathcal{K} = \begin{bmatrix}
           (\res,q)^{-1}(\langle q, t \rangle \times \langle q \rangle) \\ \langle q, t \rangle
        \end{bmatrix}  = \begin{bmatrix}
            \langle q, u \rangle \\ \langle q,t \rangle
        \end{bmatrix}
    \]
    for $q \in \mathbb{P} \cup \{0\}$.  We moreover have $(\mathcal{J}_1)_p = \mathcal{K}_p$
\end{corollary}

We can now complete our computation of the ideals for this connected component by adding the entire ring at the bottom level (remembering that the empty intersection is the entire ring). Doing this for $\mathcal{J}_1$ produces $\mathcal{B}^2_q$ and for $\mathcal{J}_2$ produces $\mathcal{C}^2_q$; for $\mathcal{K}$, it gives rise to one ideal which we have yet to see:
\[
\mathcal{E}_q = \begin{bmatrix}
    \langle q,u \rangle \\ \langle q,t \rangle \\  \langle 1 \rangle
\end{bmatrix}
\]

All in all we have:

\begin{lemma}\label{lem:comp6}
    Let $\cO_a = \cO_2$. Then
\[
\Spec^{(\cO_2 , \cO_2)}(\underline{A}_{C_{p^2}}) \cong \dfrac{\{\mathcal{B}_q^2,\mathcal{C}_q^2,\mathcal{C}_q^3,\mathcal{E}_q\}_{q \in \mathbb{P} \cup \{0\}}}{\mathcal{B}_p^2 = \mathcal{C}_p^2 = \mathcal{E}_p}
\]    
\end{lemma}

Once again, it remains to add in the possible additive norms which in this case the only further choice is $\cO_a = \cO_\mathrm{comp}$. We have that $\mathcal{C}_q^2$ fails to be an ideal, and that $\mathcal{E}_q$ remains an ideal, giving the following.

\begin{lemma}\label{lem:comp7}
    Let $\cO_a = \cO_\mathrm{comp}$. Then
\[
\Spec^{(\cO_2 , \cO_\mathrm{comp})}(\underline{A}_{C_{p^2}}) \cong \frac{\{\mathcal{B}_q^2,\mathcal{C}_q^3,\mathcal{E}_q\}_{q \in \mathbb{P} \cup \{0\}}}{\mathcal{B}^2_p = \mathcal{E}_p}
\]    
\end{lemma}

\subsection*{$\cO_m = \cO_3$}

As we are in a case where we are not multiplicatively cohomological we cannot rely on the saturated hull of the multiplicative part to yield the result, and instead we must compute this explicitly.

We begin by observing that if $\mathcal{P}$ is an $(\mathcal{O}_3,\mathcal{O}_{\mathrm{comp}})$-prime ideal, then its restriction $i^*_{C_p}P$ to an ideal of $\underline{A}_{C_p}$ is a Tambara ideal. Indeed, if $\mathtt{Q}(i^*_{C_p}\mathcal{P},x,y)$ holds for some $x,y \in \underline{A}_{C_{p}}$, then $\mathtt{Q}(\mathcal{P},x,y)$ also holds since the only new generalized products are obtained by norming a generalized product from level $C_{p^2}/e$ to level $C_{p^2}/C_{p^2}$. Therefore $P(C_{p^2}/e) = \langle q \rangle$  and $P(C_{p^2}/C_p) = \langle q \rangle$ or $\langle q, t-p \rangle$, where $q$ is a prime or zero. 

Now, in the $(\mathcal{O}_1,\mathcal{O}_1)$ computation above, we computed the inverse image under restriction of these ideals. As any $(\mathcal{O}_3,\mathcal{O}_{\mathrm{comp}})$-ideal of $\underline{A}_{C_{p^2}}$ is necessarily an $(\mathcal{O}_1,\mathcal{O}_1)$-ideal, we see that any $(\mathcal{O}_3,\mathcal{O}_{\mathrm{comp}})$-prime ideal must be a sub-ideal of  $\mathcal{D}_q = \begin{bmatrix}
        \langle q,t-p\rangle \\ \langle q \rangle \\ \langle q \rangle
    \end{bmatrix}$ or of  $\mathcal{C}_q^3 = \begin{bmatrix}
        \langle q,t-p, u-p^2 \rangle \\ \langle q,t-p \rangle \\ \langle q \rangle
    \end{bmatrix}$. 
    Thus we are left with only determining possibilities for the top entries. We will tackle each of these in turn. 
    
    Let $q \neq p$. For $\mathcal{D}_q$ this gives us a minimal ideal of 
    \[
    \mathcal{D}_q' = \begin{bmatrix}
        \langle q \rangle \\ \langle q \rangle \\ \langle q \rangle 
    \end{bmatrix}
    \]
    while for $\mathcal{C}_q^3$ we get a minimal ideal of 
    \[
    \mathcal{F}_q= \begin{bmatrix}
        \langle q,u-pt \rangle \\ \langle q,t-p \rangle \\ \langle q \rangle
    \end{bmatrix}
    \]
    When $q=p$ then $\mathcal{D}_p' = \mathcal{F}_q = \mathcal{C}_p^3$.
    
    This warrants a little justification, and we will do so for the $\mathcal{D}_q'$ case. Note that we must contain
    \begin{align*}
        \tr_p^{p^2}(q) &= qt\\
        \tr_1^{p^2}(q) &= qu\\
        \nm_1^{p^2} (q) &= q + \dfrac{q^p-q}{p}t + \dfrac{q^{p^2}-q^p}{p^2}u
    \end{align*}
    We note that $q$ divides the second and third factors of $\nm_1^{p^2} (q)$. Therefore we see that as an ideal is necessarily a subgroup this ideal contains $q$, and as such we must contain at least $\langle q \rangle$.  
    
    At this stage we have no more option but to compute the $\mathtt{Q}$ condition for all possible options. 

    We have that 
    \[
        \mathcal{D}_p' = \begin{bmatrix}
            \langle q \rangle \\ \langle q \rangle \\ \langle q \rangle
        \end{bmatrix}
    \]
    is not a prime ideal. Consider the element $x = pt-u \in \underline{A}(C_{p^2}/C_{p^2})$ and $y = p-t \in \underline{A}(C_{p^2}/C_{p^2})$. Clearly neither of these are in the ideal $\langle q \rangle$, but we will show that all generalized products of $x$ and $y$ are zero, and thus in the ideal.

First, we observe that $(pt-u)(p-t) = p^2t-pt^2 - pu + tu = 0$. Next, $\res_1^{p^2}(pt-u)=0=\res_1^{p^2}(p-t)$ and we are done, as the only nontrivial norms we have to form generalized products with are at the bottom layer. Hence this is not a prime ideal. From this we moreover see that the top level must contain either $pt-u$ or $p-t$. In particular this is the insight needed to check that $\mathcal{F}_q$ is indeed an $(\cO_3,\cO_3)$-prime:

\begin{lemma}\label{lem:cp2nonsat}
    Let $\cO_a = \cO_\mathrm{comp}$. Then 
    \[
\Spec^{(\cO_3 , \cO_\mathrm{comp})}(\underline{A}_{C_{p^2}}) \cong  \dfrac{ {\{\mathcal{C}_q^3, \mathcal{D}_q, \mathcal{F}_q\}_{q \in \mathbb{P} \cup \{0\}}}}{\mathcal{C}_p^3 = \mathcal{D}_p = \mathcal{F}_p}
\]
\end{lemma}

\begin{remark}
We used above that the restriction $i^*_{C_p}\mathfrak{p}$ of an $(\mathcal{O}_3,\mathcal{O}_{\mathrm{comp}})$-ideal to $\underline{A}_{C_p}$ is a (complete) Tambara ideal. In general, we do not know if the restriction of a prime ideal to a subgroup is always a (restricted) prime ideal; above, it was important that all generalized products outside of the restriction factored through a common subgroup (the trivial subgroup). 
\end{remark}


\subsection*{$\cO_m = \cO_\mathrm{comp}$}

This final case is simply a computation of the full Nakaoka spectrum of the Burnside Tambara functor for $C_{p^2}$, the result of which we reproduce here from \cite{nakaoka_spectrum}. We introduce one more family of prime ideals:
\[ \mathcal{G}_q  =
\begin{bmatrix}
   \langle q \rangle \\ \langle q \rangle \\ \langle q \rangle
\end{bmatrix}
\]

\begin{lemma}\label{lem:completeCp2}
    Let $\cO_a = \cO_\mathrm{comp}$. Then 
    \[
\Spec^{(\cO_\mathrm{comp} , \cO_\mathrm{comp})}(\underline{A}_{C_{p^2}}) \cong  \dfrac{ {\{\mathcal{C}_q^3, \mathcal{D}_q, \mathcal{G}_q\}_{q \in \mathbb{P} \cup \{0\}}}}{\mathcal{C}_p^3 = \mathcal{D}_p = \mathcal{G}_p}
\]
\end{lemma}

\begin{remark}\label{rem:itissensitive}
    We can now observe through \cref{lem:cp2nonsat} and \cref{lem:completeCp2} that the computation is sensitive to the saturated hull of the multiplicative part. Explicitly, we see that the $(\cO_\mathrm{comp},\cO_\mathrm{comp})$-Tambara functor is a domain (i.e., the $0$ ideal is prime), while the $(\cO_3,\cO_\mathrm{comp})$-Tambara functor is not.
\end{remark}





\begin{table}[H]
\scalebox{0.75}{
\begin{tabular}{ccc}
\textbf{Name}   & $\mathbf{(q \neq p)}$                                                                             & $\mathbf{(q = p)}$ \\ \hline
\rule{0pt}{.4in}\vspace{.1in} $\mathcal{A}_q$   & $\begin{bmatrix}\langle q,t,u \rangle \\ \langle 1 \rangle \\ \langle 1 \rangle\end{bmatrix}$ & $\begin{bmatrix}\langle p,t,u \rangle \\ \langle 1 \rangle \\ \langle 1 \rangle\end{bmatrix}$                          \\ \hline
\rule{0pt}{.4in}\vspace{.1in}$\mathcal{B}^1_q$ & $\begin{bmatrix}\langle q,t-p,u \rangle \\ \langle 1 \rangle \\ \langle 1 \rangle\end{bmatrix}$  & $\begin{bmatrix}\langle p,t,u \rangle \\ \langle 1 \rangle \\ \langle 1 \rangle\end{bmatrix}$ \\ \hline
\rule{0pt}{.4in}\vspace{.1in}$\mathcal{B}^2_q$ & $\begin{bmatrix}\langle q,t-p,u \rangle \\ \langle q,t \rangle \\ \langle 1 \rangle\end{bmatrix}$       & $\begin{bmatrix}\langle p,t,u \rangle \\ \langle p,t \rangle \\ \langle 1 \rangle\end{bmatrix}$ \\ \hline
\rule{0pt}{.4in}\vspace{.1in}$\mathcal{C}^1_q$ & $\begin{bmatrix}\langle q,t-p,u-p^2 \rangle \\ \langle 1 \rangle \\ \langle 1 \rangle\end{bmatrix}$     &  $\begin{bmatrix}\langle p,t,u \rangle \\ \langle 1 \rangle \\ \langle 1 \rangle\end{bmatrix}$ \\ \hline
\rule{0pt}{.4in}\vspace{.1in}$\mathcal{C}^2_q$ & $\begin{bmatrix}\langle q,t-p,u-p^2 \rangle \\ \langle q,t-p \rangle \\ \langle 1 \rangle\end{bmatrix}$ & $\begin{bmatrix}\langle p,t,u \rangle \\ \langle p,t \rangle \\ \langle 1 \rangle\end{bmatrix}$  \\ \hline
\rule{0pt}{.4in}\vspace{.1in}$\mathcal{C}^3_q$ & $\begin{bmatrix}\langle q,t-p,u-p^2 \rangle \\ \langle q,t-p \rangle \\ \langle q \rangle\end{bmatrix}$ & $\begin{bmatrix}\langle p,t,u \rangle \\ \langle p,t \rangle \\ \langle p \rangle\end{bmatrix}$ \\ \hline
\rule{0pt}{.4in}\vspace{.1in}$\mathcal{D}_q$   & $\begin{bmatrix}\langle q,t-p\rangle \\ \langle q \rangle \\ \langle q \rangle \end{bmatrix}$           &     $\begin{bmatrix}\langle p,t,u \rangle \\ \langle p,t \rangle \\ \langle p \rangle\end{bmatrix}$                     \\ \hline
\rule{0pt}{.4in}\vspace{.1in}$\mathcal{E}_q$   & $\begin{bmatrix}    \langle q,u \rangle \\ \langle q,t \rangle \\  \langle 1 \rangle\end{bmatrix}$      &    $\begin{bmatrix} \langle p,t,u \rangle \\ \langle p,t \rangle \\ \langle 1 \rangle \end{bmatrix}$                      \\ \hline
\rule{0pt}{.4in}\vspace{.1in}$\mathcal{F}_q$   & $\begin{bmatrix} \langle q,u-pt \rangle \\ \langle q,t-p \rangle \\ \langle q \rangle \end{bmatrix}$    &        $\begin{bmatrix} \langle p,t,u \rangle \\ \langle p,t \rangle \\ \langle p \rangle \end{bmatrix}$                  \\ \hline
\rule{0pt}{.4in}\vspace{.1in}$\mathcal{G}_q$   & $\begin{bmatrix}   \langle q \rangle \\ \langle q \rangle \\ \langle q \rangle \end{bmatrix}$           &   $\begin{bmatrix}\langle p,t,u \rangle \\ \langle p,t \rangle \\ \langle p \rangle\end{bmatrix}$                      
\end{tabular}
}
\caption{A table listing the families of primes that appear in the computation of the bi-incomplete spectra for the Burnside Tambara functor $\underline{A}_{C_{p^2}}$.}\label{tab:1}
\end{table}

\newpage

\renewcommand{\arraystretch}{1.75}

\begin{longtable}{cc|c}
$\mathcal{O}_m$ & $\mathcal{O}_a$ & $\Spec^{(\cO_m,\cO_a)}(\underline{A}_{C_{p^2}})$  \\ \hline
$\cO_\mathrm{triv}$ & $\cO_\mathrm{triv}$ & $\mathcal{A}_q,\mathcal{B}_q^1,\mathcal{B}_q^2,\mathcal{C}_q^1,\mathcal{C}_q^2,\mathcal{C}_q^3$ \\ \hline
$\cO_\mathrm{triv}$ & $\cO_1$ & $\mathcal{A}_q,\mathcal{B}_q^1,\mathcal{B}_q^2,\mathcal{C}_q^1,\mathcal{C}_q^3$  \\ \hline
$\cO_1$ & $\cO_1$ &  $\mathcal{A}_q,\mathcal{B}_q^1,\mathcal{C}_q^1, \mathcal{C}_q^3, \mathcal{D}_q$ \\ \hline
$\cO_\mathrm{triv}$ & $\cO_2$ & $\begin{gathered}
\begin{tikzpicture}[scale=0.45]
\draw[gray, thick] plot [smooth] coordinates {(0,-1) (8,-1)} node [black, right] {$\mathcal{C}_q^3$};
\draw[gray, thick] plot [smooth] coordinates {(0,1) (2.5,1) (4,0.5) (5.5,1) (8,1)} node [black, right] {$\mathcal{B}_q^2$};
\draw[gray, thick] plot [smooth] coordinates {(0,0) (2.5,0) (4,0.5) (5.5,0) (8,0)} node [black, right] {$\mathcal{C}_q^2$};
\draw[gray, thick] plot [smooth] coordinates {(0,2) (2.5,2)  (5.5,2) (8,2)} node [black, right] {$\mathcal{A}_q$};
\draw[black, fill=gray] (0,0) circle (5pt) node{};
\draw[black, fill=gray] (0,1) circle (5pt) node{};
\draw[black, fill=gray] (0,-1) circle (5pt) node{};
\draw[black, fill=gray] (0,2) circle (5pt) node{};
\filldraw[black] (4,2) circle (3pt);
\filldraw[black] (4,0.5) circle (3pt);
\filldraw[black] (4,-1) circle (3pt);
\draw[black!25!red, thick, <-] (4,1.8) -- (4,0.7);
\draw[black!25!red, thick, ->] (4,-0.8) -- (4,0.3);
\begin{scope}[yshift=-30]
\draw[black!25!red, thick, <-] (2,0.8) -- (2,0.2);
\draw[black!25!red, thick, <-] (6,0.8) -- (6,0.2);
\draw[black!25!red, thick, <-] (7,0.8) -- (7,0.2);
\draw[black!25!red, thick, <-] (8,0.8) -- (8,0.2);
\draw[black!25!red, thick, <-] (1,0.8) -- (1,0.2);
\draw[black!25!red, thick, <-] (0,0.7) -- (0,0.3);
\draw[black!25!red, thick, <-] (5,1) -- (5,0.2);
\draw[black!25!red, thick, <-] (3,1) -- (3,0.2);
\end{scope}
\end{tikzpicture}
\end{gathered}$ \\ \hline
$\cO_2$ & $\cO_2$ &  $\begin{gathered}
\begin{tikzpicture}[scale=0.45]
\draw[gray, thick] plot [smooth] coordinates {(0,-1) (8,-1)} node [black, right] {$\mathcal{C}_q^3$};
\draw[gray, thick] plot [smooth] coordinates {(0,0) (2.5,0) (4,1) (5.5,0) (8,0)} node [black, right] {$\mathcal{C}_q^2$};
\draw[gray, thick] plot [smooth] coordinates {(0,1) (2.5,1) (4,1) (5.5,1) (8,1)} node [black, right] {$\mathcal{E}_q$};
\draw[gray, thick] plot [smooth] coordinates {(0,2) (2.5,2) (4,1) (5.5,2) (8,2)} node [black, right] {$\mathcal{B}^2_q$};
\draw[black, fill=gray] (0,0) circle (5pt) node{};
\draw[black, fill=gray] (0,1) circle (5pt) node{};
\draw[black, fill=gray] (0,-1) circle (5pt) node{};
\draw[black, fill=gray] (0,2) circle (5pt) node{};
\filldraw[black] (4,1) circle (3pt);
\filldraw[black] (4,-1) circle (3pt);
\draw[black!25!red, thick, ->] (0,1.3) -- (0,1.7);
\draw[black!25!red, thick, ->] (1,1.1) -- (1,1.9);
\draw[black!25!red, thick, ->] (2,1.1) -- (2,1.9);
\draw[black!25!red, thick, ->] (6,1.1) -- (6,1.9);
\draw[black!25!red, thick, ->] (7,1.1) -- (7,1.9);
\draw[black!25!red, thick, ->] (8,1.1) -- (8,1.9);
\draw[black!25!red, thick, ->] (3,1.1) -- (3,1.6);
\draw[black!25!red, thick, ->] (5,1.1) -- (5,1.6);
\draw[black!25!red, thick, ->] (4,-0.8) -- (4,0.7);

\draw[black!25!red, thick, <-] (0,-0.3) -- (0,-0.7);
\draw[black!25!red, thick, <-] (1,-0.1) -- (1,-0.9);
\draw[black!25!red, thick, <-] (2,-0.1) -- (2,-0.9);
\draw[black!25!red, thick, <-] (6,-0.1) -- (6,-0.9);
\draw[black!25!red, thick, <-] (7,-0.1) -- (7,-0.9);
\draw[black!25!red, thick, <-] (8,-0.1) -- (8,-0.9);
\draw[black!25!red, thick, <-] (3,0.1) -- (3,-0.9);
\draw[black!25!red, thick, <-] (5,0.1) -- (5,-0.9);
\end{tikzpicture}
\end{gathered}$ \\ \hline
$\cO_\mathrm{triv}$ & $\cO_3$ & $
\begin{gathered}
\begin{tikzpicture}[scale=0.45]
\draw[gray, thick] plot [smooth] coordinates {(0,-1) (8,-1)} node [black, right] {$\mathcal{C}_q^3$};
\draw[gray, thick] plot [smooth] coordinates {(0,0) (2.5,0)  (5.5,0) (8,0)} node [black, right] {$\mathcal{B}_q^2$};
\draw[gray, thick] plot [smooth] coordinates {(0,2) (2.5,2) (4,1.5) (5.5,2) (8,2)} node [black, right] {$\mathcal{A}_q$};
\draw[gray, thick] plot [smooth] coordinates {(0,1) (2.5,1) (4,1.5) (5.5,1) (8,1)} node [black, right] {$\mathcal{B}_q^1$};
\draw[black, fill=gray] (0,0) circle (5pt) node{};
\draw[black, fill=gray] (0,1) circle (5pt) node{};
\draw[black, fill=gray] (0,2) circle (5pt) node{};
\draw[black, fill=gray] (0,-1) circle (5pt) node{};
\filldraw[black] (4,1.5) circle (3pt);
\filldraw[black] (4,0) circle (3pt);
\filldraw[black] (4,-1) circle (3pt);
\draw[black!25!red, thick, <-] (4,1.3) -- (4,0.2);
\draw[black!25!red, thick, <-] (2,0.8) -- (2,0.2);
\draw[black!25!red, thick, <-] (6,0.8) -- (6,0.2);
\draw[black!25!red, thick, <-] (7,0.8) -- (7,0.2);
\draw[black!25!red, thick, <-] (8,0.8) -- (8,0.2);
\draw[black!25!red, thick, <-] (1,0.8) -- (1,0.2);
\draw[black!25!red, thick, <-] (0,0.7) -- (0,0.3);
\draw[black!25!red, thick, <-] (5,1) -- (5,0.2);
\draw[black!25!red, thick, <-] (3,1) -- (3,0.2);
\draw[black!25!red, thick, ->] (4,-0.8) -- (4,-0.2);
\end{tikzpicture}
\end{gathered}
$ \\ \hline
$\cO_1$ & $\cO_3$ &  $\begin{gathered}
\begin{tikzpicture}[scale=0.45]
\draw[gray, thick] plot [smooth] coordinates {(0,-2) (2.5,-2) (4,-1.5) (5.5,-2) (8,-2)} node [black, right] {$\mathcal{A}_q$};
\draw[gray, thick] plot [smooth] coordinates {(0,-1) (2.5,-1) (4,-1.5) (5.5,-1) (8,-1)} node [black, right] {$\mathcal{B}_q^1$};
\draw[gray, thick] plot [smooth] coordinates {(0,0) (2.5,0) (4,0.5) (5.5,0) (8,0)} node [black, right] {$\mathcal{D}_q$};
\draw[gray, thick] plot [smooth] coordinates {(0,1) (2.5,1) (4,0.5) (5.5,1) (8,1)} node [black, right] {$\mathcal{C}_q^3$};
\draw[black, fill=gray] (0,0) circle (5pt) node{};
\draw[black, fill=gray] (0,1) circle (5pt) node{};
\draw[black, fill=gray] (0,-1) circle (5pt) node{};
\draw[black, fill=gray] (0,-2) circle (5pt) node{};
\filldraw[black] (4,0.5) circle (3pt);
\filldraw[black] (4,-1.5) circle (3pt);
\draw[black!25!red, thick, ->] (0,0.3) -- (0,0.7);
\draw[black!25!red, thick, ->] (1,0.1) -- (1,0.9);
\draw[black!25!red, thick, ->] (2,0.1) -- (2,0.9);
\draw[black!25!red, thick, ->] (6,0.1) -- (6,0.9);
\draw[black!25!red, thick, ->] (7,0.1) -- (7,0.9);
\draw[black!25!red, thick, ->] (8,0.1) -- (8,0.9);
\draw[black!25!red, thick, ->] (3,0.3) -- (3,0.7);
\draw[black!25!red, thick, ->] (5,0.3) -- (5,0.7);
\draw[black!25!red, thick, <-] (4,-1.3) -- (4,0.3);
\draw[black!25!red, thick, ->] (0,-0.3) -- (0,-0.7);

\draw[black!25!red, thick, ->] (1,-0.1) -- (1,-0.9);
\draw[black!25!red, thick, ->] (2,-0.1) -- (2,-0.9);
\draw[black!25!red, thick, ->] (6,-0.1) -- (6,-0.9);
\draw[black!25!red, thick, ->] (7,-0.1) -- (7,-0.9);
\draw[black!25!red, thick, ->] (8,-0.1) -- (8,-0.9);
\draw[black!25!red, thick, ->] (3,0.0) -- (3,-1);
\draw[black!25!red, thick, ->] (5,0.0) -- (5,-1);

\end{tikzpicture}
\end{gathered}$ \\ \hline
$\cO_\mathrm{triv}$ & $\cO_\mathrm{comp}$ &  $\begin{gathered}
\begin{tikzpicture}[scale=0.45]
\draw[gray, thick] plot [smooth] coordinates {(0,-1) (8,-1)} node [black, right] {$\mathcal{C}_q^3$};
\draw[gray, thick] plot [smooth] coordinates {(0,0) (2.5,0)  (5.5,0) (8,0)} node [black, right] {$\mathcal{B}_q^2$};
\draw[gray, thick] plot [smooth] coordinates {(0,1) (2.5,1)  (5.5,1) (8,1)} node [black, right] {$\mathcal{A}_q$};
\draw[black, fill=gray] (0,0) circle (5pt) node{};
\draw[black, fill=gray] (0,1) circle (5pt) node{};
\draw[black, fill=gray] (0,-1) circle (5pt) node{};
\filldraw[black] (4,1) circle (3pt);
\filldraw[black] (4,0) circle (3pt);
\filldraw[black] (4,-1) circle (3pt);
\draw[black!25!red, thick, <-] (4,0.8) -- (4,0.2);
\draw[black!25!red, thick, ->] (4,-0.8) -- (4,-0.2);
\end{tikzpicture}
\end{gathered}$ \\ \hline
$\cO_1$ & $\cO_\mathrm{comp}$ &  $\begin{gathered}
\begin{tikzpicture}[scale=0.45]
\draw[gray, thick] plot [smooth] coordinates {(0,-1) (8,-1)} node [black, right] {$\mathcal{A}_q$};
\draw[gray, thick] plot [smooth] coordinates {(0,0) (2.5,0) (4,0.5) (5.5,0) (8,0)} node [black, right] {$\mathcal{D}_q$};
\draw[gray, thick] plot [smooth] coordinates {(0,1) (2.5,1) (4,0.5) (5.5,1) (8,1)} node [black, right] {$\mathcal{C}_q^3$};
\draw[black, fill=gray] (0,0) circle (5pt) node{};
\draw[black, fill=gray] (0,1) circle (5pt) node{};
\draw[black, fill=gray] (0,-1) circle (5pt) node{};
\filldraw[black] (4,0.5) circle (3pt);
\filldraw[black] (4,-1) circle (3pt);
\draw[black!25!red, thick, ->] (0,0.3) -- (0,0.7);
\draw[black!25!red, thick, ->] (1,0.1) -- (1,0.9);
\draw[black!25!red, thick, ->] (2,0.1) -- (2,0.9);
\draw[black!25!red, thick, ->] (6,0.1) -- (6,0.9);
\draw[black!25!red, thick, ->] (7,0.1) -- (7,0.9);
\draw[black!25!red, thick, ->] (8,0.1) -- (8,0.9);
\draw[black!25!red, thick, ->] (3,0.3) -- (3,0.7);
\draw[black!25!red, thick, ->] (5,0.3) -- (5,0.7);
\draw[black!25!red, thick, <-] (4,-0.8) -- (4,0.3);
\end{tikzpicture}
\end{gathered}$\\ \hline
$\cO_2$ & $\cO_\mathrm{comp}$ &  $\begin{gathered}
\begin{tikzpicture}[scale=0.45]
\draw[gray, thick] plot [smooth] coordinates {(0,-1) (8,-1)} node [black, right] {$\mathcal{C}_q^3$};
\draw[gray, thick] plot [smooth] coordinates {(0,0) (2.5,0) (4,0.5) (5.5,0) (8,0)} node [black, right] {$\mathcal{E}_q$};
\draw[gray, thick] plot [smooth] coordinates {(0,1) (2.5,1) (4,0.5) (5.5,1) (8,1)} node [black, right] {$\mathcal{B}^2_q$};
\draw[black, fill=gray] (0,0) circle (5pt) node{};
\draw[black, fill=gray] (0,1) circle (5pt) node{};
\draw[black, fill=gray] (0,-1) circle (5pt) node{};
\filldraw[black] (4,0.5) circle (3pt);
\filldraw[black] (4,-1) circle (3pt);
\draw[black!25!red, thick, ->] (0,0.3) -- (0,0.7);
\draw[black!25!red, thick, ->] (1,0.1) -- (1,0.9);
\draw[black!25!red, thick, ->] (2,0.1) -- (2,0.9);
\draw[black!25!red, thick, ->] (6,0.1) -- (6,0.9);
\draw[black!25!red, thick, ->] (7,0.1) -- (7,0.9);
\draw[black!25!red, thick, ->] (8,0.1) -- (8,0.9);
\draw[black!25!red, thick, ->] (3,0.3) -- (3,0.7);
\draw[black!25!red, thick, ->] (5,0.3) -- (5,0.7);
\draw[black!25!red, thick, ->] (4,-0.8) -- (4,0.3);
\end{tikzpicture}
\end{gathered}$\\ \hline
$\cO_3$ & $\cO_\mathrm{comp}$ &  $\begin{gathered}
\begin{tikzpicture}[scale=0.45]
\draw[gray, thick] plot [smooth] coordinates {(0,0) (8,0)} node [black, right] {$\mathcal{C}_q^3$};
\draw[gray, thick] plot [smooth] coordinates {(0,1) (2.5,1) (4,0) (5.5,1) (8,1)} node [black, right] {$\mathcal{D}_q$};
\draw[gray, thick] plot [smooth] coordinates {(0,-1) (2.5,-1) (4,0) (5.5,-1) (8,-1)} node [black, right] {$\mathcal{F}_q$};
\draw[black, fill=gray] (0,0) circle (5pt) node{};
\draw[black, fill=gray] (0,1) circle (5pt) node{};
\draw[black, fill=gray] (0,-1) circle (5pt) node{};
\filldraw[black] (4,0) circle (3pt);
\draw[black!25!red, thick, <-] (0,0.3) -- (0,0.7);
\draw[black!25!red, thick, <-] (1,0.1) -- (1,0.9);
\draw[black!25!red, thick, <-] (2,0.1) -- (2,0.9);
\draw[black!25!red, thick, <-] (6,0.1) -- (6,0.9);
\draw[black!25!red, thick, <-] (7,0.1) -- (7,0.9);
\draw[black!25!red, thick, <-] (8,0.1) -- (8,0.9);
\begin{scope}[yshift=-30]
    \draw[black!25!red, thick, ->] (0,0.3) -- (0,0.7);
\draw[black!25!red, thick, ->] (1,0.1) -- (1,0.9);
\draw[black!25!red, thick, ->] (2,0.1) -- (2,0.9);
\draw[black!25!red, thick, ->] (6,0.1) -- (6,0.9);
\draw[black!25!red, thick, ->] (7,0.1) -- (7,0.9);
\draw[black!25!red, thick, ->] (8,0.1) -- (8,0.9);
\end{scope}
\draw[black!25!red, thick, <-] (3,0.1) -- (3,0.6);
\draw[black!25!red, thick, <-] (5,0.1) -- (5,0.6);
\draw[black!25!red, thick, ->] (3,-0.6) -- (3,-0.1);
\draw[black!25!red, thick, ->] (5,-0.6) -- (5,-0.1);
\end{tikzpicture}
\end{gathered}$ \\ \hline
$\cO_\mathrm{comp}$ & $\cO_\mathrm{comp}$ &  $\begin{gathered}
\begin{tikzpicture}[scale=0.45]
\draw[gray, thick] plot [smooth] coordinates {(0,0) (8,0)} node [black, right] {$\mathcal{D}_q$};
\draw[gray, thick] plot [smooth] coordinates {(0,1) (2.5,1) (4,0) (5.5,1) (8,1)} node [black, right] {$\mathcal{C}_q^3$};
\draw[gray, thick] plot [smooth] coordinates {(0,-1) (2.5,-1) (4,0) (5.5,-1) (8,-1)} node [black, right] {$\mathcal{G}_q$};
\draw[black, fill=gray] (0,0) circle (5pt) node{};
\draw[black, fill=gray] (0,1) circle (5pt) node{};
\draw[black, fill=gray] (0,-1) circle (5pt) node{};
\filldraw[black] (4,0) circle (3pt);
\draw[black!25!red, thick, ->] (0,0.3) -- (0,0.7);
\draw[black!25!red, thick, ->] (1,0.1) -- (1,0.9);
\draw[black!25!red, thick, ->] (2,0.1) -- (2,0.9);
\draw[black!25!red, thick, ->] (6,0.1) -- (6,0.9);
\draw[black!25!red, thick, ->] (7,0.1) -- (7,0.9);
\draw[black!25!red, thick, ->] (8,0.1) -- (8,0.9);
\begin{scope}[yshift=-30]
    \draw[black!25!red, thick, ->] (0,0.3) -- (0,0.7);
\draw[black!25!red, thick, ->] (1,0.1) -- (1,0.9);
\draw[black!25!red, thick, ->] (2,0.1) -- (2,0.9);
\draw[black!25!red, thick, ->] (6,0.1) -- (6,0.9);
\draw[black!25!red, thick, ->] (7,0.1) -- (7,0.9);
\draw[black!25!red, thick, ->] (8,0.1) -- (8,0.9);
\end{scope}
\draw[black!25!red, thick, ->] (3,0.1) -- (3,0.6);
\draw[black!25!red, thick, ->] (5,0.1) -- (5,0.6);
\draw[black!25!red, thick, ->] (3,-0.6) -- (3,-0.1);
\draw[black!25!red, thick, ->] (5,-0.6) -- (5,-0.1);
\end{tikzpicture}
\end{gathered}$
\vspace{-5mm}
\\
\caption{A description of the bi-incomplete spectra for the Burnside Tambara functor $\underline{A}_{C_{p^2}}$ using the notation of \cref{tab:1}. For completeness we also describe the topology for small cases which is sufficient to show that we get 12 homeomorphism types. The solid black points correspond to $q=p$ and the large gray circles are generic points, with each line being a homeomorphic copy of $\Spec(\mathbb{Z})$. The calculations are grouped by the choice of $\cO_a$.}\label{tab:ofspectra}
\end{longtable}

\subsection{$G = C_{pq}$}

Let $p$ and $q$ be distinct primes and let $G=C_{pq}$. In this final example we will compute prime ideals for a specific compatible pair which will highlight phenomena that we have not yet seen. In particular, every time that we have considered a restricted Tambara functor it has simply been a Tambara functor for a subgroup. This need not be the case.

There is a self-compatible transfer system for $C_{pq}$ defined as
\[
\cO = \{ (e,C_p),(e,C_q) \}.
\]
Then $\pi_0 \mathcal{O}$ contains two elements:
\begin{itemize}
\item the path component $\mathcal{O}'$ containing only $C_{pq}$, and
\item the path component $\mathcal{O}''$ containing $e$, $C_p$, and $C_q$. 
\end{itemize}
We will describe $\Spec(\underline{A}_{C_{pq}})$, where $\underline{A}_{C_{pq}}$ is the restriction of the Burnside Tambara functor to $(\mathcal{O},\mathcal{O})$-Tambara functors via \cref{thm:pi0}.


Since $i^*_{\mathcal{O}'}\underline{A}_{C_{pq}} = A(C_{pq})$ with trivial Weyl action, we have $\Spec(i^*_{\mathcal{O}'}\underline{A}_{C_{pq}}) \cong \Spec(A(C_{pq}))$. This is classical; see \cite[Sec. 5]{dress_notes}. 

The computation of $\Spec(i^*_{\mathcal{O}''}\underline{A}_{C_{pq}})$ is more interesting; in particular, $\mathcal{O}''$ is not the complete transfer system for any group (a fortiori any subgroup of $C_{pq}$) as it does not have a a maximal element. However, its further restriction to $e$ and $C_p$ (or to $e$ and $C_q$) is the complete transfer system for a cyclic group of prime order which we have computed in \cref{subsec:cpburnside}.

\begin{lemma}
Let $\mathcal{P}$ be a prime ideal of $i^*_{\mathcal{O}''}\underline{A}_{C_{pq}}$. Then the restrictions of $\mathcal{P}$ to ideals in $\underline{A}_{C_p}$ and $\underline{A}_{C_q}$ are prime $(\cO_\mathrm{comp},\cO_\mathrm{comp})$-Tambara ideals. 
\end{lemma}

\begin{proof}
Let $T = i^*_{\mathcal{O}''}\underline{A}_{C_{pq}}$, let $\mathcal{P}$ be a prime ideal of $T$, and let $\mathcal{P}'$ be the restriction of $\mathcal{P}$ to an ideal of $\underline{A}_{C_p}$ (the argument for $C_q$ is identical). Suppose that $x \in T(C_{pq}/H)$ and $y \in T(C_{pq}/K)$ for some $H,K \in \{e,C_p\}$ and that $\mathtt{Q}(\mathcal{P}',x,y)$ holds. Since $\mathcal{P}$ is closed under norms, it follows immediately that $\mathtt{Q}(\mathcal{P},x,y)$ also holds, and thus by primality of $\mathcal{P}$, $x \in \mathcal{P}(C_{pq}/H) = \mathcal{P}'(C_p/H)$ or $y \in \mathcal{P}(C_{pq}/K) = \mathcal{P}'(C_{p}/K)$. Thus $\mathcal{P}'$ is prime. 
\end{proof}

On the other hand, not all combinations of prime ideals in $\underline{A}_{C_p}$ and $\underline{A}_{C_q}$ which coincide at the underlying level give rise to a prime ideal in $i^*_{\mathcal{O}''}\underline{A}_{C_{pq}}$:

\begin{proposition}\label{prop:cpqbottom}
There is a bijection
$$\Spec(i^*_{\mathcal{O}''}\underline{A}_{{C_{pq}}}) \cong \Spec(\underline{A}_{C_p}) \times_{\Spec(\mathbb{Z})} \Spec(\underline{A}_{C_q}) \setminus \{(\langle r \rangle, \langle r \rangle): r \neq p,q \text{ and $r$ prime or zero}\}.$$
\end{proposition}

\begin{proof}
The previous lemma implies that the prime ideals of $i^*_{\mathcal{O}''}\underline{A}_{C_{pq}}$ restrict to pairs of primes in $\underline{A}_{C_p}$ and $\underline{A}_{C_q}$, and these obviously must coincide at the underlying level. We must determine which pairs comprise prime $i^*_{\mathcal{O}''}\underline{A}_{C_{pq}}$-ideals. Let us write $A(C_p) = \mathbb{Z}[x_p]/\langle x_p^2-px_p \rangle$ and $A(C_q) = \mathbb{Z}[x_q]/\langle x_q^2-qx_q \rangle$. 

First, consider the ideal $\mathcal{P} = (\langle r \rangle, \langle r \rangle)$, where $r \neq p,q$. Then $\mathtt{Q}(\mathcal{P},x_p-p,x_q-q)$ holds, but $x_p-p$ and $x_q-q$ are not in $\mathcal{P}$. 

On the other hand, consider the case $\mathcal{P} = (\mathcal{P}_1,\mathcal{P}_2)$ with at least one $\mathcal{P}_i$ maximal, say $\mathcal{P}_1(C_p/C_p) = \langle r,x_p-p \rangle$ with $r$ a prime or zero. Note that in this case $\mathcal{P}_1(C_p/C_p) = (\res^{p}_1)^{-1}(q)$ is the inverse image of the underlying level. Let $x \in \underline{A}(C_{pq}/C_p)$ and let $y \in \underline{A}(C_{pq}/C_q)$. Suppose that $\mathtt{Q}(x,y,\mathcal{P})$ holds and that $x \notin \mathcal{P}$. We must show that $y \in \mathcal{P}$. The $\mathtt{Q}$-condition implies that $\res^{p}_1(x) \res^{q}_1(y) \in (r)$, so either $\res^{p}_1(x) \in (r)$ or $\res^{q}_1(y) \in (r)$. Since we assumed $x \notin \mathcal{P}$, we have $\res^{q}_1(y) \in (r)$, i.e., $y \in (\res^{q}_1)^{-1}(r)$. 

If $\mathcal{P}_2$ is maximal, then we are done, so suppose otherwise, i.e., suppose that $\mathcal{P}_2(C_q/C_q) = (r)$. Suppose $y \notin (r)$. The $\mathtt{Q}$-condition implies that $\nm_1^{q}(\res_1^{p}(x)) \cdot y \in (r)$, so $\nm_1^{q}(\res_1^{p}(x)) \in (r)$. Let $z = \res^{p}_1(x)$. Then $r$ divides $\nm_1^{q}(z) = z + \frac{z^q-z}{q}x_q$, and consequently, $r$ divides $\res_1^{p}(x)$. But then $\res_1^{p}(x) \in (r)$, contradicting our hypothesis that $x \notin \mathcal{P}$. 
\end{proof}

Between our calculation of the prime ideals of $i^*_{\mathcal{O}'}\underline{A}_{C_{pq}}$ and the prime ideals of $i^*_{\mathcal{O}''}\underline{A}_{C_{pq}}$, it is straightforward to compute $\Spec(\underline{A}_{C_{pq}})$ using \cref{thm:pi0}. Since our main motivation for this section was really illustrating the computation of the spectrum of a restricted Tambara functor which is not the restriction to a subgroup, we leave the final details of the computation (computing the inverse images of restrictions and determining inclusions) to the interested reader.

\section{Vistas: $C_2$-equivariant tensor-triangular computations}\label{sec:vista}

Let $(\sfT, \otimes, \mathbbm{1}_{\sfT})$ be an essentially small tensor-triangulated (tt-)category. Then we have the associated Balmer spectrum $\Spc(\sfT)$ of the prime thick-tensor ideals equipped with the topology determined by a universal support theory \cite{balmer_spectrum}. 

\begin{remark}\label{rem:balmertop}
Unwrapping the definition of the topology we have for $\mathsf{P} \in \Spc(\sfT)$ that
\[
\overline{\{\sfP\}} = \{\mathsf{Q} \in \Spc(\sfT) \mid \mathsf{Q} \subseteq \sfP\}.
\]
One should compare this to the Zariski case, or the Nakaoka topology (c.f., \cref{rem:closureofpoints}) and observe that there there is a reversal in the ordering.
\end{remark}

In \cite{balmer_3spectra} Balmer provides a continuous comparison map 
\[\Spc(\sfT) \to \Spec(\mathrm{End}(\mathbbm{1}_{\sfT}))\]
from the Balmer spectrum of $\sfT$ to the Zariski spectrum of the endomorphism ring of the unit.

When $\sfT = \mathsf{Sp}_G^\omega$ for a finite group $G$, this provides a comparison map to the Zariski spectrum of the Burnside ring of $G$ \cite{balmersanders}. There has been a conjectural idea of enhancing tt-categories with equivariant structure so that instead of taking endomorphism rings, we obtain endomorphism $(\cO_m,\cO_a)$-Tambara functors. One should then also obtain a comparison map between some sort of ``Balmer--Nakaoka'' spectrum and the Nakaoka spectrum of the endomorphism bi-incomplete Tambara functor. 

One obstacle to this is that when we look at these ``Tambara tt-categories'', we can no longer expect the prime ideals to be levelwise prime, just as in the case of the Nakaoka spectrum of the Burnside Tambara functor. This provides an obstruction to repurposing Balmer's proof of the comparison map in this more structured setting, as it critically hinges on using local rings arising from localizing away from prime ideals.

Indeed, take some $\sfP \in \Spc(\sfT)$ and consider the localization $\varphi \colon \sfT \to \sfT_\p$ which induces a ring homomorphism $f \colon \mathrm{End}(\mathbbm{1}_\sfT) \to \mathrm{End}(\mathbbm{1}_{\sfT_\sfP})$. Applying $\Spec$ we obtain a map $f \colon \Spec(\mathrm{End}(\mathbbm{1}_{\sfT_\sfP})) \to \Spec(\mathrm{End}(\mathbbm{1}_{\sfT}))$. Now, the former space is local (as we have the Zariski spectrum of a local ring). As such it has a distinguished $\mathfrak{m}$ (i.e., the unique maximal ideal). We then declare the image of $\p$ under the desired comparison map to be $f^\ast (\m)$. 

In this final section we will compute this conjectural structured spectrum for $\sfT = \mathsf{Sp}_{C_2}^\omega$, which should be thought of as the categorical lift of $\underline{A}_{C_2}$, and provide evidence that the existence of the expected comparison maps is more nuanced than expected.

Foregoing any formal construction, we simply describe the ``Tambara tt-category'' that we wish to study to be the  object
\[\begin{tikzcd}[row sep=large]
	{\mathsf{Sp}_{C_2}^\omega} \\
	{\mathsf{Sp}^\omega}
	\arrow["{\Phi^e}"{description}, from=1-1, to=2-1]
	\arrow["{\mathrm{coind}_1^{2}}", shift left=4, curve={height=-6pt}, from=2-1, to=1-1]
	\arrow["{N_1^2}"', shift right=4, curve={height=6pt}, from=2-1, to=1-1]
\end{tikzcd}\]
where
\begin{itemize}
    \item $\Phi^e$ is the trivial geometric fixed points of a $C_2$-spectrum (i.e., the restriction functor), which is a strong monoidal triangulated functor, 
    \item $\mathrm{coind}_1^2$ is the coinduction functor $F((C_2)_+,-)$, which is a triangulated functor, and
    \item $N_1^2$ is the Hill--Hopkins--Ravenel norm functor \cite{hhr}, which is a strong monoidal functor.
\end{itemize}
We will denote this object by $\underline{\mathsf{Sp}}_{C_2}$.

These functors play the role of the restriction, transfer, and norm, respectively. We obtain similar relations as in the Tambara functor case:
\begin{itemize}
    \item Coinduction is the right adjoint to the geometric fixed points $\Phi^e$. Therefore by \cite{bds_wirth}, the projection formula yields an isomorphism
    \[
        \mathrm{coind}_1^2(X) \otimes Y \cong \mathrm{coind}_1^2(X \otimes \Phi^e(Y))
    \]
    which is simply Frobenius reciprocity. We also highlight that as $C_2$ is finite, the Wirthm\"{u}ller isomorphism tells us that the coinduction coincides with the left adjont of $\Phi^e$, induction, which is $(C_2)_+ \otimes -$, as the dualizing object is simply the sphere spectrum.
    \item One can extract from the definition of the norm that
    \[
    N_1^2(X \coprod Y) \cong N_1^2(X) \coprod N_1^2(Y) \coprod \mathrm{coind}_1^2(X \otimes Y)
    \]
    which provides Tambara reciprocity.
    \item We also have the double coset formulae:
    \[
    \Phi^e \mathrm{coind}_1^2 (X) \cong X \coprod X \qquad \text{and} \qquad \Phi^e N_1^2(X) \cong X \otimes X.
    \]
\end{itemize}
 We can replicate the definitions of thick tensor (tt)-ideals in this structured setting, defining them to be levelwise collections of tt-ideals closed under the structure functors. We may similarly extend the definition of prime ideals and the $\mathtt{Q}$-condition. Finally, we will equip these spectra with the topology arising from the Balmer spectrum so that closures of a prime is constructed by taking those primes contained in it.
 
 We aim to construct the coefficient system, Green, and Tambara primes in this setting, and as such, we sensibly first recall the construction of the prime ideals in both levels (see, for instance, \cite{balmersanders}).  We have chosen this simple example as we have a full description of the topology of the Balmer spectra in both cases, which by Balmer's classification result (along with the fact that both categories are rigid) provides an explicit description of all ideals.

Denote by $\mathsf{K}_{p,n}$ the kernel in $\mathsf{Sp}^\omega$ of the $p$-local $(n-1)$-st Morava $K$-theory (composed with localization $\mathsf{Sp}^\omega \to \mathsf{Sp}^\omega_{(p)}$ at $p$). In particular $\mathsf{K}_{p,1} := \mathsf{Sp}^{\omega,\mathrm{tor}} =: \mathsf{K}_{0,1}$ is the subcategory of torsion finite spectra, independently of $p$, while $\mathsf{K}_{p,\infty} = \cap_{n \geq 1} \mathsf{K}_{p,n} = \ker(\mathsf{Sp}^\omega \to \mathsf{Sp}^\omega_{(p)})$ is the subcategory of $p$-acyclic finite spectra. For the topology, the closed points are the $\mathsf{K}_{p,\infty}$ while $\mathsf{K}_{0,1}$ is the generic point. The inclusions $\mathsf{K}_{p,n} \subsetneq \mathsf{K}_{p,n-1}$  provide a description of the closure
\[
\overline{\{\mathsf{K}_{p,n}\}} = \{\mathsf{K}_{p,k} \mid n \leq k \leq \infty\}.
\]
The Thomason subsets (which are in bijection with the tt-ideals under Balmer's classification) are either empty, the entire space, or an arbitrary union of columns $\overline{\{\mathsf{K}_{p,n}\}}$ with $n \geq 2$.

Equipped with our non-equivariant primes we can consider the topology of $C_2$-equivariant spectra (we could work more generally for an abelian group $G$ but this is not relevant to our story).

For a subgroup $H$ of $C_2$ we denote by $\mathsf{P}_{H,p,n} := (\Phi^H)^{-1}(\mathsf{K}_{p,n})$. That is,
\[
    \mathsf{P}_{H,p,n} = \{X \in \mathsf{Sp}^\omega_{C_2} \mid \Phi^H(X) \in \mathsf{K}_{p,n}\}.
\]
There are inclusion $\mathsf{P}_{H,p,n} \subsetneq \mathsf{P}_{H,p,n-1}$. There are also inclusions which arise due to blueshift, namely $\mathsf{P}_{e,2,n} \subseteq \mathsf{P}_{C_2,2,n-1}$ for all $n \geq 2$.

\subsection{Coefficient system primes}

We will begin by computing the coefficient system primes. By using \cref{ex:coefficientsystem} we obtain some candidate primes, but we will need to check that there are no other options as one should not naively assume that commutative algebra results always carry over to the tt-setting. The candidate primes are:
\[
\mathsf{A}_{p,n} = \begin{bmatrix}
    \mathsf{P}_{{C_2},p,n} \\ \langle \mathbb{S} \rangle
\end{bmatrix}
\qquad
\mathsf{B}_{p,n}^1 = \begin{bmatrix}
    \mathsf{P}_{e,p,n} \\ \langle \mathbb{S} \rangle
\end{bmatrix}
\qquad
\mathsf{B}_{p,n}^2 = \begin{bmatrix}
    \mathsf{P}_{e,p,n} \\ \mathsf{K}_{p,n}
\end{bmatrix}
\]
Along with the usual chromatic inclusions, there are inclusions $\mathsf{B}_{p,n}^2 \subseteq \mathsf{B}_{p,n}^1$ and $\mathsf{B}_{2,n}^1 \subseteq \mathsf{A}_{2,n-1}$.


We need to check that there are no other prime ideals. As we must be levelwise prime and closed under restriction, we are reduced to checking if 
\[
\begin{bmatrix}
    \mathsf{P}_{e,p,m} \\ \mathsf{K}_{p,n}
\end{bmatrix}
\]
is a coefficient system prime for $m > n$. Take $X \in \mathsf{P}_{e,p,n} \setminus \mathsf{P}_{e,p,m}$ at the top level and $\mathbb{S}$ at the bottom level. Then the only generalized product we can form is $\Phi^e(X) \otimes \mathbb{S} \cong \Phi^e(X) \in \mathsf{K}_{p,n}$ even though neither of the elements are in the ideal. Thus, this is not a prime ideal.

\begin{lemma}
    \[
\Spc^{(\cO_\mathrm{triv} , \cO_\mathrm{triv})}(\underline{\mathsf{Sp}}_{C_{2}}) \cong  \{ \mathsf{A}_{p,n}, \mathsf{B}_{p,n}^1, \mathsf{B}_{p,n}^2 \}_{p \in \mathbb{P}, n \in \mathbb{N} \cup \infty}.
\]    
\end{lemma}

\subsection{Green functor primes}

We follow \cref{Prop:ChangeOaPrimes} and see which of our coefficient system Balmer primes are closed under coinduction, which we have discussed is equivalent to being closed under induction. We make use of the following result:

\begin{lemma}[{\cite[Lemma 3.9]{MR4036448}}]\label{lem:coinductthing}
    Let $H$ be a subgroup of a finite group $G$ and let $X$ be a finite $H$-spectrum. Then
    \begin{enumerate}
        \item The type of $\Phi^H(\mathrm{ind}_H^G(X))$ is equal to the type of $\Phi^H(X)$.
        \item If $K$ is not subconjugate to $H$, then the geometric fixed points $\Phi^K(\mathrm{ind}_H^G(X))$ are trivial.
    \end{enumerate}
\end{lemma}

From this we observe that $\Phi^e (\mathrm{ind}_1^2(\mathbb{S}))$ has the same type as $\mathbb{S}$, while $\Phi^{C_2} (\mathrm{ind}_1^2(\mathbb{S}))$ is trivial. It follows that $\mathsf{A}_{p,n}$ remains an ideal, while $\mathsf{B}_{p,n}^1$ does not, exactly as in the Burnside case.

\begin{lemma}
    \[
\Spc^{(\cO_\mathrm{triv} , \cO_\mathrm{comp})}(\underline{\mathsf{Sp}}_{C_{2}}) \cong  \{ \mathsf{A}_{p,n}, \mathsf{B}_{p,n}^2 \}_{p \in \mathbb{P}, n \in \mathbb{N} \cup \infty}.
\]    
\end{lemma}

\subsection{Tambara functor primes}

We now compute the Tambara Balmer primes. This is the situation where we foresee potential issues as we will no longer be levelwise prime. We will require the following result:
\begin{lemma}\label{lem:normandgeom}
    Let $X$ be a finite type $n$ spectrum and let $H$ be a subgroup of a finite group $G$. Then
    \begin{enumerate}
        \item  There is a natural isomorphism $\Phi^G(N_H^G(X)) \cong \Phi^H(X)$.
        \item $\Phi^{e}(N_e^G(X))$ is also of type $n$.
    \end{enumerate}
\end{lemma}

We approach this in an \emph{ad hoc} manner, rather than attempting to compute using a ghost construction as in \cite{4DS} which we will postpone to future work. We use the fact that the bottom level must be a Balmer prime of $\mathsf{Sp}^\omega$, and that we must be closed under $\Phi^e$, coinduction, and norm. As such, we take $\mathsf{B}_{p,n}^2$ are our starting point.

\begin{lemma}
    $\mathsf{B}_{p,n}^2$ is a Tambara prime for all $p$ and for all $n$.
\end{lemma}

\begin{proof}
    We first check that it is a Tambara ideal. This follows immediately from \cref{lem:normandgeom}(2). The fact that it is prime then follows from \cref{prop:add_mult}.
\end{proof}

We introduce notation for a new ideal of $\mathsf{Sp}_{C_2}^\omega$:
\[
\mathsf{I}_{p,n} = \mathsf{P}_{e,p,n} \cap \mathsf{P}_{C_2,p,n} =   \{ X \in \mathsf{Sp}_{C_2}^\omega \mid \Phi^e(X) \in \mathsf{K}_{p,n} \text{ and } \Phi^{C_2}(X) \in \mathsf{K}_{p,n} \}
\]

\begin{lemma}
    The collection
    \[
\mathsf{C}_{p,n} = \begin{bmatrix}\mathsf{I}_{p,n} \\ \mathsf{K}_{p,n} \end{bmatrix}
    \]
    is a prime Tambara ideal.
\end{lemma}

\begin{proof}
    As $\mathsf{C}_{p,n} \subseteq \mathsf{B}^2_{p,n}$ we know that it is closed under restriction. We now check that it is closed under the coinduction and the norm. This reduces to checking that:
    \begin{enumerate}
        \item If $X \in \mathsf{K}_{p,n}$ then $\mathrm{coind}_1^2(X) \in \mathsf{P}_{C_2,p,n}$;
        \item If $X \in \mathsf{K}_{p,n}$ then $N_1^2(X) \in \mathsf{P}_{C_2,p,n}$;
    \end{enumerate}
    The first of these claims follows from \cref{lem:coinductthing}(2), while the second claim follows from \cref{lem:normandgeom}(1). As such, $\mathsf{C}_{p,n}$ is a Tambara  ideal.

    We are now left with checking that $\mathsf{C}_{p,n}$ is prime. As it is a subideal of $\mathsf{A}_{p,n}$ we need only the $\mathtt{Q}$-condition for things that have been left out. The result follows similar to the Burnside Tambara case.
\end{proof}

\begin{lemma}
    \[
\Spc^{(\cO_\mathrm{comp} , \cO_\mathrm{comp})}(\underline{\mathsf{Sp}}_{C_{2}}) \cong  \{ \mathsf{B}_{p,n}^2, \mathsf{C}_{p,n} \}_{p \in \mathbb{P}, n \in \mathbb{N} \cup \infty}.
\]    
\end{lemma}

\subsection{Comparison maps}\label{subsec:comparisonmaps}

In the previous sections we have computed the three spectra for our structured tt-category, however we have not yet made any claims regarding the comparison maps that we seek. In \cref{rem:closureofpoints} and \cref{rem:balmertop} we have highlighted how inclusions of ideals lead to opposite effects in the resulting topology. Our naming conventions have been kept consistent throughout our examples, for example looking at the rational level we have
\[
\mathcal{B}_{0}^2 \subset \mathcal{B}_0^1 \qquad \text{ and } \qquad \sfB_{0,1}^2 \subset \sfB_{0,1}^1.
\]
As such, while $\mathcal{B}_0^1 \in \overline{\{\mathcal{B}_0^2\}}$, $\mathsf{B}_{0,1}^1 \not\in \overline{\{\mathcal{B}_{0,1}^2\}}$. As such, the obvious choice of a comparison map 
 \[
\Spc^{(\cO_\mathrm{triv} , \cO_\mathrm{triv})}(\underline{\mathsf{Sp}}_{C_{2}}) \to \Spec^{(\cO_\mathrm{triv} , \cO_\mathrm{triv})}(\underline{A}_{C_{2}})
\]
\[
\mathsf{A}_{p,n} \mapsto \mathcal{A}_p \qquad \mathsf{B}_{p,n}^1 \mapsto \mathcal{B}_p^1 \qquad \mathsf{B}_{p,n}^2 \mapsto \mathcal{B}_p^2
\]
fails to be continuous.

In \cref{fig:comparison_coefficient,,fig:comparison_green,,fig:comparison_tambara}, we display the Balmer Nakaoka spectra on top of the corresponding Nakaoka spectrum in the coefficient, Green, and Tambara situations. In these figures a green squiggly arrow denotes topological specialization, as such closure goes up. We see that in all cases the obvious comparison map is not continuous, and the obstruction arises from the type of inclusions mentioned above.

We can form continuous comparison maps, but even in this simple example, they feel very \emph{ad hoc}.  For example there is a continuous comparison map
\[
\Spc^{(\cO_\mathrm{triv} , \cO_\mathrm{comp})}(\underline{\mathsf{Sp}}_{C_{2}}) \to \Spec^{(\cO_\mathrm{triv} , \cO_\mathrm{comp})}(\underline{A}_{C_{2}})
\]
\[
\mathsf{A}_{p,n} \mapsto \mathcal{B}_p \qquad \mathsf{B}_{p,n}^2 \mapsto \mathcal{A}_p^2
\]
In some sense we see that we need to consider the ``Balmer direction'' and ``Nakaoka direction'' separately.

\newpage

\begin{figure}[htp]
    \centering
    \scalebox{0.65}{
\begin{tikzpicture}[yscale=1.5, xscale=1.25]
    \tikzset{
  rightsquigarrow/.style={
    decorate,
    decoration={
    zigzag,
    segment length=4,
    amplitude=.4,post=lineto,
    post length=2pt
}
  }
    }
    \tikzset{>=latex}

    \node[circle,fill=gray,inner sep=2.5pt,minimum size=2pt,draw=black,label=below:$\mathsf{B}^2_{0,1}$] (a1) at (0,0) {};
    \node[circle,fill=gray,inner sep=2.5pt,minimum size=2pt,draw=black,label=below:$\mathsf{B}^1_{0,1}$] (b1) at (1,0) {};
    \node[circle,fill=gray,inner sep=2.5pt,minimum size=2pt,draw=black,label=below:$\mathsf{A}_{0,1}$] (c1) at (2,0) {};

    \node[circle,fill=black,inner sep=2pt,minimum size=2pt] (c22) at (-1,1) {};
    \node[circle,fill=black,inner sep=2pt,minimum size=2pt] (b22) at (-2,1) {};
    \node[circle,fill=black,inner sep=2pt,minimum size=2pt] (a22) at (-3,1) {};

    \node[circle,fill=black,inner sep=2pt,minimum size=2pt] (c23) at (-1,2) {};
    \node[circle,fill=black,inner sep=2pt,minimum size=2pt] (b23) at (-2,2) {};
    \node[circle,fill=black,inner sep=2pt,minimum size=2pt] (a23) at (-3,2) {};

    \node[circle,fill=black,inner sep=2pt,minimum size=2pt] (c24) at (-1,3) {};
    \node[circle,fill=black,inner sep=2pt,minimum size=2pt] (b24) at (-2,3) {};
    \node[circle,fill=black,inner sep=2pt,minimum size=2pt] (a24) at (-3,3) {};

    \node[circle,fill=black,inner sep=2pt,minimum size=2pt] (c25) at (-1,4) {};
    \node[circle,fill=black,inner sep=2pt,minimum size=2pt] (b25) at (-2,4) {};
    \node[circle,fill=black,inner sep=2pt,minimum size=2pt] (a25) at (-3,4) {};

    \node[circle,fill=none,inner sep=2pt,minimum size=2pt] (c26) at (-1,5) {$\vdots$};
    \node[circle,fill=none,inner sep=2pt,minimum size=2pt] (b26) at (-2,5) {$\vdots$};
    \node[circle,fill=none,inner sep=2pt,minimum size=2pt] (a26) at (-3,5) {$\vdots$};

    \node[circle,fill=black,inner sep=2pt,minimum size=2pt, label=above:$\mathsf{A}_{2,n}$] (c2i) at (-1,5.25) {};
    \node[circle,fill=black,inner sep=2pt,minimum size=2pt,label=above:$\mathsf{B}^1_{2,n}$] (b2i) at (-2,5.25) {};
    \node[circle,fill=black,inner sep=2pt,minimum size=2pt,label=above:$\mathsf{B}^2_{2,n}$] (a2i) at (-3,5.25) {};

    \node[circle,fill=black,inner sep=2pt,minimum size=2pt] (ap2) at (3,1) {};
    \node[circle,fill=black,inner sep=2pt,minimum size=2pt] (bp2) at (4,1) {};
    \node[circle,fill=black,inner sep=2pt,minimum size=2pt] (cp2) at (5,1) {};

    \node[circle,fill=black,inner sep=2pt,minimum size=2pt] (ap3) at (3,2) {};
    \node[circle,fill=black,inner sep=2pt,minimum size=2pt] (bp3) at (4,2) {};
    \node[circle,fill=black,inner sep=2pt,minimum size=2pt] (cp3) at (5,2) {};

    \node[circle,fill=black,inner sep=2pt,minimum size=2pt] (ap4) at (3,3) {};
    \node[circle,fill=black,inner sep=2pt,minimum size=2pt] (bp4) at (4,3) {};
    \node[circle,fill=black,inner sep=2pt,minimum size=2pt] (cp4) at (5,3) {};

    \node[circle,fill=black,inner sep=2pt,minimum size=2pt] (ap5) at (3,4) {};
    \node[circle,fill=black,inner sep=2pt,minimum size=2pt] (bp5) at (4,4) {};
    \node[circle,fill=black,inner sep=2pt,minimum size=2pt] (cp5) at (5,4) {};

    \node[circle,fill=none,inner sep=2pt,minimum size=2pt] (cp6) at (5,5) {$\vdots$};
    \node[circle,fill=none,inner sep=2pt,minimum size=2pt] (bp6) at (4,5) {$\vdots$};
    \node[circle,fill=none,inner sep=2pt,minimum size=2pt] (ap6) at (3,5) {$\vdots$};
    
    \node[circle,fill=black,inner sep=2pt,minimum size=2pt,label=above:$\mathsf{B}^2_{p,n}$] (api) at (3,5.25) {};
    \node[circle,fill=black,inner sep=2pt,minimum size=2pt,label=above:$\mathsf{B}^1_{p,n}$] (bpi) at (4,5.25) {};
    \node[circle,fill=black,inner sep=2pt,minimum size=2pt,label=above:$\mathsf{A}_{p,n}$] (cpi) at (5,5.25) {};

    \draw[black!25!green, thick,->,rightsquigarrow] (c1) -- (c22);
    \draw[black!25!green, thick,->,rightsquigarrow] (b1) -- (b22);
    \draw[black!25!green, thick,->,rightsquigarrow] (a1) -- (a22);
    
    \draw[black!25!green, thick,->,rightsquigarrow] (c1) -- (cp2);
    \draw[black!25!green, thick,->,rightsquigarrow] (b1) -- (bp2);
    \draw[black!25!green, thick,->,rightsquigarrow] (a1) -- (ap2);

    \foreach \a/\b in {p2/p3, p3/p4, p4/p5, p5/p6} {
        \draw[black!25!green, thick,->,rightsquigarrow] (a\a) -- (a\b);
        \draw[black!25!green, thick,->,rightsquigarrow] (b\a) -- (b\b);
        \draw[black!25!green, thick,->,rightsquigarrow] (c\a) -- (c\b);
    }

    \foreach \a/\b in {22/23, 23/24, 24/25, 25/26} {
        \draw[black!25!green, thick,->,rightsquigarrow] (a\a) -- (a\b);
        \draw[black!25!green, thick,->,rightsquigarrow] (b\a) -- (b\b);
        \draw[black!25!green, thick,->,rightsquigarrow] (c\a) -- (c\b);
    }

    \draw[black!25!green, thick,->,rightsquigarrow] (b1) -- (a1);
    \draw[black!25!green, thick,->,rightsquigarrow] (b22) -- (a22);
    \draw[black!25!green, thick,->,rightsquigarrow] (b23) -- (a23);
    \draw[black!25!green, thick,->,rightsquigarrow] (b24) -- (a24);
    \draw[black!25!green, thick,->,rightsquigarrow] (b25) -- (a25);
    \draw[black!25!green, thick,->,rightsquigarrow] (b2i) -- (a2i);

    \draw[black!25!green, thick,->,rightsquigarrow] (bp2) -- (ap2);
    \draw[black!25!green, thick,->,rightsquigarrow] (bp3) -- (ap3);
    \draw[black!25!green, thick,->,rightsquigarrow] (bp4) -- (ap4);
    \draw[black!25!green, thick,->,rightsquigarrow] (bp5) -- (ap5);
    \draw[black!25!green, thick,->,rightsquigarrow] (bpi) -- (api);

    \draw[black!25!green, thick,->,rightsquigarrow] (c22) -- (b23);
    \draw[black!25!green, thick,->,rightsquigarrow] (c23) -- (b24);
    \draw[black!25!green, thick,->,rightsquigarrow] (c24) -- (b25);
    \draw[black!25!green, thick,->,rightsquigarrow] (c25) -- (b26);
    \draw[black!25!green, thick,->,rightsquigarrow] (c2i) -- (b2i);

    \draw[dotted,thick] (2.7,0.75) rectangle (5.3,5.75);
    \draw[dotted,thick] (2.55,0.65) rectangle (5.45,5.85);

    \begin{scope}[yshift = -70]
    \node[circle,fill=gray,inner sep=2.5pt,minimum size=2pt,draw=black,label=below:$\mathcal{B}^2_{0}$] (a1) at (0,0) {};
    \node[circle,fill=gray,inner sep=2.5pt,minimum size=2pt,draw=black,label=below:$\mathcal{B}^1_{0}$] (b1) at (1,0) {};
    \node[circle,fill=gray,inner sep=2.5pt,minimum size=2pt,draw=black,label=below:$\mathcal{A}_{0}$] (c1) at (2,0) {};

    \node[circle,fill=black,inner sep=2pt,minimum size=2pt] (c22) at (-1.5,1) {};
    \node[circle,fill=black,inner sep=2pt,minimum size=2pt,label=above:{$\mathcal{B}^1_{2}=\mathcal{A}_2$}] (b22) at (-1.5,1) {};
    \node[circle,fill=black,inner sep=2pt,minimum size=2pt,label=above:$\mathcal{B}^2_{2}$] (a22) at (-3,1) {};

    \node[circle,fill=black,inner sep=2pt,minimum size=2pt,label=above:$\mathcal{B}^2_{p}$] (ap2) at (3,1) {};
    \node[circle,fill=black,inner sep=2pt,minimum size=2pt,label=above:$\mathcal{B}^1_{p}$] (bp2) at (4,1) {};
    \node[circle,fill=black,inner sep=2pt,minimum size=2pt,label=above:$\mathcal{A}_{p}$] (cp2) at (5,1) {};

    \draw[black!25!green, thick,->,rightsquigarrow] (c1) -- (c22);
    \draw[black!25!green, thick,->,rightsquigarrow] (b1) -- (b22);
    \draw[black!25!green, thick,->,rightsquigarrow] (a1) -- (a22);
    
    \draw[black!25!green, thick,->,rightsquigarrow] (c1) -- (cp2);
    \draw[black!25!green, thick,->,rightsquigarrow] (b1) -- (bp2);
    \draw[black!25!green, thick,->,rightsquigarrow] (a1) -- (ap2);

    \draw[black!25!green, thick,->,rightsquigarrow] (a1) -- (b1);
    \draw[black!25!green, thick,->,rightsquigarrow] (a22) -- (b22);
    \draw[black!25!green, thick,->,rightsquigarrow] (ap2) -- (bp2);

    \draw[dotted,thick] (2.7,0.75) rectangle (5.3,1.75);
    \draw[dotted,thick] (2.55,0.65) rectangle (5.45,1.85);
    \end{scope}
\end{tikzpicture}
}
    \caption{The coefficient system prime spectra for $\underline{\mathsf{Sp}}_{C_2}$ (top) and $\underline{{A}}_{C_2}$ (bottom).}
    \label{fig:comparison_coefficient}
\end{figure}
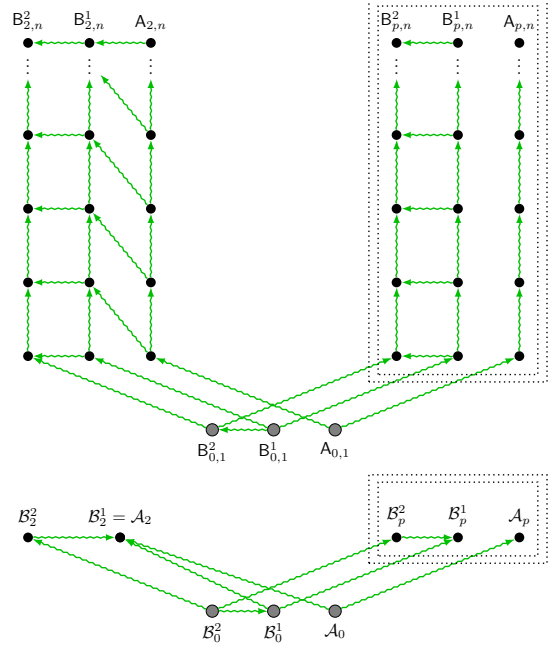

\begin{figure}[htp]
    \centering
    \begin{minipage}{.5\textwidth}
     \centering
    \scalebox{0.65}{
\begin{tikzpicture}[yscale=1.5, xscale=0.65]

    \tikzset{
  rightsquigarrow/.style={
    decorate,
    decoration={
    zigzag,
    segment length=4,
    amplitude=.4,post=lineto,
    post length=2pt
}
  }
    }
    \tikzset{>=latex}

    \node[circle,fill=gray,inner sep=2.5pt,minimum size=2pt,draw=black,label=below:$\mathsf{B}^2_{0,1}$] (a1) at (0,0) {};
    \node[circle,fill=gray,inner sep=2.5pt,minimum size=2pt,draw=black,label=below:$\mathsf{A}_{0,1}$] (c1) at (2,0) {};

    \node[circle,fill=black,inner sep=2pt,minimum size=2pt] (c22) at (-1,1) {};
    \node[circle,fill=black,inner sep=2pt,minimum size=2pt] (a22) at (-3,1) {};

    \node[circle,fill=black,inner sep=2pt,minimum size=2pt] (c23) at (-1,2) {};
    \node[circle,fill=black,inner sep=2pt,minimum size=2pt] (a23) at (-3,2) {};

    \node[circle,fill=black,inner sep=2pt,minimum size=2pt] (c24) at (-1,3) {};
    \node[circle,fill=black,inner sep=2pt,minimum size=2pt] (a24) at (-3,3) {};

    \node[circle,fill=black,inner sep=2pt,minimum size=2pt] (c25) at (-1,4) {};
    \node[circle,fill=black,inner sep=2pt,minimum size=2pt] (a25) at (-3,4) {};

    \node[circle,fill=none,inner sep=2pt,minimum size=2pt] (c26) at (-1,5) {$\vdots$};
    \node[circle,fill=none,inner sep=2pt,minimum size=2pt] (a26) at (-3,5) {$\vdots$};

    \node[circle,fill=black,inner sep=2pt,minimum size=2pt, label=above:$\mathsf{A}_{2,n}$] (c2i) at (-1,5.25) {};
    \node[circle,fill=black,inner sep=2pt,minimum size=2pt,label=above:$\mathsf{B}^2_{2,n}$] (a2i) at (-3,5.25) {};

    \node[circle,fill=black,inner sep=2pt,minimum size=2pt] (ap2) at (3,1) {};
    \node[circle,fill=black,inner sep=2pt,minimum size=2pt] (cp2) at (5,1) {};

    \node[circle,fill=black,inner sep=2pt,minimum size=2pt] (ap3) at (3,2) {};
    \node[circle,fill=black,inner sep=2pt,minimum size=2pt] (cp3) at (5,2) {};

    \node[circle,fill=black,inner sep=2pt,minimum size=2pt] (ap4) at (3,3) {};
    \node[circle,fill=black,inner sep=2pt,minimum size=2pt] (cp4) at (5,3) {};

    \node[circle,fill=black,inner sep=2pt,minimum size=2pt] (ap5) at (3,4) {};
    \node[circle,fill=black,inner sep=2pt,minimum size=2pt] (cp5) at (5,4) {};

    \node[circle,fill=none,inner sep=2pt,minimum size=2pt] (cp6) at (5,5) {$\vdots$};
    \node[circle,fill=none,inner sep=2pt,minimum size=2pt] (ap6) at (3,5) {$\vdots$};
    
    \node[circle,fill=black,inner sep=2pt,minimum size=2pt,label=above:$\mathsf{B}^2_{p,n}$] (api) at (3,5.25) {};
    \node[circle,fill=black,inner sep=2pt,minimum size=2pt,label=above:$\mathsf{A}_{p,n}$] (cpi) at (5,5.25) {};

    \draw[black!25!green, thick,->,rightsquigarrow] (c1) -- (c22);
    \draw[black!25!green, thick,->,rightsquigarrow] (a1) -- (a22);
    
    \draw[black!25!green, thick,->,rightsquigarrow] (c1) -- (cp2);
    \draw[black!25!green, thick,->,rightsquigarrow] (a1) -- (ap2);

    \foreach \a/\b in {p2/p3, p3/p4, p4/p5, p5/p6} {
        \draw[black!25!green, thick,->,rightsquigarrow] (a\a) -- (a\b);
        \draw[black!25!green, thick,->,rightsquigarrow] (c\a) -- (c\b);
    }

    \foreach \a/\b in {22/23, 23/24, 24/25, 25/26} {
        \draw[black!25!green, thick,->,rightsquigarrow] (a\a) -- (a\b);
        \draw[black!25!green, thick,->,rightsquigarrow] (c\a) -- (c\b);
    }

    \draw[black!25!green, thick,->,rightsquigarrow] (c22) -- (a23);
    \draw[black!25!green, thick,->,rightsquigarrow] (c23) -- (a24);
    \draw[black!25!green, thick,->,rightsquigarrow] (c24) -- (a25);
    \draw[black!25!green, thick,->,rightsquigarrow] (c25) -- (a26);
    \draw[black!25!green, thick,->,rightsquigarrow] (c2i) -- (a2i);

    \draw[dotted,thick] (2.5,0.75) rectangle (5.5,5.75);
    \draw[dotted,thick] (2.35,0.65) rectangle (5.65,5.85);

    \begin{scope}[yshift = -70]
    \node[circle,fill=gray,inner sep=2.5pt,minimum size=2pt,draw=black,label=below:$\mathcal{B}^2_{0}$] (a1w) at (0,0) {};
    \node[circle,fill=gray,inner sep=2.5pt,minimum size=2pt,draw=black,label=below:$\mathcal{A}_{0}$] (c1w) at (2,0) {};

    \node[circle,fill=black,inner sep=2pt,minimum size=2pt] (c22w) at (-1,1) {};
    \node[circle,fill=black,inner sep=2pt,minimum size=2pt,label=above:{$\mathcal{A}_2$}] (b22w) at (-1,1) {};
    \node[circle,fill=black,inner sep=2pt,minimum size=2pt,label=above:$\mathcal{B}^2_{2}$] (a22w) at (-3,1) {};

    \node[circle,fill=black,inner sep=2pt,minimum size=2pt,label=above:$\mathcal{B}^2_{p}$] (ap2w) at (3,1) {};
    \node[circle,fill=black,inner sep=2pt,minimum size=2pt,label=above:$\mathcal{A}_{p}$] (cp2w) at (5,1) {};

    \draw[black!25!green, thick,->,rightsquigarrow] (c1w) -- (c22w);
    \draw[black!25!green, thick,->,rightsquigarrow] (a1w) -- (a22w);
    
    \draw[black!25!green, thick,->,rightsquigarrow] (c1w) -- (cp2w);
    \draw[black!25!green, thick,->,rightsquigarrow] (a1w) -- (ap2w);

    \draw[black!25!green, thick,->,rightsquigarrow] (a22w) -- (b22w);

    \draw[dotted,thick] (2.5,0.75) rectangle (5.5,1.75);
    \draw[dotted,thick] (2.35,0.65) rectangle (5.65,1.85);
    \end{scope}
\end{tikzpicture}
}
    \captionof{figure}{The Green prime spectra for $\underline{\mathsf{Sp}}_{C_2}$ (top) and $\underline{{A}}_{C_2}$ (bottom).}
    \label{fig:comparison_green}
    \end{minipage}%
\begin{minipage}{.5\textwidth}
    \centering
    \scalebox{0.65}{
\begin{tikzpicture}[yscale=1.5, xscale=0.65]
    \tikzset{
  rightsquigarrow/.style={
    decorate,
    decoration={
    zigzag,
    segment length=4,
    amplitude=.4,post=lineto,
    post length=2pt
}
  }
    }
    \tikzset{>=latex}

    \node[circle,fill=gray,inner sep=2.5pt,minimum size=2pt,draw=black,label=below:$\mathsf{C}_{0,1}$] (a1) at (0,0) {};
    \node[circle,fill=gray,inner sep=2.5pt,minimum size=2pt,draw=black,label=below:$\mathsf{B}_{0,1}^2$] (c1) at (2,0) {};

    \node[circle,fill=black,inner sep=2pt,minimum size=2pt] (c22) at (-1,1) {};
    \node[circle,fill=black,inner sep=2pt,minimum size=2pt] (a22) at (-3,1) {};

    \node[circle,fill=black,inner sep=2pt,minimum size=2pt] (c23) at (-1,2) {};
    \node[circle,fill=black,inner sep=2pt,minimum size=2pt] (a23) at (-3,2) {};

    \node[circle,fill=black,inner sep=2pt,minimum size=2pt] (c24) at (-1,3) {};
    \node[circle,fill=black,inner sep=2pt,minimum size=2pt] (a24) at (-3,3) {};

    \node[circle,fill=black,inner sep=2pt,minimum size=2pt] (c25) at (-1,4) {};
    \node[circle,fill=black,inner sep=2pt,minimum size=2pt] (a25) at (-3,4) {};

    \node[circle,fill=none,inner sep=2pt,minimum size=2pt] (c26) at (-1,5) {$\vdots$};
    \node[circle,fill=none,inner sep=2pt,minimum size=2pt] (a26) at (-3,5) {$\vdots$};

    \node[circle,fill=black,inner sep=2pt,minimum size=2pt, label=above:$\mathsf{B}^2_{2,n}$] (c2i) at (-1,5.25) {};
    \node[circle,fill=black,inner sep=2pt,minimum size=2pt,label=above:$\mathsf{C}_{2,n}$] (a2i) at (-3,5.25) {};

    \node[circle,fill=black,inner sep=2pt,minimum size=2pt] (ap2) at (3,1) {};
    \node[circle,fill=black,inner sep=2pt,minimum size=2pt] (cp2) at (5,1) {};

    \node[circle,fill=black,inner sep=2pt,minimum size=2pt] (ap3) at (3,2) {};
    \node[circle,fill=black,inner sep=2pt,minimum size=2pt] (cp3) at (5,2) {};

    \node[circle,fill=black,inner sep=2pt,minimum size=2pt] (ap4) at (3,3) {};
    \node[circle,fill=black,inner sep=2pt,minimum size=2pt] (cp4) at (5,3) {};

    \node[circle,fill=black,inner sep=2pt,minimum size=2pt] (ap5) at (3,4) {};
    \node[circle,fill=black,inner sep=2pt,minimum size=2pt] (cp5) at (5,4) {};

    \node[circle,fill=none,inner sep=2pt,minimum size=2pt] (cp6) at (5,5) {$\vdots$};
    \node[circle,fill=none,inner sep=2pt,minimum size=2pt] (ap6) at (3,5) {$\vdots$};
    
    \node[circle,fill=black,inner sep=2pt,minimum size=2pt,label=above:$\mathsf{C}_{p,n}$] (api) at (3,5.25) {};
    \node[circle,fill=black,inner sep=2pt,minimum size=2pt,label=above:$\mathsf{B}^2_{p,n}$] (cpi) at (5,5.25) {};

    \draw[black!25!green, thick,->,rightsquigarrow] (c1) -- (c22);
    \draw[black!25!green, thick,->,rightsquigarrow] (a1) -- (a22);
    
    \draw[black!25!green, thick,->,rightsquigarrow] (c1) -- (cp2);
    \draw[black!25!green, thick,->,rightsquigarrow] (a1) -- (ap2);

    \foreach \a/\b in {p2/p3, p3/p4, p4/p5, p5/p6} {
        \draw[black!25!green, thick,->,rightsquigarrow] (a\a) -- (a\b);
        \draw[black!25!green, thick,->,rightsquigarrow] (c\a) -- (c\b);
    }

    \foreach \a/\b in {22/23, 23/24, 24/25, 25/26} {
        \draw[black!25!green, thick,->,rightsquigarrow] (a\a) -- (a\b);
        \draw[black!25!green, thick,->,rightsquigarrow] (c\a) -- (c\b);
    }

    \draw[black!25!green, thick,->,rightsquigarrow] (c22) -- (a22);
    \draw[black!25!green, thick,->,rightsquigarrow] (c23) -- (a23);
    \draw[black!25!green, thick,->,rightsquigarrow] (c24) -- (a24);
    \draw[black!25!green, thick,->,rightsquigarrow] (c25) -- (a25);
    \draw[black!25!green, thick,->,rightsquigarrow] (c2i) -- (a2i);

    \draw[black!25!green, thick,->,rightsquigarrow] (cp2) -- (ap2);
    \draw[black!25!green, thick,->,rightsquigarrow] (cp3) -- (ap3);
    \draw[black!25!green, thick,->,rightsquigarrow] (cp4) -- (ap4);
    \draw[black!25!green, thick,->,rightsquigarrow] (cp5) -- (ap5);
    \draw[black!25!green, thick,->,rightsquigarrow] (cpi) -- (api);

    \draw[black!25!green, thick,->,rightsquigarrow] (c1) -- (a1);

    \draw[dotted,thick] (2.5,0.75) rectangle (5.5,5.75);
    \draw[dotted,thick] (2.35,0.65) rectangle (5.65,5.85);

    \begin{scope}[yshift = -70]
    \node[circle,fill=gray,inner sep=2.5pt,minimum size=2pt,draw=black,label=below:$\mathcal{C}_{0}$] (a1w) at (0,0) {};
    \node[circle,fill=gray,inner sep=2.5pt,minimum size=2pt,draw=black,label=below:$\mathcal{B}^2_{0}$] (c1w) at (2,0) {};

    \node[circle,fill=black,inner sep=2pt,minimum size=2pt] (c22w) at (-2,1) {};
    \node[circle,fill=black,inner sep=2pt,minimum size=2pt,label=above:{$\mathcal{B}^2_2 = \mathcal{C}_2$}] (b22w) at (-2,1) {};
    \node[circle,fill=black,inner sep=2pt,minimum size=2pt] (a22w) at (-2,1) {};

    \node[circle,fill=black,inner sep=2pt,minimum size=2pt,label=above:$\mathcal{C}_{p}$] (ap2w) at (3,1) {};
    \node[circle,fill=black,inner sep=2pt,minimum size=2pt,label=above:$\mathcal{B}^2_{p}$] (cp2w) at (5,1) {};

    \draw[black!25!green, thick,->,rightsquigarrow] (c1w) -- (c22w);
    \draw[black!25!green, thick,->,rightsquigarrow] (a1w) -- (a22w);
    
    \draw[black!25!green, thick,->,rightsquigarrow] (c1w) -- (cp2w);
    \draw[black!25!green, thick,->,rightsquigarrow] (a1w) -- (ap2w);

    \draw[black!25!green, thick,->,rightsquigarrow] (a1w) -- (c1w);
    \draw[black!25!green, thick,->,rightsquigarrow] (ap2w) -- (cp2w);
     
    \draw[dotted,thick] (2.5,0.75) rectangle (5.5,1.75);
    \draw[dotted,thick] (2.35,0.65) rectangle (5.65,1.85);
    \end{scope}
\end{tikzpicture}
}
\captionof{figure}{The Tambara prime spectra for $\underline{\mathsf{Sp}}_{C_2}$ (top) and $\underline{{A}}_{C_2}$ (bottom).}
    \label{fig:comparison_tambara}
    \end{minipage}
\end{figure}

\newpage

\addcontentsline{toc}{part}{References}

\let\oldaddcontentsline\addcontentsline
\renewcommand{\addcontentsline}[3]{}
\printbibliography
\let\addcontentsline\oldaddcontentsline

\end{document}